\documentclass[leqno,11pt]{amsart}%
\usepackage{eurosym}
\usepackage{amsmath}
\usepackage{mathrsfs}
\usepackage{amsfonts}
\usepackage{graphicx}
\usepackage{color}
\usepackage{amsfonts}
\usepackage{amssymb}%
\usepackage{esint}
\usepackage{todonotes}
\usepackage[abbrev]{amsrefs}
\usepackage{mathtools}
\usepackage{tikz-cd}
\usepackage{float}

\allowdisplaybreaks

\usepackage{amsfonts}
\usepackage{amsxtra}

\numberwithin{equation}{section}
\newcommand{\N}{\mathbb N}

\newcommand{\eps}{\varepsilon}

\newcommand{\dive}{\mathrm{div}\,}

\def\real{\mathbb R}

\def\N{\mathbb N}

\def\R{\mathbb R}
\def\Rn{{\mathbb R}^{n}}
\def\rn{\Rn}

\def\proofof#1{\begin{proof}[Proof of #1]}

\def\d{\mathrm{d}}
\def\Om{\Omega}

\def\E{\mathcal{E}}
\def\dx{\d x}

\def\dz{\d z}

\def\={^{\wedge}}

\def\eqn#1$$#2$${\begin{equation}\label#1#2\end{equation}}

\def\mir1{\mathcal L_1}

\def\Sym{\real^{n\times n}_{\textup{sym}}}

\DeclareMathOperator\spt{spt}

\newtheorem{thm}{Theorem}[section]

\newtheorem{lem}[thm]{Lemma}
\newtheorem{ex}[thm]{Example}
\newtheorem{prop}[thm]{Proposition}

\theoremstyle{definition}

\newtheorem{rem}[thm]{Remark}

\newtheorem{lemalph}{Lemma}

\newcommand{\xint}[3]{{\setbox0=\hbox{$#1{#2#3}{\int}$}
   \vcenter{\hbox{$#2#3$}}\kern-.5\wd0}}
\newcommand{\mint}{\mathchoice
   {\xint\displaystyle\textstyle-}
   {\xint\textstyle\scriptstyle-}
   {\xint\scriptstyle\scriptscriptstyle-}
   {\xint\scriptscriptstyle\scriptscriptstyle-}
   \!\int}

\makeatletter \catcode`@=11
\newbox\tr@tto
\setbox\tr@tto=\hbox{{\count0=0\dimen0=-,9pt\dimen1=1,1pt\loop\ifnum
    \count0<11 \advance \count0 by1 \vrule width.51pt height\dimen1
    depth\dimen0\kern-0.17pt\advance\dimen0 by-0.05pt\advance\dimen1
    by0.1pt\repeat \loop\ifnum\count0<21\advance \count0 by1 \vrule
    width.6pt height\dimen1 depth\dimen0\kern-0.2pt \advance\dimen0
    by-0.1pt\advance\dimen1 by 0.05pt\repeat}}
\def\medint{\displaystyle\copy\tr@tto\kern-10.4pt\int}
\numberwithin{equation}{section}

\allowdisplaybreaks

\title{Fractional higher differentiability of solutions to strongly nonlinear Stokes systems
}

\author {Andrea Cianchi, Flavia Giannetti, Antonia Passarelli di Napoli \& Christoph Scheven}

\address{Andrea Cianchi, Dipartimento di Matematica e Informatica \lq\lq U. Dini"\\
Universit\`a di Firenze,
Viale Morgagni 67/a,
50134 Firenze,
Italy} \email{andrea.cianchi@unifi.it}

\address{Flavia Giannetti, Dipartimento di Matematica e Applicazioni "R.
Caccioppoli" \\ Universit\`{a} di Napoli ``Federico
II", via Cintia - 80126 Napoli, Italy}
\email{giannett@unina.it}

\address{Antonia Passarelli di Napoli, Dipartimento di Matematica e Applicazioni "R.
Caccioppoli" \\ Universit\`{a} di Napoli ``Federico
II", via Cintia - 80126 Napoli, Italy}
\email{antpassa@unina.it}

\address{Christoph Scheven, Faculty for Mathematics \\ University Duisburg-Essen, D-45117 Essen, Germany}
\email{christoph.scheven@uni-due.de}


\begin{document}

\begin{abstract}  { This work concerns   stationary Stokes type systems governed by a general class of non-necessarily power-type nonlinearities. Fractional regularity properties of the symmetric gradient of local solutions are established, depending on a balance between the nonlinearity of the differential operator and the degree of integrability of the datum on right-hand side. The non-polynomial character of the differential operators  calls for the use of Orlicz and Orlicz-Sobolev spaces as an appropriate functional framework for both the solutions and the datum.
The regularity result  amounts to the membership of
a nonlinear expression of the symmetric gradient
in Besov spaces. Fractional regularity of the pressure term is also exhibited and is formulated in terms of Orlicz-Besov spaces. Fractional Sobolev regularity of the symmetric gradient and of the pressure follow as a consequence.}
Parallel results for the symmetric gradient of local solutions to the associated plain  elliptic system are also offered. A new version of a  Poincar\'e-Sobolev inequality in Orlicz spaces, in modular form, on domains with finite measure plays a role in the proofs.
\end{abstract}
\maketitle

\section{Introduction}\label{sec:intro}
 
  Stokes type systems model the motion of  non-Newtonian fluids, namely
incompressible fluids 
affected  by a nonlinear relation between   the deformation velocity and the deviatoric tensor stress. The stationary version of these systems takes the form
\begin{equation}\label{equa}
    \left\{
  \begin{array}{l}
    -\dive a(x,\E u)+\nabla\pi= f\\[0.6ex]
    \dive u=0,
  \end{array}
  \right.
\end{equation}
  in an open  set  $\Omega \subset \R^n$, with $n \geq 2$. Here, $a:\Omega\times\Sym\to\Sym$ is a function satisfying     ellipticity and growth conditions in the variable  
in $\Sym$,  the space of symmetric $n\times n$ matrices, and $f: \Omega \to \R^n$ is  measurable. The unknowns are the functions 
$u: \Omega \to \R^n$ and $\pi : \Omega \to \R$, and $\E u$ denotes the symmetric gradient of $u$. In the physical problem, $u$ stands for the velocity of the fluid, $\E u$ represents its deformation velocity, $\pi$ denotes  its pressure,  $a$ is the deviatoric tensor stress, and $f$ is an external force.  We refer to \cites{AsMar,BAH} for a detailed description of the pertaining models. 

\par The present paper is concerned with fractional differentiability properties of $\E u$ and $\pi$ for local solutions $(u,\pi)$ to the system \eqref{equa}. 
The major novelties amount to a precise   balance assumption, in the general setting of Orlicz spaces, between the  nonlinearity of    $a(x, \E u)$ in $\E u$ and the degree of integrability of $f$, and to the lack of smoothness of $a(x, \E u)$ in $x$.

The ellipticity and growth of the function $a(x, P)$ in the symmetric gradient variable $P$ and the integrability of the datum $f$ are prescribed in terms of Young functions $\phi$ and $\psi$, respectively,   that are not necessarily powers. This allows for strongly nonlinear behaviors of the differential operator in \eqref{equa}, which need not be of polynomial type.
The Orlicz space $L^\psi_{\rm loc}(\Omega, \Rn)$ and the Orlicz-Sobolev space $W^{1,\phi}_{\rm loc}(\Omega, \Rn)$ thus provide an appropriate functional framework for the datum and the solutions.
A fractional differentiability assumption in the variable $x$, in the spirit of \cite{DevoreSharpley:1984} and \cite{Hajlasz}, is allowed on the function $a$. The regularity of $\E u$ is described via the membership in a  Besov space, of the form $B^{\sigma, 2, \infty}_{\rm loc}(\Omega, \mathbb R^{n\times n})$, of 
the  nonlinear expression  $V(\E u)= \sqrt{\frac{\phi'(|\E u|)}{|\E u|}}\,\E u$. A Besov space of the same kind is also well suited to account for the fractional regularity of $\pi$. 
 The use of the expression $V(\E u)$ is critical in discussing differentiability properties of $\E u$, because of the possible degeneracy of the differential operator in \eqref{equa}. Analogous expressions, defined with 
 $\E u$
 replaced with the full gradient $\nabla u$, are classically employed in
 the analysis of (fractional) higher order differentiability properties of solutions to elliptic and parabolic problems involving
 $\nabla u$ -- see e.g. the monograph \cite{Giusti}  and the more recent contributions   \cites{DDHSW, DieEtt08, GM,  GPdol, GPS0, MaPas, Ming, Pas}.

The results about the function $V(\E u)$ for the Stokes system \eqref{equa} have a parallel version for the  plain symmetric gradient system 
\begin{equation}\label{elliptic}
    -\dive a(x,\E u)= f.
\end{equation}
They just follow via a simplification in the proof for \eqref{equa}, due to a more straightforward choice of test functions, which do not need to obey the divergence free constraint in the weak formulation of \eqref{elliptic}.
 
 Loosely speaking, our conclusions assert that   $V(\E u)$ and, in the case of \eqref{equa}, also $\pi$, locally belong to  suitable Besov spaces, as soon as the datum $f$ lies in an Orlicz space $L^\psi _{\rm loc}(\Omega, \Rn)$ which is smaller, in a qualified sense, than that required for weak solutions to be well defined. The latter assumption is precisely formulated via an interpolation inequality involving the functions $\psi$, $\phi$, and its sharp Sobolev conjugate $\phi_n$. The smoothness parameter $\sigma$ in the relevant Besov spaces   depends on the interpolation parameter, on the upper Simonenko index of $\phi'$, and on the exponent in the fractional differentiability of $a$ with respect to $x$. Thanks to classical embedding theorems for Besov spaces, these pieces of information on $V(\E u)$ and $\pi$ imply that they also belong to fractional Sobolev spaces with suitable smoothness parameters.

{Some special instances of these results are available in the existing literature.}
 The case of the $p$-Laplacian Stokes system,  considered in  \cite{GPS2} for $p\geq 2$, is recovered with the choice $\phi(t)=t^p$. It is also complemented and augmented, in that  the range $1<p<2$ is included and a fractional differentiability dependence of the function $a$ in the $x$-variable is permitted. Systems governed by a Young function $\phi(t)$, possibly of power type $t^p$ for any $p\in (1,\infty)$, are the subject of \cite{GPS}. In that paper, however, only data $f\in L^{\widetilde \phi}_{\rm loc}(\Omega, \Rn)$ are admitted, where $\widetilde \phi$ denotes the Young conjugate of $\phi$. Furthermore, in \cite{GPS} 
the assumptions on the dependence on $x$ are more restrictive.  
 The result  of \cite{GPS} implies, in particular, that $V(\E u)\in W^{1,2}_{\rm loc}(\Omega, \mathbb R^{n\times n})$, provided that $\phi(t)=t^p$ with $1<p<2$ and $f\in L^{p'}_{\rm loc}(\Omega, \Rn)$. The same kind of regularity for local solutions to \eqref{elliptic} and for the same range of values of $p$ was established in \cite{Ser}.  In the recent paper \cite{Cavagnoli}, the local Sobolev regularity of $V(\E u)$ for local solutions to \eqref{equa}  has been shown to hold under the weaker assumption that  $f\in L^{\frac{np}{n(p-1)+2-p}}_{\rm loc}(\Omega, \Rn)$. Such an assumption is known to be optimal  even for an analogous
 conclusion in the case of the standard $p$-Laplace equation -- see \cites{ClopGentileAPdN}.
 Related results on the $W^{1,2}_{\rm loc}(\Omega, \mathbb R^{n\times n})$ regularity of $V(\E u)$ for local solutions to the Stokes system \eqref{equa}, subject to diverse growth conditions on the differential operator and on the summability of the right-hand side,  are contained in the papers \cites{Breit, BrFu, DiKa, FuSe, Na}. They do not address, however,  regularity issues about the function $\pi$.
Global $W^{1,2}(\Omega, \mathbb R^{n\times n})$ regularity of $V(\E u)$ is proved in \cite{BeDi} for every $p\in (1,\infty)$ and  for solutions to homogeneous Dirichlet problems for the system \eqref{elliptic}, under the assumption  $f\in W^{1,p'}(\Omega, \Rn)$. Partial results in this connection can also be found in \cites{BKR, BeRu1,  SerSh}.

Results concerning other quantities than $V(\E u)$ are available, although the regularity theory of systems depending of the symmetric gradient is still much less developed than its counterpart for the full gradient. For instance, $C^{1,\alpha}$--regularity of solutions to symmetric gradient $p$-Laplacian type systems and Stokes systems holds for $n=2$  \cites{DKS, KMS}, but it is still an open question in higher dimension.
For $n\geq 3$, only partial $C^{1,\alpha}$--regularity  of solutions to Stokes type systems has been established -- see the papers \cites{Breit, BrFu, Fuc, FGR}. In
\cite{AcMi}, the same kind of regularity is exhibited for solutions to more general Stokes systems driven by an $x$--dependent power type nonlinearity  in $\E u$. Maximal higher-integrability of the gradient of solutions to the Stokes system \eqref{equa} is the topic of \cite{DiKa}.

The paper is organized as follows. {The main results are stated in Section \ref{sec:main}, where
 they are illustrated by implementations to specific choices of the Young functions $\phi$ and $\psi$. This is preceded by two background sections, which can be skipped by readers who are familiar with the current theory of Orlicz and Orlicz-Sobolev spaces.
Specifically, notions and  specific properties concerning Young functions, to be used in our proofs, are the content of Section \ref{sec:young}. In the subsequent Section \ref{spaces}, the function spaces entering our analysis are introduced and various pertaining embeddings and inequalities are collected. We then continue with Section \ref{sec:SP},} which is devoted to an ad hoc modular Poincar\'e-Sobolev  inequality, in Orlicz-Sobolev spaces, tailored for our applications. In combination with  techniques that are peculiar
to the Orlicz space realm, this is a critical step in our approach.
 The proofs of the main results are accomplished in the remaining Sections \ref{sec:proof1}--\ref{sec:proof3}. A priori Besov space estimates  for $V(\E u)$, under additional regularity assumptions on the function $a$ and the datum $f$ are established in Section \ref{sec:proof1}. They are removed in Section \ref{sec:proof2} via an approximation process. This double-step argument is needed because of the  weakness of the assumptions imposed on the $x$--dependence of the function $a$. 
 The subject of Section \ref{sec:proof3} is the proof of the regularity of $\pi$. It is accomplished  by regarding the first line in  \eqref{equa} as  an equation where both $f$ and $u$ are assigned, and  $\pi$ is the sole unknown.

\section{Young functions}\label{sec:young}

A  Young
function $\phi$  is a left-continuous convex function  from  $[0, \infty )$ into  $[0, \infty ]$ that vanishes
at $0$ and is not constant in $(0, \infty)$. Any Young function $\phi$ can be represented as
\begin{equation}\label{A}
\phi(t) = \int _0^t \phi'(s)\, \d s \quad \quad \textup{for}\ t\in[0,\infty),
\end{equation}
 where $\phi '$ denotes the non-decreasing, left-continuous representative of the derivative of $\phi$. 
\\ If $\phi$ is a Young function, then the function
\begin{align}
    \label{incr}
    \frac{\phi(t)}{t} \quad \text{is non-decreasing.}
\end{align}
Moreover,
\begin{align}
    \label{{kt}}
    k\phi(t) \leq \phi(kt) \quad \text{for $k \geq 1$ and $t\geq 0$.}
\end{align}
The Young conjugate of a Young function~$\phi$ is defined as
\begin{align*}
  \widetilde \phi(t) = \sup\{st - \phi(s): s\geq 0\} 
  \quad \text{for $t \geq 0$.}
\end{align*}
One has that $\widetilde \phi$ is also a Young function and
\begin{equation}
  \label{Young-psi}
  st\le \phi(s)+\widetilde \phi(t)
  \qquad\mbox{for $s,t\ge0$.}
\end{equation}
Furthermore,
\begin{align}
    \label{inverses}
   t\leq \phi^{-1}(t)\widetilde \phi^{-1}(t) \leq 2t \quad \text{for $t \geq 0$,}
\end{align}
where the inverses are defined in the generalized right-continuous sense.
\\ The following inequalities hold:
\begin{equation}\label{complementary-N-functions}
  \widetilde\phi\Big(\frac{\phi(t)}{t}\Big)
  \le
   \phi(t)
  \le
  \widetilde \phi\Big(\frac{2\phi(t)}{t}\Big)
  \qquad
  \text{for $t>0$.}
\end{equation}
Moreover,
\begin{align}
    \label{july35}
    \widetilde\phi (\phi'(t))
  \le
   \phi(2t) \qquad
  \text{for $t>0$.}
\end{align}
  A Young function $\phi$ is said to dominate another Young function $\psi$ globally if there exists a constant $c>0$  such that
\begin{align}
    \label{dominate-young}
   \psi (t) \leq \phi(ct)
\end{align}
for $t \geq 0$. The function $\phi$ is said to dominate $\psi$ near infinity  
 if there exists a constant $t_0>0$ such that \eqref{dominate-young} holds for $t \geq t_0$.
 We say that the functions $\phi$ and $\psi$ are equivalent globally [near infinity] if they dominate each other globally [near infinity]. The equivalence of Young functions in this sense will be denoted by $\phi \approx \psi$.
\\
A  function $\phi : [0, \infty) \to [0,\infty)$ is said to satisfy the $\Delta_2$-condition -- briefly, $\phi \in \Delta_2$ -- if there exists a constant $c\geq 2$ such that
\begin{align}
    \label{delta2}
    \phi(2t)\le c\,\phi(t) \quad \text{for $t\geq 0$.}
\end{align}
The function $\phi$ is said to satisfy the $\nabla_2$-condition -- briefly, $\phi \in \nabla_2$ -- 
if there exists a constant $c>2$ such that
\begin{align}
    \label{nabla2}
    \phi(2t)\geq  c\,\phi(t) \quad \text{for $t\geq 0$.}
\end{align}
The optimal constants $c$
appearing in \eqref{delta2} and \eqref{nabla2} will be denoted by $\Delta_2(\phi)$ and $\nabla_2(\phi)$.
\\ One has that $\phi \in \Delta_2$ if and only if $\widetilde \phi \in \nabla_2$. Also, if $\phi \in \Delta_2$, then $\phi' \in \Delta_2$.
\\ Assume that the Young function $\phi$ is strictly positive in $(0, \infty)$ and $\phi \in C^2(0, \infty)$. We set 
\begin{align}
    \label{indices}
    i_\phi= \inf_{t>0} \frac{t\phi''(t)}{\phi'(t)} +1 \quad \text{and} \quad s_\phi= \sup_{t>0} \frac{t\phi''(t)}{\phi'(t)} +1.
\end{align}
Namely, $i_\phi$ and $s_\phi$ are the so called Simonenko indices of $\phi'$, shifted by $1$.
One has that $\phi \in \Delta_2$ if and only if $s_\phi <\infty$, and $\widetilde \phi \in \Delta_2$ if and only if $i_\phi>1$. In what follows, when stating any property involving the indices  $i_\phi$ and $s_\phi$ we shall tacitly assume that they satisfy the inequalities:
\begin{align}
    \label{indices3}
    1<i_\phi \leq s_\phi <\infty.
\end{align}
One can verify that
\begin{align}
    \label{indices1}
    i_\phi\leq    \frac{t\phi'(t)}{\phi(t)}  \leq  s_\phi \quad \text{for $t >0$.}
\end{align}
 Also, 
\begin{align}
    \label{indices2}
    s_\phi '\leq    \frac{t{\widetilde \phi}'(t)}{\widetilde \phi(t)}  \leq  i_\phi ' \quad \text{for $t >0$,}
\end{align}
where $i_\phi'$ and $s_\phi'$ stand for the H\"older conjugates of $i_\phi$
and $s_\phi$.
\\ From \eqref{indices1}, one has that
 \begin{align}
   \min\{ \lambda^{i_\phi}, \lambda^{s_\phi}\} \phi(t)
    &\le \phi(\lambda t)
      \le\max\{ \lambda^{i_\phi}, \lambda^{s_\phi}\}\phi(t) \quad \text{for $t, \lambda \geq 0$}
      \label{homogeneity-phi}
      \end{align}
      and, from \eqref{indices2},
      \begin{align}
      \min\{ \lambda^{i_\phi'}, \lambda^{s_\phi'}\}\widetilde \phi(t)
    &\le \widetilde \phi(\lambda t)
      \le \max\{ \lambda^{i_\phi'}, \lambda^{s_\phi'}\}\widetilde \phi(t) \quad \text{for $t, \lambda \geq 0$.}
      \label{homogeneity-phi-star}
  \end{align}
  In particular, Equation \eqref{homogeneity-phi} implies that
  \begin{align}
      \label{july57}
      \phi'(0)=0 \quad \text{and} \quad \lim_{t\to \infty}\phi'(t)=\infty.
  \end{align}
  Coupling \eqref{july35} with \eqref{homogeneity-phi} yields
  \begin{align}
    \label{july59}
    \widetilde\phi (\phi'(t))
  \le
   c(s_\phi)\phi(t) \qquad
  \text{for $t>0$.}
\end{align}
  Furthermore,
\begin{equation}\label{july36}
    \phi(s+t)\le c(s_\phi)\big(\phi(s)+\phi(t)\big) \quad \text{for $s,t \geq 0$.}
\end{equation}
Given a Young function $\phi$ and $a\geq 0$, the shifted Young function $\phi_a$ is defined, according to \cite{DieForTomWan19} (see also \cite{DieEtt08} for an earlier version), as 
\begin{align}
  \phi_a(t) = \int_0^t \frac{\phi'(\max \{ s,  {a}  \})}{\max \{ s, {a} \}} s \, \d s,
\end{align}
 for $t \in
\R$. The parameter $a$ is called
the shift.  Plainly, $\phi_0= \phi$. Moreover, if
$a \eqsim b$, then $\phi_a(t) \eqsim \phi_b(t)$, with equivalence constants depending on  $i_\phi$, $s_\phi$, and the equivalence constants in the relation $a \eqsim b$.
 Here, and in what follows, the relation $\eqsim$ between two expressions means that they are bounded by each other up to positive constants depending on appropriate quantities. This relation should not be confused with the relation $\approx$
between Young functions introduced above.
We refer to the papers \cites{DieEtt08, DieForTomWan19} for proofs of the properties of shifted Young functions recalled in this section.
\\ If $\phi \in C^2(0, \infty)$, then   $\phi_a\in C^2((0, \infty) \setminus \{a\})$, and 
\begin{equation*}
  \frac{t\phi_a''(t)}{\phi_a'(t)}
  =
  \left\{
  \begin{array}{cl}
    1&\mbox{ if $0\le t<a$}\\[0.8ex]
    \displaystyle \frac{t\phi''(t)}{\phi'(t)}&\mbox{ if $t>a$.}
  \end{array}
  \right.
\end{equation*}
Assume that $\phi$ fulfills \eqref{indices3}. Then,
\begin{align}
    \min\{\lambda^{i_\phi},\lambda^2\} \phi_a(t)
    &\le \phi_a(\lambda t)
      \le \max\{\lambda^{s_\phi},\lambda^2\}\phi_a(t) \quad \text{for $\lambda, t \geq 0$.}
      \label{homogeneity-phi-shifted}
      \end{align}
      Hence, by \eqref{indices2},
      \begin{align} 
   \min\{\lambda^{s_\phi'},\lambda^2\} \widetilde{\phi_a}(t)
    &\le \widetilde{\phi_a}(\lambda t)
      \le\max\{\lambda^{i_\phi'},\lambda^2\}\widetilde{\phi_a}(t) \quad \text{for $\lambda, t \geq 0$.}
      \label{homogeneity-phi-star-shifted}
  \end{align}
\\ Equations \eqref{Young-psi},
   \eqref{homogeneity-phi-shifted},
and~\eqref{homogeneity-phi-star-shifted} ensure  that, for
every $\delta>0$ there exists~$c_\delta=c_\delta(\delta,i_\phi,s_\phi)\geq 1$ such
that 
\begin{align}
  \label{eq:young}
  \begin{aligned}
    s\,t &\leq \delta\,\phi_{a}(t) + c_\delta\, \widetilde{\phi_a}(s) \qquad \text{for $s,t,a\geq 0$.}
  \end{aligned}
\end{align}
One has that
\begin{equation}\label{bound-phi-star}
    \widetilde {\phi_a}(\phi_a'(t))
    \le
    c(s_\phi)\phi_a(t) \quad \text{for $t \geq 0$.}
\end{equation}
Moreover,
\begin{equation}\label{sub-additivity-phi}
    \phi_a(s+t)\le c(s_\phi)\big(\phi_a(s)+\phi_a(t)\big) \quad \text{for $s,t \geq 0$.}
\end{equation}
There exists a constant $c=c(i_\phi, s_\phi)$ such that
  \begin{align}
  \label{eq:removal-shift2}
    \phi_{a}(t) &\leq \delta\, \phi(a) + c\, \delta\,
    \phi (t/\delta),
  \end{align}
for every $a\geq 0$, $t \geq 0$, and $\delta \in (0,1]$.
\\
The following relation holds for every 
 $P,Q\in\R^{n\times
    n}$ such that $|P|+|Q|>0$:
      \begin{equation}\label{july21}
      \phi_{|Q|}(|P-Q|)
      \eqsim \frac{\phi'(\max\{|P|,|Q|\})}{\max\{|P|,|Q|\}}|P-Q|^2,
  \end{equation}
  with equivalence constants depending only on $i_\phi$ and $s_\phi$.
  \\ Also,
 \begin{align}\label{shift4}
\phi''(|P|+|Q|)|P-Q|&\eqsim \phi'_{|P|}(|P-Q|) 
\end{align}
\begin{align}\label{shift4'}\phi''(|P|+|Q|)|P-Q|^2&\eqsim \phi_{|P|}(|P-Q|) 
\end{align}
   with equivalence constants depending only on $i_\phi$ and $s_\phi$.
\\
Define the function $V: \mathbb R^{n\times n} \to \mathbb R^{n\times n}$
as
\begin{align}
    \label{V}
    V(P)= \sqrt{\frac{\phi'(|P|)}{|P|}}\,P \quad \text{for $P \in \mathbb R^{n\times n}$.}
\end{align}
 Then, 
for every $\delta>0$ there exists~$c_\delta=c_\delta(\delta,i_\phi,s_\phi)$ such that
 \begin{equation}
  \label{add-shift}
  \phi(t)
  \le
  c_\delta\phi_{|P|}(t)+\delta|V(P)|^2
  \le
  c_\delta\phi_{|P|}(t)+\delta {s_\phi}\,\phi(|P|)
\end{equation}
for every $P\in\R^{n\times n}$, where $c_\delta=c_\delta(\delta,i_\phi,s_\phi)$.
\\ One has that
 \begin{align}\label{a-coercivity-split}
    \phi_{|{Q}|} \left( |{P-Q}| \right)
    \eqsim |{V(P)-V(Q)}|^2
    \end{align}
    for every~$P,Q \in \R^{n\times n}$, and 
    \begin{align}
    \label{seconda-split}
     |{V(Q)}|^2
    \eqsim
    \phi_{|{Q}|}(|{Q}|) \eqsim \phi(|{Q}|)
  \end{align}
  for every~$P\in \R^{n\times n}$.
\\ 
Let $n \in \mathbb N$, with $n\geq 2$, and let $\phi$ be a Young function.
The Sobolev conjugate  of $\phi$ is the Young function $\phi_n$ given by
\begin{align}
    \label{phin}
    \phi_n(t) = \phi (H_n^{-1}(t)) \quad \text{for $t \geq 0$,}
\end{align}
where the function
$H_n: [0, \infty) \to [0, \infty)$
is defined as
\begin{equation}\label{Hn}
H_n(s) =\bigg(\int_0^s\Big(\frac\tau{\phi(\tau)}\Big)^{\frac1{n-1}}\d\tau\bigg)^{\frac{n-1}{n}} \quad \text{for $s \geq 0$.}
\end{equation} 
Here, $H_n^{-1}$ denotes the left-continuous inverse of $H_n$, and $\phi$ is extended to $[0, \infty]$ by setting $\phi(\infty)=\infty$. 
\\ The definition of $H_n$ requires that
\begin{align}
    \label{conv0}
\int_0\Big(\frac{t}{\phi(t)}\Big)^{\frac{1}{n-1}}\, \d t < \infty.
\end{align}
 This condition is not restrictive in view of our purposes. 
In what follows, the function $\phi_n$ is tacitly assumed to be defined with $\phi$ modified near zero, if necessary, in such a way that \eqref{conv0} is satisfied. The choice of this possible modification of $\phi$ is immaterial in our main results. Indeed, they  depend on  $\phi_n$ only up to equivalence near infinity, and the latter depends  on $\phi$ only up to equivalence near infinity. 
\\ Note that, if
\begin{equation}\label{convinf}
    \int^\infty\Big(\frac{t}{\phi(t)}\Big)^{\frac{1}{n-1}}\d t<\infty,
\end{equation}
then $H^{-1}(t)=\infty$ for $t\ge t_\infty$, where we have set $t_\infty =\lim_{s\to\infty} 
H(s)$. Consequently,  $\phi_n(t)=\infty$ for $t\ge t_\infty$ as well.
\\ Also, there exist constants $c, t_0 >0$ such that
\begin{align}
    \label{july60}
    \phi(t) \leq \phi_n(ct) \quad \text{for $t\geq t_0$.}
\end{align}

\section{Function spaces and functional inequalities}\label{spaces}

Let $\Omega \subset \rn$ be an open set in $\rn$ and  let $v: \Omega \to \mathbb R^N$. We define 
 $\tau_hv : \Omega \to \mathbb R^N$ by 
  \begin{equation}
    \begin{gathered}
      \tau_h v(x)= v(x+he_i)-v(x),
    \end{gathered}
\label{def:diff-quotients}
\end{equation}
for 
$x\in\Omega$ and
$h\neq0$ such that $x+he_i\in\Omega$. Here, 
$e_i$,   denotes the $i$-th coordinate unit vector in $\mathbb R^n$, for $i\in\{1,\dots,n\}$.
  \\ Assume that the functions $v, w : \Omega \to \mathbb R^N$ are weakly differentiable. Then,
$$ (\tau_{h}v)_{x_i}=\tau_{h}(v_{x_i}).$$
  If the support of at least one of the functions $v$ and
  $w$ is compactly contained in $\Omega$, then 
\begin{align}
    \label{july19}
    \int_{\Omega} v\, \tau_{h} w\, \dx =\int_{\Omega} w\, \tau_{-h}v\, \dx
\end{align}
for sufficiently small $|h|$.
\\ Moreover,
 the  product rule for  finite differences takes the form:
\begin{equation}\label{product-rule}
\tau_{h}(v w)(x)=v(x+he_i)\tau_{h}w(x)+w(x)\tau_{h}v(x),
\end{equation}
 for $x\in \Omega$ and $h$ such that $x+he_i \in \Omega$.
\\
Let $\phi$ be a Young function and let $N\in \mathbb N$. The  Orlicz space $L^\phi(\Omega, \mathbb R^N)$ is defined as 
\begin{align}
    \label{orlicz}
    L^\phi (\Omega, \mathbb R^N)= \bigg\{v: \Omega \to \mathbb R^N: \exists \lambda >0\,\,\text{s.t.} \,\, \int_\Omega\phi\bigg(\frac{|v|}{\lambda}\bigg)\, \d x <\infty\bigg\}.
\end{align}
{One has that $L^\phi(\Omega, \mathbb R^N)$  is a Banach space, equipped with the Luxemburg norm given by
\begin{align}
    \label{orlicznorm}
    \|v\|_{L^\phi(\Omega, \mathbb R^N)}= \inf\bigg\{\lambda >0: \int_\Omega\phi\bigg(\frac{|v|}{\lambda}\bigg)\, \d x \leq 1\bigg\}.
\end{align}}
We also introduce  the notation $$L^\phi_\perp (\Omega, \mathbb R^N)= \{ v\in L^\phi (\Omega, \mathbb R^N): v_\Omega =0\},$$ 
where $v_\Omega = \frac 1{|\Omega|}\int_\Omega v\d x$, the mean value of $v$ over $\Omega$.   When $N=1$, the target space $\mathbb R^N$ in the notation of the Orlicz spaces and the other spaces introduced below will be omitted. Spaces of matrix-valued functions will also be considered and are defined analogously.
\\ If $v\in L^{\phi}_{\rm loc}(\Omega, \mathbb R^N)$, then
\begin{equation}\label{est-translation}
\int_{B_{R}} \phi(|v(x+he_i)|)\ \dx\leq  \int_{B_{R+|h|}}\phi(|v(x)|)\ \dx 
\end{equation}
for every $R>0$ and $h \neq 0$ such that $B_{R+|h|} \Subset \Omega$. Here, $B_R$ denotes a ball of radius $R$.
\\ The Orlicz-Sobolev space $W^{1,\phi}(\Omega, \mathbb R^N)$ is  given by
\begin{align}
    \label{orliczsobolev}
    W^{1,\phi} (\Omega, \mathbb R^N)= \big\{v \in L^\phi (\Omega, \mathbb R^N): v \,\,\text{is weakly differentiable and}\,\, \nabla v \in L^\phi(\Omega, \mathbb R^{N\times n})\big\}.
\end{align}
The local versions $L^\phi _{\rm loc}(\Omega, \mathbb R^N)$
and $W^{1,\phi}_{\rm loc} (\Omega, \mathbb R^N)$ of these spaces are defined as usual.   Observe that $L^\phi _{\rm loc}(\Omega, \mathbb R^N)\subset L^\psi _{\rm loc}(\Omega, \mathbb R^N)$ for some Young functions $\phi$ and $\psi$ if and only if $\phi$ dominates $\psi$ near infinity. A parallel characterization holds for inclusion relation between local Orlicz-Sobolev spaces. 
\\
The space $W^{1,\phi}_0 (\Omega, \mathbb R^N)$ of functions in $W^{1,\phi} (\Omega, \mathbb R^N)$ vanishing on $\partial \Omega$ can be defined as
\begin{align}
    \label{orliczsobolev0}
    W^{1,\phi}_0 (\Omega, \mathbb R^N)= \big\{v \in W^{1,\phi} (\Omega, \mathbb R^N):  \,\,\text{the extension of $v$ by $0$ to $\mathbb R^n\setminus \Omega$ is weakly differentiable in $\mathbb R^n $}\big\}.
\end{align}
In what follows, functions in the space $W^{1,\phi}_0 (\Omega, \mathbb R^N)$ will be assumed to be defined as $0$ in $\rn \setminus \Omega$ without further mentioning.
\\
Let $\sigma \in (0,1]$. 
 We define the local fractional Orlicz-Sobolev space $W^{\sigma, \phi}_{\rm}(\Omega, \mathbb R^N)$ as 
\begin{equation}\label{fractional-semi-norm}
  W^{\sigma, \phi}_{\rm loc}(\Omega, \mathbb R^N) =\Bigg\{v\in L^\phi_{\rm loc}(\Omega, \mathbb R^N): \forall \, \text{open set}\, \Omega' \Subset \Omega \, \exists \,\lambda>0 \, \text{s.t.}\,\int_{\Omega'}\int_{\Omega'}
    \phi\left(\frac{|v(x)-v(y)|}{\lambda|x-y|^{\sigma}}\right)\frac{\dx\d y}{|x-y|^{n}}<\infty\Bigg\},
\end{equation}
 and the local Besov space  $B^{\sigma, \phi, \infty}_{\rm loc}(\Omega, \mathbb R^N)$ 
as 
\begin{align}\label{besov}
B^{\sigma, \phi, \infty}_{\rm loc}(\Omega, \mathbb R^N)= \bigg\{v\in L^\phi_{\rm loc}(\Omega, \mathbb R^N): \forall\,B_R\Subset  \Omega\,\exists\, \lambda, \delta>0 \, \text{s.t.}\sup_{|h|<\delta }    \int_{B_R}\phi\bigg(\bigg|\frac{\tau_hv}{\lambda |h|^\sigma}\bigg|\bigg)\, \d x < \infty\bigg\}.
\end{align}
{Moreover,  if $B_R$ is any ball such that $B_{2R}\Subset \Omega$, we set, in analogy with \eqref{orlicznorm},
\begin{align}
    \label{besovnorm}
    \|v\|_{B^{\sigma, \phi, \infty}(B_R, \,\mathbb R^N)}= \inf\bigg\{\lambda >0: \sup_{|h|<R }    \int_{B_R}\phi\bigg(\bigg|\frac{\tau_hv}{\lambda |h|^\sigma}\bigg|\bigg)\, \d x \leq 1\bigg\}.
\end{align}}
One has that, 
    {\begin{equation}
        \label{inlc1}
        B^{\sigma, \phi, \infty}_{\rm loc}(\Omega, \mathbb R^N) \subset W^{\beta, \phi}_{\rm loc}(\Omega, \mathbb R^N) \quad \text{if $0<\beta < \sigma <1$.}
    \end{equation}
    The inclusion \eqref{inlc1} follows from \cite[Lemma 2.13]{GPS}.
Moreover, \cite[Lemma 2.10]{GPS} implies that
\begin{align}
    \label{july10}
    W^{1,\phi}_{\rm loc}(\Omega, \mathbb R^N) \subset  B^{1,\phi, \infty}_{\rm loc}(\Omega, \mathbb R^N),
\end{align}
and 
\begin{equation}\label{est-diff-quotient}
    \int_{B_{R}} \phi\left(\frac{|\tau_{h} v|}{|h|}\right)\dx
    \leq
    \int_{B_{R+|h|}}\phi(|\nabla v|)\, \dx
  \end{equation}
 for $v\in W^{1,\phi}_{\rm loc}(\Omega, \mathbb R^N)$, and
 every $R>0$ and $h \neq 0$ such that $B_{R+|h|} \Subset \Omega$. 
\\
Notice that, although the lemmas from \cite{GPS} quoted above are stated for Young functions $\phi$ satisfying additional assumptions, the same arguments in their proofs yield a version of the results for arbitrary Young functions $\phi$.
\\
}
A Sobolev type inequality in Orlicz spaces tells us that, if $\phi$ is a Young function and $|\Omega|<\infty$, then
\begin{align}
    \label{OS-emb}
    W^{1,\phi}_0(\Omega, \mathbb R^N) \to L^{\phi_n}(\Omega, \mathbb R^N),
\end{align}
where $\phi_n$ is defined as in \eqref{phin},
 see \cite{cianchi_CPDE} (and \cite{cianchi_IUMJ} for an equivalent formulation). Here, and in what follows, the arrow $\to$ denotes a continuous embedding.
 If, in addition, $\Omega$ is regular enough -- a bounded Lipschitz domain, for instance -- then 
    \begin{align}
    \label{OS-emb'}
    W^{1,\phi}(\Omega, \mathbb R^N) \to L^{\phi_n}(\Omega, \mathbb R^N).
\end{align}
Moreover, the target $L^{\phi_n}(\Omega, \mathbb R^N)$ is optimal (i.e. smallest possible) in the embeddings \eqref{OS-emb} and \eqref{OS-emb'} among all Orlicz spaces. As mentioned above, in these embeddings, the function $\phi_n$ is defined with $\phi$ modified near zero, if necessary, in such a way that the condition \eqref{conv0} is satisfied. 
\\  In the special case when $\phi(t)=t^p$ for some $p \geq 1$, the spaces defined above recover the classical Lebesgue, Sobolev, Fractional Sobolev and Besov spaces, and are denoted simply by $L^p (\Omega, \mathbb R^N)$, $W^{1,p} (\Omega, \mathbb R^N)$, $W^{\sigma,p}_{\rm loc} (\Omega, \mathbb R^N)$, and $B^{\sigma, p, \infty}_{\rm loc}(\Omega, \mathbb R^N)$.
\\ 
{ A Poincar\'e-Sobolev inequality for Besov spaces tells us that, if $0<\beta < \sigma <1$, then
$B^{\sigma,2,\infty}_{\rm loc}(\Omega, \R^N) \subset L^{\frac{2n}{n-2\beta}}_{\rm loc}(\Omega, \R^N)$. Moreover, 
if $B_R\Subset\Omega$ and $r<R$, then
\begin{align}\label{besov-emb}
\left(\int_{B_{r}}|v|^{\frac{ {2n}}{ n-2\beta}}\,\dx  \right)^{\frac{n-2\beta}{2n}}
\leq  c(R-r)^{\sigma-\beta} \bigg(\sup_{0<h< \frac{R-r}{2}}\int_{B_{\frac{R+r}{2}}}\frac{|\tau_hv|^2}{|h|^{2\sigma}}\,\dx\bigg)^\frac 12+ \frac{c}{(R-r)^{\beta}} \bigg(\int_{B_{R}}|v|^2\,\dx\bigg)^\frac 12
\end{align}
for some constant $c=c(n,\beta,     \sigma)$ and for every $u \in B^{\sigma,2,\infty}_{\rm loc}(\Omega, \R^N)$, see \cite[Lemma 2.8]{GPS2}.}
\\
The symmetric gradient of a function $v \in L^1_{\rm loc}(\Omega, \mathbb R^n)$ is defined as
\begin{equation}
  \E v =\tfrac12\big(\nabla v+(\nabla v)^T\big),
\end{equation}
in the distributional sense. Define the symmetric gradient Orlicz-Sobolev space as
\begin{align}
    \label{sym-orliczsobolev}
    E^{1,\phi} (\Omega, \mathbb R^n)= \big\{v \in L^\phi (\Omega, \mathbb R^n):  \,\, \mathcal E v \in L^\phi(\Omega, \mathbb R^{n\times n})\big\}.
\end{align}
Its subspace $ E^{1,\phi}_0 (\Omega, \mathbb R^n)$ and the space  $ E^{1,\phi}_{\rm loc} (\Omega, \mathbb R^n)$ can be defined in the same spirit as $W^{1,\phi}_0(\Omega, \mathbb R^n)$ and $W^{1,\phi}_{\rm loc}(\Omega, \mathbb R^n)$.
 \\ A Korn type inequality on balls in Orlicz spaces asserts that, if $\phi \in \Delta_2 \cap \nabla _2$, then  
$$E^{1,\phi}_0 (B_R, \mathbb R^n) \to W^{1,\phi}_0(B_R, \mathbb R^n),$$
and there exists a constant $c$, depending on $n$, $\Delta_2(\phi)$,
  and $\nabla_2(\phi)$  such that
 \begin{equation}
    \label{korn-0}\int_{B_R}\phi(|\nabla u|)\,\dx
    \le
    c\int_{B_R}\phi(|\E u|)\dx
  \end{equation}
  for every $u  \in E^{1,\phi}_0 (B_R, \mathbb R^n)$, see \cite[Thms. 6.10]{DRS:2010}. A variant of the Korn inequality in Orlicz spaces for functions $\phi$ which do not satisfy the $\Delta_2$ and $\nabla_2$ conditions can be found in \cite{cianchi_JFA}.
\\ 
An Orlicz space version of the
Bogovski\u\i\ 
 Lemma on balls $B_R$ is a consequence of \cite[Corollary 4.2]{Krepela-Ruzicka:2020}, and tells us what follows. Assume that the Young function $\phi \in \Delta_2 \cap \nabla_2$. Then there exists 
 a linear  operator
  \begin{equation}\label{july15}
\mathcal{B}:L^\phi_\perp(B_R, \mathbb R^N)\to W^{1,\phi}_0(B_R,\R^N)
  \end{equation}
  such that
\begin{equation}\label{july16}
   {\rm div}(\mathcal{B}v)=v\quad\mbox{in $B_R$}
 \end{equation}
 for every $v\in
 L^\phi_\perp(B_R, \mathbb R^N)$. 
 Moreover, there exists a constant $c$ depending on $n$,
  $\Delta_2(\phi)$, and $\nabla_2(\phi)$ such that
 \begin{equation}\label{Bogovski-gradient}
\int_{B_R}\phi(|\nabla (\mathcal{B}v)|)\, \dx\le c\int_{B_R}\phi(|v|)\, \dx
 \end{equation}
 and
 \begin{equation}\label{Bogovski-diff-quotients}
    \int_{B_R}\phi\left(\frac{|\tau_h\nabla(\mathcal{B}v)|}{|h|}\right) \dx
    \le
    c\int_{B_R}\phi\left(\frac{|\tau_hv|}{|h|}+\frac{|\tau_{-h}v|}{|h|}
      +\frac{|v|}{R}\right)\dx
  \end{equation}
  for every $v\in L^\phi_\perp(B_R, \mathbb R^N)$ and every $h\neq0$. 
In the latter inequality,   the functions  $v$ and $\mathcal{B}v$ are continued by
  zero outside of $B_R$. 
  Hence, in particular, if 
  $v\in W^{1,\phi}_0(B_R, \mathbb R^N)$, then 
\begin{align}\label{Bogovski-diff-quotients-2}
    \int_{B_{R+|h|}}\phi\left(\frac{|\tau_h\nabla (\mathcal{B}v)|}{|h|}\right) \dx
    &\le
    c\int_{B_R}\left[\phi(|\nabla v|) +\phi\left(\frac{|v|}{R}\right)\right] \dx
  \end{align}
  for some constant $c=c(n,\Delta_2(\phi),\nabla_2(\phi))$ and 
  for every $h\neq0$.
  This follows from an estimate for the right-hand side of the inequality \eqref{Bogovski-diff-quotients} by \eqref{est-diff-quotient}.

\section{Main results}\label{sec:main}

Our hypotheses on the  function $a$   in \eqref{equa} and \eqref{elliptic} read as follows.
 We assume 
that $a:\Omega\times\Sym\to\Sym$ is a Carath\'eodory function such that, for a.e. $x,y\in\Omega$, 
\begin{align}
   \label{ip1}
 (   a(x,P)-a(x,Q))\cdot (P-Q)
  &\ge \nu\phi''(|P|+|Q|)|P-Q|^2,\\
  \label{ip2}
  |a(x,P)-a(x,Q)|&\le L\phi''(|P|+|Q|)|P-Q|,\\
    \label{ip2b}
   |a(x,P)|&\le L\phi'(|P|),\\
  \label{ipholder}
  |a(x,P)-a(y,P)|
  &\le
  (k(x)+k(y))|x-y|^\alpha \phi'(|P|),
 \end{align}
  for every $P,Q\in\Sym$. 
Here:
\begin{equation}\label{hpphi}
    \phi \in C^2([0, \infty))\,\, 
   \text{is a Young function satisfying \,\,(\ref{indices3});}
\end{equation} 
\begin{align}
    \label{hppar}
    \text{$0<\nu\le L$ \,\, and  \,\, $0<\alpha\le 1$;}
\end{align}
\begin{align}\label{hpk}
    \text{$k\in L^m_{\rm loc}(\Omega)$\,\,
for some \,\, $m\in (\tfrac n\alpha, \infty]$,}
\end{align} 
and the dot $\lq\lq \cdot "$ stands for scalar product between vectors or matrices.
 
The assumptions \eqref{ip1}--\eqref{ip2b} are the natural generalizations to the Orlicz ambit of standard conditions for nonlinear  elliptic operators with a power type behavior. 
In particular, Equations \eqref{ip2b} and \eqref{july57} imply that $a(x,0)=0$. Hence, the monotonicity property \eqref{ip1} implies the ellipticity condition
\begin{align}
    \label{july58}
    a(x,P) \cdot P
  &\ge \nu\phi''(|P|)|P|^2
\end{align}
for a.e. $x\in \Omega$ and  every $P\in\Sym$.

The assumption \eqref{ipholder}
amounts to the membership of the function $a(\cdot, P)$ in a Hajlasz-Sobolev type space of  order $\alpha$ and summability degree $m$. Especially, if $\alpha=1$, then \eqref{ipholder} is equivalent to requiring that $a(\cdot, P)\in W^{1,m}_{\rm loc}(\Omega, \Sym)$ for every $P\in \Sym$, and, in particular, 
$a(\cdot, P)\in {\rm Lip}_{\rm loc}(\Omega)$ when $m=\infty$.

A couple $(u,\pi)\in W^{1,\phi}_{\rm loc}(\Omega,\real^n)\times L^{\widetilde\phi}_{\rm loc}(\Omega)$ is called a local weak solution to
the system \eqref{equa} if $$\dive u=0$$  and
\begin{equation}\label{weak-equa}
  \int_{\Omega} a(x,\E
  u)\cdot\E\varphi-\pi\,\dive\varphi  \,\,\dx
  =
  \int_{\Omega}f\cdot\varphi\,\dx
\end{equation}
 for every  $\varphi\in  W^{1,\phi}(\Omega,\real^n)$ with compact support in $\Omega$.
 
  Clearly, such a definition requires that 
the function $f\cdot \varphi$ be locally integrable for every test function  $\varphi$ as above. Thanks to the Sobolev type embedding \eqref{OS-emb'}, this is guaranteed if 
\begin{align}\label{hpf}
    f\in L^{\widetilde {\phi_n}}_{\rm loc}(\Omega, \real^n),
\end{align}
where $\widetilde {\phi_n}$ denotes the Young conjugate of the function $\phi_n$ defined by \eqref{phin}. As hinted in Section \ref{sec:intro}, we are able to prove fractional differentiability of $V(\E u)$ and $\pi$, provided that the integrability assumption \eqref{hpf} is slightly enhanced. Specifically, this amounts to requiring that 
$$f\in L^\psi_{\rm loc}(\Omega, \Rn)$$
for some Young function $\psi$
fulfilling
    \begin{equation}\label{psi-assumption-main}
  \phi^{-1}(t)^{\theta} \phi_{n}^{-1}(t)^{1-\theta} \psi^{-1}(t)\le  c t
  \qquad\mbox{for}\,\, t\geq 1,
\end{equation}
for some exponent $\theta \in (0,1]$ and some constant $c$. 

 Thanks to \eqref{inverses}, the condition \eqref{psi-assumption-main} ensures that the function $\psi$ lies, up to equivalence, near infinity between the functions $\widetilde {\phi_n}$ and $\widetilde \phi$, which correspond to the values $\theta =0$ and $\theta =1$, respectively.
The former choice is excluded, whereas the latter   entails the strongest possible conclusions.
In particular, \eqref{psi-assumption-main} ensures that 
\begin{align}\label{oct1}
    L^\psi_{\rm loc}(\Omega, \Rn) \subset L^{\widetilde {\phi_n}}_{\rm loc}(\Omega, \real^n)
\end{align}
 -- see Remark \ref{rem:oct} below -- and hence that 
weak solutions to the system \eqref{equa}
are well defined.
\begin{thm}
    \label{main}
   Let $a:\Omega\times\Sym\to\Sym$ be a Carath\'eodory function satisfying \eqref{ip1}--\eqref{ipholder}. 
   Assume that $f\in L^\psi_{\rm loc}(\Omega,\mathbb{R}^n)$ for some Young function
   $\psi$  satisfying \eqref{psi-assumption-main} for some  $\theta \in (0, 1]$ and $c>0$.
     Let $(u,\pi)$ be a local weak solution to the system \eqref{equa}. \\ (i) One has that
\begin{align}
    \label{reg-u}
    V(\mathcal E u)\in B^{\sigma, 2, \infty}
   _{\rm loc}(\Omega, \mathbb R^{n\times n}),
\end{align}
where $\sigma =  \min\big\{\alpha, \theta, \theta s_\phi'/2\big\}$. Moreover,
{\begin{align}
    \label{oct109}    \|V(\E u)\|_{B^{\sigma, 2, \infty}(B_R, \mathbb R^{n\times n})}
    \le  c_1 
\bigg(\int_{B_{2R}}\psi(|f|)\,\dx\bigg)^\varkappa
  + c_2
\bigg(\int_{B_{2R}}\phi(|\nabla u|)\,\dx\bigg)^{\frac12}
   +c_3
\end{align}}
for every ball $B_{2R} \Subset \Omega$
where the exponent $\varkappa$   depends on the constants appearing in the assumptions, the constants $c_1, c_2, c_3$
also on $R$, and $c_2$ also on  
$\|k\|_{L^m(B_{2R})}$.
\\
(ii) 
One has that
\begin{align}
    \label{reg-pi}
    \pi\in  B^{\varrho, \widetilde \phi, \infty}
   _{\rm loc}(\Omega),
\end{align}
where  $ \displaystyle\varrho =\frac{\sigma}{\max\{1, i_\phi'/2\}}$.
Moreover,
{\begin{align}
    \label{boundpi}
%    \sup_{|h|<R}    \int_{B_{R}}\widetilde \phi\bigg(\frac{|\tau_h\pi|}{|h|^\varrho}\bigg)\, \d x
\|\pi\|_{ B^{\varrho, \widetilde \phi, \infty}(B_R)}
    \le
 c_1 
\bigg(\int_{B_{2R}}\psi(|f|)\,\dx\bigg)^\varkappa  +
c_2 \Bigg[\bigg(\int_{B_{2R}}\phi(|\nabla u|)^\varsigma\,\dx\bigg)^{\frac 1\varsigma} 
 + \int_{B_{2R}}\widetilde\phi(|\pi|)\,\dx\Bigg]
   +c_3
\end{align}}
for every ball $B_{3R} \Subset \Omega$,
where the exponent $\varkappa$ and the constants $c_1, c_2, c_3$ depend on the same quantities as above, and the exponent  
$\varsigma$ depends 
 on the constants appearing in the assumptions.
\end{thm}
As explained in Section \ref{sec:intro}, a  result parallel to Part (i) of Theorem \ref{main} holds for local solutions to the system \eqref{elliptic}. Recall that a function $u \in W^{1,\phi}_{\rm loc}(\Omega, \Rn)$ is said to be a weak solution to \eqref{elliptic} if 
\begin{equation}\label{weak-elliptic}
  \int_{\Omega} a(x,\E
  u)\cdot\E\varphi \,\,\dx
  =
  \int_{\Omega}f\cdot\varphi\,\dx
\end{equation}
 for every  $\varphi\in  W^{1,\phi}(\Omega,\real^n)$ with compact support in $\Omega$.
 
\begin{thm}\label{main1}
   Assume that $a$ and $f$ are as in Theorem \ref{main}. Let $u$ be a local weak solution to the system \eqref{elliptic}. Then, the same conclusions as in Part (i) of Theorem \ref{main} hold.
\end{thm}
 The proof of Theorem \ref{main1}
is omitted, as it just consists of a simplification of the proof of Theorem \ref{main}. Indeed, as  a close inspection of that proof reveals, the conclusions of Part (i) are derived via a suitable choice of divergence-free test functions $\varphi$ in \eqref{weak-equa}. They are constructed in two steps, the second one being needed to ensure that the divergence-free condition be fulfilled.
Thanks to the vanishing of $\dive \varphi$, the equations \eqref{weak-equa}
and  \eqref{weak-elliptic} agree. Since the latter is assumed to hold for non-necessarily divergence-free $\varphi$, one can just make use of the test functions $\varphi$ obtained in the first step. This entails an integral identity of the same form as in the proof of Theorem \ref{main}, save that some  terms are now missing. The  terms which are still present can be handled exactly in the same way.

\begin{rem}\label{rem:oct}
    By \eqref{july60}, the assumption \eqref{psi-assumption-main} implies that
    \begin{align}
        \label{july61}
       \phi_{n}^{-1}(t) \psi^{-1}(t)\le  c t
  \qquad\mbox{for}\,\, t\geq 1,
    \end{align}
    for some constant $c$.  Therefore, owing to \eqref{inverses}, there exist $c,t_0 >0$ such that
    \begin{align}
        \label{july62}
        \widetilde {\phi_n} (t)\leq \psi (ct) \quad \text{for $t \geq t_0$.}
    \end{align}
    Hence,  \eqref{psi-assumption-main} guarantees that \eqref{oct1} holds, and hence any function $f$ as in Theorems \ref{main} -- \ref{main1} fulfills \eqref{hpf}.
\end{rem}

\begin{rem}
    \label{recover}
    Theorem \ref{main} recovers, in particular, a special case of a result from \cite{GPS2} corresponding to  $\phi(t)=t^p$, with $p \geq 2$ and $m=\infty$. It also reproduces an instance   from \cite{GPS}, with the choice $\theta =1$  and $m=\infty$. 
\end{rem}
\begin{rem}
    \label{gradient-int}
The inclusion \eqref{reg-u} implies, in particular, that
\begin{align}
    \label{oct20}
    V(\E u) \in W^{\rho,2}_{\rm loc}(\Omega, \mathbb R^{n\times n})
\end{align}
for every $\rho \in (0,\sigma)$. This is a consequence of \eqref{inlc1}. 
\\ Moreover, thanks to the embedding of \cite[Theorem 8.1]{peetre}, the inclusion \eqref{reg-u} ensures that
\begin{align}
    \label{oct21}
    \phi(|\E u|) \in L^{\frac{n}{n-2\sigma}, \infty}_{\rm loc}(\Omega).
\end{align}
Here, $L^{\frac{n}{n-2\sigma}, \infty}_{\rm loc}(\Omega)$ denotes a (local) Marcinkiewicz space, also called weak Lebesgue space.
\end{rem}

\begin{rem}\label{sobolev-reg}
Assume that the  hypotheses of Theorems \ref{main} and \ref{main1} hold with $$\text{$\alpha =1$, \quad $\theta =1$, \quad  $s_\phi\leq 2$}.$$ Then, $\sigma=1$ and, since $B_{\rm loc}^{1,2,\infty}(\Omega, \mathbb R^{n\times n})= W^{1,2}_{\rm loc}(\Omega, \mathbb R^{n\times n})$, one has   that $$V(\E u) \in W^{1,2}_{\rm loc}(\Omega, \mathbb R^{n\times n}).$$
\end{rem}

We conclude this section with an application of Theorem \ref{main} to the case when $\phi$ is a power, and, more generally, a power times a power of a logarithm. Let us emphasize that, even in the former case, the conclusions about the borderline regime corresponding to a power which agrees with  the space dimension  $n$ are new. 

\begin{ex} {\rm Assume that 
$\phi (t)=t^p$ for some $p>1$.
\\ If $1<p<n$, then
$$\phi_n (t) \approx t^{\frac{np}{n-p}}.$$
Hence, the condition \eqref{psi-assumption-main}  holds with
\begin{align}
    \label{psi1}
    \psi (t) = t^{\frac{np}{({1-\theta}) p + n(p-1)}} \quad \text{for $t \geq 0$.}
\end{align}
 If $p=n$, then
$$\phi_n (t) \approx e^{t^{n'}}-1 \quad \text{near infinity.}$$
 The condition \eqref{psi-assumption-main} is thus fulfilled with 
$$\psi (t) \approx t^{\frac n{n-\theta}}(\log t)^{\frac{({1-\theta})(n-1)}{n- \theta}} \quad \text{near infinity.}$$
If $p>n$, then 
$$\phi_n(t) = \infty \quad \text{near infinity,}$$
and \eqref{psi-assumption-main} holds with
$$\psi (t) \approx t^{\frac{p}{p-\theta}} \quad \text{near infinity.}$$
}
\end{ex}

\begin{ex} {\rm Assume that $\phi (t)=t^p \log^\beta (1+t)$ near infinity, for some $p>1$ and $\beta \in \mathbb R$.
\\ If $1<p<n$, then 
$$\phi_n (t) \approx t^{\frac{np}{n-p}}(\log (1+t))^{\frac{n\beta}{n-p}}\quad \text{near infinity.}$$
The condition \eqref{psi-assumption-main} therefore holds with
$$\psi (t) \approx t^{\frac{np}{({1-\theta}) p + n(p-1)}}(\log (1+t))^{-\frac{n\beta}{({1-\theta}) p + n(p-1)}}\quad \text{near infinity.}$$
If  $p=n$ and $\beta <n-1$, then
$$\phi_n (t) \approx e^{t^{\frac{n}{n-1-\beta}}}-1 \quad \text{near infinity.}$$
Hence, the condition \eqref{psi-assumption-main} is fulfilled with
$$\psi (t) \approx t^{\frac{n}{n -\theta}}(\log (1+t))^{\frac{({1-\theta})(n-1)-\beta}{n-\theta}}\quad \text{near infinity.}$$
If  $p=n$ and $\beta =n-1$, then
$$\phi_n (t) \approx e^{e^{t^{n'}}}-e \quad \text{near infinity,}$$
and \eqref{psi-assumption-main} holds with
$$\psi (t) \approx t^{\frac{n}{n -\theta}}(\log (1+t))^{-\frac{\theta (n-1)}{n-\theta}}(\log(\log (1+t)))^{\frac{({1-\theta})(n-1)}{n-\theta}}\quad \text{near infinity.}$$
If either $p=n$ and $\beta >n-1$, or $p>n$, then
$$\phi_n(t) = \infty \quad \text{near infinity,}$$
whence \eqref{psi-assumption-main} holds with
$$\psi (t) \approx t^{\frac{p}{p -\theta}}(\log (1+t))^{-\frac{\theta \beta}{p-\theta}}\quad \text{near infinity.}$$}

\end{ex}

\section{A modular Poincar\'e-Sobolev inequality  on domains with finite measure}\label{sec:SP}

The definition \eqref{phin} of the Sobolev conjugate $\phi_n$ of a Young function $\phi$ requires that the condition \eqref{conv0} be fulfilled. When dealing with embedding theorems on sets of finite measure, such as \eqref{OS-emb} and \eqref{OS-emb'}, and hence of norm inequalities, this is not a restriction. Indeed, as hinted above, replacing $\phi$ with a Young function equivalent near infinity and satisfying \eqref{conv0} results in equivalent norms in $W^{1,\phi}_0(\Omega, \rn)$ and $W^{1,\phi}(\Omega, \rn)$.

Our proofs require, however, a sharp Poincar\'e-Sobolev type inequality in integral form. The
replacement of $\phi$ described above is not possible in this kind of inequalities without changing their form. In order to avoid the unnecessary restriction \eqref{conv0} on the function $\phi$, 
 we establish a Poincar\'e-Sobolev inequality, for functions in $W^{1,\phi}_0(\Omega, \rn)$, in open sets $\Omega$ with finite measure, involving a variant of the function $\phi_n$, denoted by $ \phi_{0,n}$. The crucial point is that this new function $  \phi_{0,n}$ is well defined for every Young function $\phi$, and is equivalent to $\phi_n$ near infinity, independently of the specific modification of  $\phi$ near zero used in the definition of $\phi_n$. This  property enables us to make use of $\phi_{0,n}$ instead of $\phi_n$ throughout our proofs. Indeed, since  the key assumption 
 \eqref{psi-assumption-main} only involves the behavior of $\phi_n$ near infinity, it can be equivalently formulated with $\phi_n$ replaced with $\phi_{0,n}$.

The function $\phi_{0,n}$ is  defined through the following steps.
Let $\phi$ be any  Young function such that $\phi (1)>0$. Of course, this is not a restriction, since this assumption is always satisfied, up to scaling. In particular, it is automatically fulfilled if $\phi \in\Delta_2$, an hypothesis which, however, is not required for the results of the present section.
Define the function $K_n: [0, \infty) \to [0, \infty)$ as 
\begin{equation}
    \label{K}
    K_n(t)= \chi_{[1, \infty)}(t)\bigg(\int_1^t\Big(\frac \tau{\phi (\tau)}\Big)^{\frac{1}{n-1}}\d \tau\bigg)^{\frac 1{n'}} \quad \text{for $t \geq 0$,}
\end{equation}
and $\overline \phi_{0,n}: [0, \infty) \to [0, \infty)$
as
\begin{equation}
    \label{B}
    \overline \phi_{0,n}(t) = \chi_{[1, \infty)}(t)\bigg(\displaystyle\frac{\phi(K_n^{-1}(t-1))}{K_n^{-1}(t-1)}(t-1)\bigg)^{n'}\quad \text{for $t \geq 0$,}  
\end{equation}
where $K_n^{-1}: [0, \infty) \to [1, \infty]$ denotes the generalized right-continuous inverse of $K_n$. In particular,  we have  that $K_n^{-1}(0)=1$.
\\
Notice that the function $\overline \phi_{0,n}$ need not be convex.  However, since $\overline \phi_{0,n}(t)/(t-1)$ is non-decreasing, it is equivalent to the convex function \begin{align}
    \label{july2}
   \phi_{0,n}(t)= \int_0^t \frac{\overline \phi_{0,n}(\tau)}{\tau-1}\d \tau.
\end{align}
Indeed, one has that
\begin{align}
    \label{SP1}
 \overline \phi_{0,n}(t/2)  \leq   \phi_{0,n}(t)\leq  \overline \phi_{0,n}(t).
\end{align}
The equivalence of the function $\phi_{0,n}$ to $ \phi_{n}$ near infinity is the content of the following result.
\begin{prop}
    \label{equiv}
    Let $\phi$ be a Young function such that $\phi(1)>0$ and let  $ \phi_{0,n}$ be the function given by \eqref{july2}.
    Let
    $\phi_n$ be defined as in \eqref{phin},  with $\phi$ modified near $0$, if necessary, in such a way that \eqref{conv0} holds.
    Then, there exist positive constants $c_1, c_2, t_0$ such that
    \begin{align}
        \label{june10}
        \phi_n(c_1t) \leq  \phi_{0,n}(t) \leq \phi_n(c_2t) \quad \text{for $t\geq t_0$.}
    \end{align}
\end{prop}
\begin{proof} By the definition of $  \phi_{0,n}(t)$ and changes of variables, one has that
\begin{align}
    \label{july1}
     \phi_{0,n}(t)& =  \chi_{[1, \infty)}(t) \phi_{0,n}(t)= \chi_{[1, \infty)}(t)\int_0^{t-1} \frac{\overline \phi_{0,n}(r+1)}{r}\d r  = 
     \chi_{[1, \infty)}(t) \int_0^{t-1}\bigg(\displaystyle\frac{\phi(K_n^{-1}(r))}{K_n^{-1}(r)}r\bigg)^{n'}\frac{1}{r}\,\d r
     \\ \nonumber &=\frac{\chi_{[1, \infty)}(t) }{n'}\int_1^{K_n^{-1}(t-1)}\frac{\phi(s)}{s}\,\d s   \qquad\mbox{for}\,\, t\geq 0.
\end{align}
By the property \eqref{incr}, 
\begin{equation}
    \label{july4}
  \phi (K_n^{-1}(t-1)/2)  \leq \int_1^{K_n^{-1}(t-1)}\frac{\phi(s)}{s}\,\d s \leq \phi ({K_n^{-1}(t-1)})
\end{equation}
if $t \geq K_n(2)+1$. We claim that 
\begin{align}    \label{july5}
    \frac{K_n^{-1}(t-1)}{2}\geq  K_n^{-1}(t/4-1) \quad \text{if $t \geq t \geq 4(K_n(3/2)+1)$.}
\end{align}
To verify this claim, observe that, as $K_n$ is a concave function in $[1, \infty)$, the function $B: [0, \infty) \to [0, \infty)$, given by
$$B(t)= \chi_{[1, \infty)}(t) (K_n^{-1}(t-1)-1) \quad \text{for $t \geq 0$,}$$
is a Young function. By the property \eqref{{kt}} applied to the function $B$, 
\begin{align}
    \label{july6}
\frac{K_n^{-1}(t-1)}{2} \geq K_n^{-1}(t/2-1)-1/2 \quad \text{if $t \geq 2$.}
\end{align}
The   property \eqref{{kt}} again ensures that
\begin{align}
    \label{july7}
 K_n^{-1}(t/2-1) \geq 2 K_n^{-1}(t/4-1)-1 \quad \text{if $t \geq 4$.}
\end{align}
On the other hand,
\begin{align}
    \label{july8}
2 K_n^{-1}(t/4-1)-1 \geq \quad  K_n^{-1}(t/4-1)+ 1/2 \quad \text{if $t \geq 4(K_n(3/2)+1)$.}
\end{align}
Equation \eqref{july5} follows from
\eqref{july6}--\eqref{july8}.
\\Combining Equations \eqref{july1} --  \eqref{july5}  tells us that
\begin{align}
    \label{july10'}
    \frac{1}{n'}\phi(K_n^{-1}(t/4-1)) \leq  \phi_{0,n}(t) \leq \frac{1}{n'}\phi(K_n^{-1}(t-1))
\end{align}
near infinity. 
\\ Now, observe that there exist positive constants $c_1$
and $c_2$ such that
\begin{align}
    \label{july11}
    H_n^{-1}(c_1t) \leq K_n^{-1}(t-1) \leq H_n^{-1}(c_2t) \quad \text{near infinity,}
\end{align}
where $H_n$ is defined with $\phi$ modified near zero, if necessary, in such a way that 
   the condition \eqref{conv0} is fulfilled. Indeed, Equation \eqref{july11}
   is equivalent to
   \begin{align}
    \label{july12'}
    H_n(s)/{c_2} \leq K_n(s) +1\leq H_n(s)/c_1 \quad \text{near infinity,}
\end{align}
and these inequalities hold by the very definitions of $H_n$ and $K_n$. 
Equation \eqref{june10} follows from \eqref{july10'}, \eqref{july11}, and \eqref{{kt}}. 
\end{proof}

The Poincar\'e-Sobolev inequality 
 this section is devoted to is stated in
Theorem \ref{SPnew} below. The critical step in its proof is an
interpolation 
inequality provided by the subsequent Theorem \ref{interpolation}.
These results 
 are variants of those from \cite{cianchi_CPDE}. As mentioned above, the novelty here is that the assumption \eqref{conv0}  can be dispensed with in both inequalities, at the cost of an additional term, depending on the Lebesgue measure $|\Omega|$ of $\Omega$.

\begin{thm}
    \label{SPnew}
    Let $\Omega$ be an open set in $\rn$ with $|\Omega|<\infty$ and let $\phi$ be a Young function such that $\phi(1)>0$. Then, there exists a constant $c=c(n)$ such that
    \begin{align}
        \label{SPnew1}
        \int_\Om   \phi_{0,n}\Bigg(\frac{c|u(x)|}{(|\Omega|+ \int_\Omega \phi(|\nabla u|)\d y)^{1/n}}\Bigg)\d x \leq \int_\Omega \phi(|\nabla u|)\d x
    \end{align}
    for every $u \in W^{1,\phi}_0(\Omega)$.
\end{thm}
In what follows, given any $p\in (1,\infty)$,  we denote by $\overline \phi_{0,p}$ and ${\phi}_{0,p}$ the functions defined as in \eqref{B} and \eqref{july2}, with $n$ replaced with $p$. Moreover, given a measure space $(M,\mu)$ and $p\in (1, \infty)$, we denote by $L^{p,1}(M, \mu)$ the Lorentz space on $(M,\mu)$ of those $\mu$-measurable functions $g : M \to \mathbb R$ which make the norm
$$\|g\|_{L^{p,1}(M, \mu)}= \int_0^\infty \mu(\{|g|>s\})^{\frac 1p}\d s$$
finite.
\begin{thm}
    \label{interpolation}
    Let $(M_1, \mu_1)$ and $(M_2, \mu_2)$ be measure spaces with $\mu_1(M_1)<\infty$. Let   $1<p<\infty$ and let $T$ be a sublinear operator such that
    \begin{align}
        \label{inter1}
        T: L^{1}(M_1, \mu_1) \to L^{p'}(M_2,\mu_2),
    \end{align}
    with norm $N_0$, 
    and
    \begin{align}
        \label{inter2}
        T: L^{p,1}(M_1, \mu_1) \to L^{\infty}(M_2,\mu_2),
    \end{align}
    with norm $N_1$. Assume that $\phi$ is a Young function such that $\phi(1)>0$.
    Then, there exists a constant $c=c(p, N_0, N_1)$ such that
    \begin{align}
        \label{inter3}
        \int_{M_2}  \phi_{0,p}\Bigg(\frac{c|Tg(y)|}{(\mu_1(M_1)+ \int_{M_1} \phi(|g|)\d \mu_1(x))^{1/p}}\Bigg)\d \mu_2(y) \leq \int_{M_1} \phi(|g|)\d \mu_1(x).
    \end{align}
\end{thm}
\begin{proof}
Given a function $g:M_1\to\R$ and $t\geq 0$, define 
$$g_t= {\rm sign} (g) \min\{|g|,t\} \quad \text{and} \quad g^t= g- g_t.$$
To begin with, observe that
\begin{align}
    \label{SP4}
    \|g_t\|_{L^{p,1}(M_1)} &= \int_0^t\mu_1(\{|g|>s\})^{\frac 1p}\d s \leq \mu_1(M_1)^{\frac 1p}+ \chi_{[1,\infty)}(t)\int_1^t\mu_1(\{|g|>s\})^{\frac 1p}\d s
    \\ \nonumber &\leq 
    \mu_1(M_1)^{\frac 1p} + K_p(t) \bigg(\int_0^\infty \frac{\phi(s)}s \mu_1(\{|g|>s\})\d s\bigg)^{\frac{1}{p}}
   \leq 
    \mu_1(M_1)^{\frac 1p} + K_p(t) \bigg(\int_0^\infty \phi'(s)\mu_1(\{|g|>s\})\d s\bigg)^{\frac{1}{p}}
    \\ \nonumber & =  \mu_1(M_1)^{\frac 1p} + K_p(t) \bigg(\int_{M_1}\phi(|g|)\d \mu_1\bigg)^{\frac 1p}   \quad \text{for $t \geq 0$.}
\end{align}
On the other hand,
\begin{align}\label{SP5}
\|g^t\|_{L^1(M_1)} & = \int_t^\infty \mu_1(\{|g|>s\})\d s= \frac{t}{\phi(t)} \int_t^\infty \frac{\phi(t)}{t}\mu_1(\{|g|>s\})\d s
\\ & \leq \nonumber  
\frac{t}{\phi(t)} \int _0^\infty \phi'(t)\mu_1(\{|g|>s\})\d s
   \leq  \frac{t}{\phi(t)} \int_{M_1}\phi(|g|)\d \mu_1 \quad \text{for $t \geq 0$.}
\end{align}
   Now,  let $k$ be a positive constant to be chosen later. One has

   \begin{align}
       \label{SP6}
       \int_{M_2}    \phi_{0,p}\Big(\frac{|Tg|}{4k}\Big)\, d\mu_2 & = \int_1^\infty    \phi_{0,p}'(t) \mu_2(\{|Tg|>4kt\})\d t \leq 
       \int_1^\infty  \frac{\overline  \phi_{0,p}(2t)}t \mu_2(\{|Tg|>4kt\})\d t
       \\ \nonumber &
\leq  \int_1^\infty  \frac{\overline  \phi_{0,p}(t)}t \mu_2(\{|Tg|>2kt\})\d t
     \\  \nonumber &  = - \int_1^\infty \frac{ \overline \phi_{0,p}(t)}{t^{p'}} \frac{d}{dt}\bigg(\int_t^\infty \mu_2(\{|Tg|>2ks\})s^{p'-1}\d s\bigg)\d t
     \\  \nonumber &  \leq
     \int_1^\infty \frac{d}{dt}\bigg(\frac{\overline  \phi_{0,p}(t)}{t^{p'}} \bigg)\bigg(\int_t^\infty \mu_2(\{|Tg|>2ks\})s^{p'-1}\d s\bigg)\d t.
   \end{align}
  Note that the first inequality holds as $ \phi_{0,p}'(t)\le\frac{\overline \phi_{0,p}(2t)}{t}$, since the function $\frac{\overline \phi_{0,p}(t)}{t-1}$ is non-decreasing, and    the last equality  holds thanks to the fact that $\phi_{0,p}(1)=0$.
\\Set
   $$\lambda (t)= K_p^{-1}(t-1) \quad \text{for $t \geq 1$.} $$
   The sublinearity of $T$ ensures that
   $$\mu_2(\{|Tg|>2ks\}) \leq \mu_2(\{|Tg_{\lambda(t)}|>ks\})+ \mu_2(\{|Tg^{\lambda(t)}|>ks\})$$
   for   $s,t\geq 1$. Moreover, since the function $\frac{ \overline \phi_{0,p}(t)}{t^{p'}}$ is non-decreasing, we have that $\frac{d}{dt}\Big(\frac{ \overline \phi_{0,p}(t)}{t^{p'}}\Big) \geq 0$ for $t \geq 1$. Hence,
   \begin{align}
       \label{SP7}
        \int_1^\infty & \frac{d}{dt}\bigg(\frac{ \overline \phi_{0,p}(t)}{t^{p'}} \bigg)\bigg(\int_t^\infty \mu_2(\{|Tg|>2ks\})s^{p'-1}\d s\bigg)\d t 
       \\ \nonumber & \leq 
 \int_1^\infty \frac{d}{dt}\bigg(\frac{\overline  \phi_{0,p}(t)}{t^{p'}} \bigg)\bigg(\int_t^\infty \mu_2(\{|Tg_{\lambda(t)}|>ks\})s^{p'-1}\d s\bigg)\d t 
 \\ \nonumber & 
      \quad + \int_1^\infty \frac{d}{dt}\bigg(\frac{ \overline \phi_{0,p}(t)}{t^{p'}} \bigg)\bigg(\int_t^\infty \mu_2(\{|Tg^{\lambda(t)}|>ks\})s^{p'-1}\d s\bigg)\d t = I+J.
   \end{align}
   Owing to \eqref{inter2} and \eqref{SP4}, with $t$ replaced with $\lambda (t)$, 
   \begin{align}
       \label{SP8}
\|Tg_{\lambda(t)}\|_{L^\infty(M_2)}\leq N_1\|g_{\lambda(t)}\|_{L^{p,1}(M_1)}&\leq N_1\bigg(\mu_1(M_1)^{\frac 1p} + K_p(\lambda (t)) \bigg(\int_{M_1}\phi(|g|)\d \mu_1\bigg)^{\frac 1p}\bigg)
\\ \nonumber & \leq N_1\bigg(\mu_1(M_1)^{\frac 1p} + t \bigg(\int_{M_1}\phi(|g|)\d \mu_1\bigg)^{\frac 1p}\bigg) \quad \text{for $t \geq 1$.} 
   \end{align}
   The choice
   \begin{align}\label{k} k= N_1 
       \bigg(\mu_1(M_1)^{\frac 1p} + \bigg(\int_{M_1}\phi(|g|)\d \mu_1\bigg)^{\frac 1p}\bigg)
   \end{align}
   entails that 
   \begin{align}
       \label{SP9}
       ks \geq N_1\bigg(\mu_1(M_1)^{\frac 1p} + t \bigg(\int_{M_1}\phi(|g|)\d \mu_1\bigg)^{\frac 1p}\bigg) 
   \end{align}
    if $s\geq t\geq 1$. From \eqref{SP8} and \eqref{SP9} we deduce that
    $\mu_2(\{|Tg_{\lambda(t)}|>ks\})=0$
    if $s\geq t\geq 1$, whence
    \begin{align}
        \label{SP10}
        I=0.
    \end{align}
    By \eqref{inter1}, we have that
\begin{align}
    \label{SP12}
\bigg(\int_t^\infty \mu_2(\{|Tg^{\lambda(t)}|>ks\})s^{p'-1}\d s\bigg)\leq  \frac{1}{p'} \Big(\frac{N_0}{k}\Big)^{p'}\bigg(\int_{\lambda (t)}^\infty \mu_1(\{|g|>s\})\d s\bigg)^{p'}.
\end{align}
Moreover, by \eqref{SP5} with $t$ replaced with $\lambda(t),$ and the 
definitions of  the latter and of  $\overline \phi_{0,p}(t)$,
    \begin{align}
        \label{SP11}
        \frac{  \overline \phi_{0,p}(t)}{t^{p'}} \bigg(\int_{\lambda (t)}^\infty \mu_1(\{|g|>s\})\d s\bigg)^{p'} & \leq \frac{\overline  \phi_{0,p}(t)}{t^{p'}}\bigg(\frac{\lambda (t)}{\phi (\lambda (t))}\bigg)^{p'} \bigg(\int_{M_1}\phi(|g|)\d \mu_1\bigg)^{p'}
\leq \bigg(\int_{M_1}\phi(|g|)\d \mu_1\bigg)^{p'}
    \end{align}
    for $t \geq 1$. 
    The following chain holds thanks to \eqref{SP12} and \eqref{SP11}:
    \begin{align}\label{SP14}
    J & \leq \frac{1}{p'} \Big(\frac{N_0}{k}\Big)^{p'}\int_1^\infty \frac{d}{dt}\bigg(\frac{\overline  \phi_{0,p}(t)}{t^{p'}} \bigg) \bigg(\int_{\lambda (t)}^\infty \mu_1(\{|g|>s\})\d s\bigg)^{p'}  \d t 
    \\ \nonumber &
    = \frac{1}{p'} \Big(\frac{N_0}{k}\Big)^{p'}\bigg[ \frac{\overline  \phi_{0,p}(t)}{t^{p'}}
    \bigg(\int_{\lambda (t)}^\infty \mu_1(\{|g|>s\})\d s\bigg)^{p'}\bigg|_{t=1}^{t=\infty} - 
    \int_1^\infty \frac{ \overline \phi_{0,p}(t)}{t^{p'}} \frac{d}{dt}\bigg(\int_{\lambda (t)}^\infty \mu_1(\{|g|>s\})\d s\bigg)^{p'}  \d t \bigg]
    \\ \nonumber & \leq 
    \frac{1}{p'} \Big(\frac{N_0}{k}\Big)^{p'}\bigg[\bigg(\int_{M_1}\phi(|g|)\d \mu_1\bigg)^{p'} 
    + p'
    \int_1^\infty \frac{ \overline \phi_{0,p}(t)}{t^{p'}}\bigg(\int_{\lambda (t)}^\infty \mu_1(\{|g|>s\})\d s\bigg)^{p'-1} 
    \mu_1(\{|g|>\lambda (t)\})\lambda '(t)\, \d t\bigg].
    \end{align}
    Note that
\begin{align}
    \label{SP15}
    \int_1^\infty& \frac{  \overline \phi_{0,p}(t)}{t^{p'}}\bigg(\int_{\lambda (t)}^\infty \mu_1(\{|g|>s\})\d s\bigg)^{p'-1} 
    \mu_1(\{|g|>\lambda (t)\})\lambda '(t)\, \d t
    \\  \nonumber & = \int_1^\infty \frac{ \overline \phi_{0,p}(\lambda ^{-1}(\tau))}{\lambda ^{-1}(\tau)^{p'}}\bigg(\int_{\tau}^\infty \mu_1(\{|g|>s\})\d s\bigg)^{p'-1} 
    \mu_1(\{|g|>\tau\})\, \d \tau
    \\  \nonumber & \leq 
    \int_1^\infty \frac{ \overline \phi_{0,p}(\lambda ^{-1}(\tau))}{\lambda ^{-1}(\tau)^{p'}} \bigg(\frac{\tau}{\phi(\tau)}\bigg)^{p'-1}
\bigg(\int_{\tau}^\infty\frac{\phi(s)}{s} \mu_1(\{|g|>s\})\d s\bigg)^{p'-1} 
    \mu_1(\{|g|>\tau\})\, \d \tau
    \\  \nonumber & \leq\int_1^\infty \frac{\phi(\tau)}{\tau} 
\bigg(\int_{\tau}^\infty\frac{\phi(s)}{s} \mu_1(\{|g|>s\})\d s\bigg)^{p'-1} 
    \mu_1(\{|g|>\tau\})\, \d \tau
     \\  \nonumber & =\frac{1}{p'}\bigg(\int_{1}^\infty\frac{\phi(s)}{s} \mu_1(\{|g|>s\})\d s\bigg)^{p'} \leq \frac{1}{p'} \bigg(\int_{M_1}\phi(|g|)\, \d \mu_1\bigg)^{p'}.
\end{align}
Combining the inequalities \eqref{SP14} and \eqref{SP15} yields
\begin{align}
    \label{SP16}
    J \leq \frac{2}{p'} \Big(\frac{N_0}{k}\Big)^{p'}\bigg(\int_{M_1}\phi(|g|)\, \d \mu_1\bigg)^{p'},
\end{align}
    whence, via \eqref{SP6} and \eqref{SP7} we obtain that
    \begin{align}
        \label{SP17}
        \int_{M_2}   \phi_{0,p}\Big(\frac{|Tg|}{4k}\Big)\, \d\mu_2 \leq \frac{2}{p'} \Big(\frac{N_0}{k}\Big)^{p'}\bigg(\int_{M_1}\phi(|g|)\, \d \mu_1\bigg)^{p'}.
    \end{align}
    Since $k$ is given by \eqref{k}, the latter inequality implies that
    \begin{align}
        \label{SP18}
           \int_{M_2}   \phi_{0,p}\Bigg(\frac{|Tg|}{4 N_1 
       \big(\mu_1(M_1)^{\frac 1p} + \big(\int_{M_1}\phi(|g|)\d \mu_1\big)^{\frac 1p}\big)
        }\Bigg)\, \d\mu_2 \leq \frac{2}{p'} \Big(\frac{N_0}{N_1}\Big)^{p'} \int_{M_1}\phi(|g|)\, \d \mu_1.
    \end{align}
    Hence, the inequality \eqref{inter3} follows.
\end{proof}

\begin{proof}
    [Proof of Theorem \ref{SPnew}] 
Given any function $u \in W^{1,\phi}_0(\Omega)$, denote by $u^*: [0,|\Omega|]\to [0,\infty] $ its decreasing rearrangement. The P\'olya-Szeg\"o principle ensures that $u^*$ is locally absolutely continuous, and 
\begin{align}
    \label{PS}
    \int_{\Omega}\phi(|\nabla u|)\, \dx \geq \int_0^{|\Omega|}\phi\big({ n\omega^{\frac1n}}s^{\frac{1}{n'}}(-{u^*}'(s)\big)\, \d s,
\end{align}
where $\omega_n$ denotes the Lebesgue measure of the unit ball in $\rn$. Moreover,
\begin{align}
    \label{june1}
    \int_\Omega {\overline{\phi}_{0,n}}(|u|)\, \d x = \int_0^{|\Omega|}{\overline{\phi}_{0,n}}(u^*(s))\, \d s= \int_0^{|\Omega|}\overline{\phi}_{0,n}\bigg(\int_s^{|\Omega|}-{u^*}'(r)\,\d r \bigg)\, \d s.
\end{align}
Thanks to Equations \eqref{PS} and \eqref{june1}, the inequality \eqref{SPnew1} will follow if we show that
\begin{align}
    \label{june2}
\int_0^{|\Omega|}\overline{\phi}_{0,n}\Bigg(\frac{c\int_s^{|\Omega|}r^{-1/n'}g(r)\,\d r}{\big(|\Omega|+ \int_0^{|\Omega|}\phi(g)dr\big)^{1/n}} \Bigg)\, \d s \leq \int_0^{|\Omega|}\phi(g)ds
\end{align}
for every measurable function $g: (0, |\Omega|) \to [0, \infty)$. One can verify that the operator $T$ defined as 
$$Tg(s) = \int_s^{|\Omega|}r^{-1/n'}|g(r)|\,\d r $$
is sublinear and such that
$$T: L^1(0, |\Omega|) \to L^{n'}(0, |\Omega|) \quad \text{and} \quad  L^{n,1}(0, |\Omega|) \to L^{\infty}(0, |\Omega|),
$$
with norms depending only on $n$. The inequality \eqref{june2} is thus a consequence of Theorem \ref{SPnew}.
\end{proof}

\section{Proof of Theorem \ref{main}, Part (i), a priori estimates}\label{sec:proof1}

  This section is devoted to a proof of the inequality \eqref{oct109} under stronger regularity assumptions on the functions $a(\cdot, P)$ and $f$. Thanks to a result from \cite{GPS}, these assumptions ensure that expressions like the one on left-hand side of \eqref{oct109} are finite, and hence  can be absorbed from one side to the other one in inequalities which are derived from the weak formulation of the problem \eqref{equa}.

We begin with a few preliminary results.
The following lemma   from \cites{DieForTomWan19, GPS} concerns relations among a Carath\'eodory function $a(x,P)$ as in Theorem \ref{main}
and shifts of the function $\phi$.

\begin{lemalph}
  \label{lem:hammer}
  \textit{
  Let $\phi$ be a Young function satisfying \eqref{indices3} and let $a:\Omega\times\Sym\to\Sym$ be a Carath\'eodory function fulfilling the
properties~\eqref{ip1} and \eqref{ip2b}.
  Then, 
  \begin{align}\label{a-coercivity}
   \big( a(x,P)-a(x,Q)\big)\cdot (P-Q )
    &\geq c  \phi_{|{Q}|} \left( |{P-Q}| \right) 
    ,\\[0.8ex]
    \label{seconda}
    a(x,Q) \cdot Q
    &  \eqsim  \phi_{|{Q}|}(|{Q}|) 
  \end{align}
  for every~$P,Q \in \Sym$ and a.e. $x,y\in\Omega $, up to 
  constants  depending only on $i_\phi,s_\phi,\nu$, and $L$.
\\  Assume, in addition, that the function $a$ also satisfies \eqref{ip2}. Then, 
  \begin{align}
    \label{a-Lipschitz}
    |{a(x,P)-a(x,Q)}| &\le c \phi_{|{Q}|}'(|{P-Q}|),\\
     (a(x,P)-a(x,Q))\cdot (P-Q)
    &\eqsim \phi_{|{Q}|} \left( |{P-Q}| \right).\label{a-coercivity-2}
  \end{align}
  for every~$P,Q \in \Sym$ and a.e. $x,y\in\Omega $, with the same dependence of the constants as above.}
\end{lemalph}
We shall also make use of a classical iteration lemma -- see e.g. \cite[Lemma 6.1]{Giusti}.
\begin{lemalph}\label{lem:Giaq}
\textit{    Let $0<R_1<R_2$ and let   
    $f:[R_1,R_2]\to[0,\infty)$ be a bounded function. Assume that 
\begin{equation*}
      f(r_1)\le\vartheta f(r_2)+\sum_{i=1}^N\frac {A_i}{(r_2-r_1)^{\alpha_i}} + B
       \qquad\mbox{for all }R_1<r_1<r_2<R_2,
    \end{equation*}
    for some nonnegative constants $B$, and $A_i$ and $\alpha_i$, with $i=1, \dots N$,
    and for $\vartheta\in(0,1)$. Then, 
    \begin{equation*}
      f(R_1)\le c
      \bigg(\sum_{i=1}^N\frac {A_i}{(R_2-R_1)^{\alpha_i}} + B
      \bigg),
    \end{equation*}
    for some constant $c$ depending on $\alpha_i$, $i=1, \dots N$,
 and $\vartheta$.   }
  \end{lemalph}

The assumption \eqref{ipholder} amounts to the membership of the function $a(\cdot,P)$ in a fractional  space of order $\alpha$, with a norm depending on $P$ just through $\phi'(|P|)$. The next result is a kind of Sobolev embedding and tells us that  $a(\cdot,P)$ belongs to a  space of lower order, with a higher exponent of integrability. A key feature of this property in view of our applications is that the dependence on $P$ of the 
norm in the new space is unchanged.

\begin{lem}\label{lem:ipholder-beta}
 Let $a:\Omega\times\Sym\to\Sym$ be a Carath\'eodory function satisfying 
the assumptions ~\eqref{ip2b} and \eqref{ipholder}  for some
Young function 
 $\phi$ fulfilling \eqref{indices3} and
some
  $k\in L^{ m}(\Omega)$ with $
  m>\frac{n}{\alpha}$. Given $\beta\in(0,\alpha)$, set
  \begin{align}
      \label{oct5}
      m_\beta = \begin{cases}
          \frac{nm}{n-(\alpha -\beta)m} &\quad \text{if \,\,$ \alpha -\frac nm < \beta < \alpha$}
          \\ \text{any number $>\frac n \beta$} &\quad \text{if\,\, $0< \beta  \leq \alpha -\frac nm$.}
      \end{cases}
  \end{align}
Then, for
  every   bounded open set $\Omega'\Subset\Omega$, there exists a function
  $k_\beta\in L^{ m_\beta}(\Omega')$ 
  such that
  \begin{equation}
    \label{ipholder-beta}
    |a(x,P)-a(y,P)|
    \le
    (k_\beta(x)+k_\beta(y))|x-y|^\beta \phi'(|P|)
\end{equation}
 for a.e. $x,y\in\Omega'$ and every $P\in\Sym$. Moreover,
\begin{equation}\label{est:lem-ipholder-beta}
  \|k_\beta\|_{L^{ m_\beta}(\Omega')}\le c(\|k\|_{L^{
  {m}}(\Omega)}+1),
  \end{equation}
for some  constant depending on $n,
m,\alpha,\beta,L,
\mathrm{diam}(\Omega')
$,  and  $\mathrm{dist}(\Omega',\partial\Omega)$. 
\end{lem}

\begin{proof}
  First, assume that $\beta>\alpha -\frac{n}{
  m} $, and hence ${m}_\beta =\frac{ m
  n}{n-(\alpha-\beta) m}$.
    Given any open set 
  $\Omega'$ as in the statement,  choose a cut-off function
  $\zeta\in \mathrm{Lip}(\Omega)$ such that $0\leq \zeta \leq 1$ in $\Omega$,    $\zeta= 1$ in
  $\Omega'$, $\spt\zeta \subset\Omega$ and $\|\nabla\zeta\|_{L^\infty}\le \mathrm{dist}(\Omega',\partial\Omega)^{-1}$. Fix $P\in\Sym$ and define the function
  $F_P : \Rn \to [0, \infty)$ as
  \begin{equation*}
    F_P(x) = \zeta(x)\frac{a(x,P)}{\phi'(|P|)}
    \qquad\mbox{for }x\in\R^n.
  \end{equation*}
  We begin by observing that there exists a function  $\overline k\in L^{ m}(\R^n)$
such that
\begin{equation}\label{global-ipholder}
    |F_P(x)-F_P(y)|
    \le
    \big(\overline k(x)+\overline k(y)\big)|x-y|^\alpha \qquad \textrm{for   $x,y\in\R^n$.}
  \end{equation}
 Set $U= \textrm{spt} \,\zeta$. The left-hand side of \eqref{global-ipholder} vanishes if
  $x,y\in\R^n\setminus U$. Hence, it trivially holds whatever $\overline k$ is. 
  Assume now that
  $y\in U$, the case when $x\in U$ being analogous. One has:
  \begin{align*}
    |F_P(x)-F_P(y)|
    &\le
    \zeta(x)\frac{|a(x,P)-a(y,P)|}{\phi'(|P|)}
    +
    |\zeta(x)-\zeta(y)|\frac{|a(y,P)|}{\phi'(|P|)}.
  \end{align*}
  If $x\in\R^n\setminus U$, the first addend on the right-hand side
  vanishes, whereas, if $x\in U$, it can be estimated by
  \eqref{ipholder}. Moreover, the second addend can be bounded via \eqref{ip2b}, the mean value theorem and the trivial inequality
$(\zeta(x)+\zeta(y))^{1-\alpha}\le \chi_U(x)+\chi_U(y)$ for $x, y \in \mathbb R^n$.
  Altogether, we obtain:
  \begin{align*}
    |F_P(x)-F_P(y)|
    &\le
      \Big(k(x)\chi_U(x)+k(y)\chi_U(y)+L\|\nabla\zeta\|_{L^\infty}^{\alpha}(\chi_U(x)+\chi_U(y))\Big)|x-y|^\alpha.
  \end{align*}
  Therefore, \eqref{global-ipholder} holds with $\overline
  k=(k+L\|\nabla\zeta\|_{L^\infty}^{\alpha})\chi_U$. 
\\ Next,
  recall that the  sharp fractional maximal function of $F_P$ is defined as
  \begin{equation*}
    (F_P)_\alpha^\sharp(x)=\sup_{\mathcal  Q\ni x}
    \frac{1}{|\mathcal Q|^{1+\alpha/n}}\int_\mathcal  Q |F_P(y)-(F_P)_Q|\,\d y
    \qquad\mbox{for }x\in\R^n,
  \end{equation*}
  where  $\mathcal Q$ denotes a cube in $\Rn$ and $(F_P)_\mathcal Q$ the average of $F_P$ over $\mathcal Q$. The property \eqref{global-ipholder} implies that
\begin{equation}\label{oct7}
  (F_P)_\alpha^\sharp(x)
  \le
  c\sup_{\mathcal Q\ni x}\mint_\mathcal Q\mint_\mathcal Q\big(\overline k(y)+\overline k(z)\big)\d y\dz
  \le
  cM\overline k(x) \qquad\mbox{for }x\in\R^n,
\end{equation}
where $M\overline k$ denotes  the maximal function  of $\overline k$. In particular,  note that the right-hand
side is independent of $P$. From
\cite[Inequality (9.12)]{DevoreSharpley:1984}  and
the inequality \eqref{oct7} one obtains that
\begin{align}\label{sharp-max-bound}
  (F_P)_\beta^\sharp(x)
  \le
    cI_{\alpha-\beta}\big[(F_P)_\alpha^\sharp\big](x)
  \le
    c'I_{\alpha-\beta}\big[M\overline k\big](x) \qquad\mbox{for }x\in\R^n,
\end{align}
for suitable constants $c$ and $c'$   depending on $n$, $\alpha$ and $\beta$. Here, $I_\gamma$ denotes the Riesz potential operator in $\Rn$ of order $\gamma$.
Set 
$$k_\beta = I_{\alpha-\beta}\big[M\overline k\big].$$
Standard 
 properties of  Riesz potentials and of the maximal function ensure that
\begin{align}\label{oct10}
  \|k_\beta\|_{L^{ {m}_\beta}({ \Rn})}
  \le
  c\|M\overline k\|_{L^{ m}(\R^n)}
  \le
  c'\|\overline k\|_{L^{ m}(\R^n)}
  \le c''(\|k\|_{L^{ m}({\Omega})}+1),
\end{align}
for some constants $c, c', c''$ depending at most on $n, m,\alpha,\beta,L,\mathrm{dist}(\Omega',\partial\Omega)$ and $|U|$. 
An estimate in terms of the  sharp fractional maximal function (see
e.g. \cite[Inequality (2.16)]{DevoreSharpley:1984}) and \eqref{sharp-max-bound} yield
\begin{align}\label{oct9}
  |F_P(x)-F_P(y)|
 \le
  c\big[(F_P)_\beta^\sharp(x)+(F_P)_\beta^\sharp(y)\big]|x-y|^\beta \le
  c'[k_\beta(x)+k_\beta(y)]|x-y|^\beta  
\end{align}
for a.e. $x,y\in\R^n$, every $P\in\Sym$, and {constants $c$ and $c'$ depending on $n, \alpha, \beta$}.
By the very definition of $F_P$, the inequalities \eqref{ipholder-beta} and \eqref{est:lem-ipholder-beta}
can be derived from \eqref{oct9} and \eqref{oct10}. 
\\ Consider now the complementary case when 
$0<\beta\le \alpha -\frac{n}{ m}$. One can
choose an arbitrary $ m_\beta>m$ and apply the
preceding result with $\beta$ replaced with $\overline\beta=\alpha-\frac{n}{
  m}+\frac{n}{ m_\beta}$. This application is legitimate, as 
  $\overline\beta> \alpha-\frac{n}{
  m}\ge\beta$. As a consequence, we obtain a function $k_\beta\in L^{
  m_\beta}(\Omega')$ such that
\begin{align*}
  |a(x,P)-a(y,P)|
  \le
  c [k_\beta(x)+k_\beta(y)]|x-y|^{\overline \beta}\phi'(|P|)\le c\,\mathrm{diam}(\Omega')^{\overline \beta -\beta}[k_\beta(x)+k_\beta(y)]|x-y|^{\beta}\phi'(|P|)
\end{align*}
for a.e. $x,y\in\Omega'$ and any $P\in\Sym$. The conclusions follow also in this case.
\end{proof}

\begin{proof}
    [Proof of Theorem \ref{main}, Part (i):  a priori estimate for $V(\E u)$ in $B^{\sigma, 2, \infty}_{\rm loc}(\Omega, \mathbb R^{n\times n})$] The purpose of this step is establishing the inequality \eqref{oct109} under the additional assumptions that 
    \begin{align}
        \label{Kinf}
        k\in L^\infty_{\rm loc}(\Omega),
    \end{align}
    and
        \begin{align}
        \label{finf}
        f\in L^\infty_{\rm loc}(\Omega).
    \end{align}
    Unless otherwise specified, throughout this proof $c$, $c'$, $c''$ will denote positive constants which  may change from one Equation to another one, and may depend on $n, \nu, L, \phi, \psi$.
  The dependence on additional parameters will be explicitly indicated via an index in the relevant constant.
 \\
    Since $(u,\pi)$ is a local weak solution to the system \eqref{equa}, in particular, the function $u\in
  W^{1,\phi}(\Omega,\R^n)$ satisfies the equation:
  \begin{equation}
    \label{equa-div-free}
    \int_\Omega  a(x,\E u)\cdot\E\varphi\,\d x
    =
    \int_\Omega f\cdot\varphi\,\d x
  \end{equation}
for every $\varphi\in W^{1,\phi}(\Omega,\R^n)$ with compact support in $\Omega$ and such that $\dive\varphi=0$.
\\ Thanks to the current assumptions \eqref{Kinf} and \eqref{finf}, by \cite[Theorem 1.2]{GPS}, we have that $ V(\mathcal E u)\in B^{\upsilon, 2, \infty}_{\rm loc}(\Omega, \mathbb R^{n\times n})$,
where $\upsilon = \min\{\alpha, s_\phi'/2\}$. Notice that 
\begin{equation}   \label{july24}
\upsilon \geq \sigma.
\end{equation}
Hence, we already know that
\begin{align}\label{july25}
    V(\mathcal E u)\in B^{\sigma, 2, \infty}_{\rm loc}(\Omega, \mathbb R^{n\times n}).
\end{align}
Fix a ball 
$B_{ \frac32R}\Subset\Omega$, and let $0< R_1<R_2\le R$.  Set  $h_0=\min\{{\frac14}(R_2-R_1),1\}$.
  Fix $r_1,r_2$ such that
  $R_1\le r_1<r_2
  \le
  \frac12(R_1+R_2)
  {\le}
  R_2-{2}h_0$ and choose a cut-off function $\zeta\in
  C^\infty (\mathbb R^n)$, with compact support in $B_{r_2}$, such that $0\leq \zeta \leq 1$, $\zeta=1$ on $B_{r_1}$ and
  \begin{equation}\label{properties-zeta}
    |\nabla ^2\zeta|+|\nabla\zeta|^2\le\frac{c}{(r_2-r_1)^2}.
  \end{equation}
Fix $h\neq0$ such that $|h|\le h_0$. We 
  construct a test function  in   Equation \eqref{equa-div-free} as follows.
  Consider the map
\begin{equation}\label{definition-g}
g=\dive(\zeta^2\tau_hu).
  \end{equation}
  Notice that
\begin{align}
    \label{july12}
g=2\zeta\nabla\zeta\cdot\tau_hu\in W^{1,\phi}_0(B_{r_2}, \rn),
\end{align}
  where the equality holds since $\dive u=0$. As $\zeta$ has a compact support in $B_{r_2}$,
  \begin{equation*}
g_{_{B_{r_2}}}=\mint_{B_{r_2}}\dive(\zeta^2\tau_hu)\,\dx=0.
  \end{equation*}
  Define
  \begin{equation*}
    w=\mathcal{B}g,
  \end{equation*}
  where $\mathcal{B}$ denotes the Bogovski\u\i\ operator on $B_{r_2}$
  and observe that
\begin{align}
    \label{july14}
    w\in W^{1,\phi}_0(B_{r_2},\R^n),
\end{align}
  and 
\begin{equation}\label{choice-w-1}
     \dive w=g\mbox{\qquad in }B_{r_2},
  \end{equation}
   thanks to \eqref{july15} and \eqref{july16}.
 \\ Set 
 \begin{align}\label{gamma}
     \gamma=\min\Big\{1,\frac1{s_\phi-1}\Big\}.
 \end{align} 
 The use  of the
estimate~\eqref{Bogovski-gradient} for the function
  $\frac{R^{1-\gamma}}{|h|}g$ and the Equations \eqref{july12} and \eqref{properties-zeta}  imply that  
\begin{equation}\label{choice-w-2}
      \displaystyle\int_{B_{r_2}}\phi\left(\frac{R^{1-\gamma}|\nabla w|}{|h|}\right)\,\dx
        \le c\int_{B_{r_2}}\phi\left(\frac{R^{1-\gamma}|g|}{|h|}\right)\,\dx
      \le c'\int_{B_{r_2}}\phi\left(\frac{R^{1-\gamma}}{r_2-r_1}\frac{|\tau_hu|}{|h|}\right)\,\dx.
  \end{equation}
 Applying ~\eqref{Bogovski-diff-quotients-2} to the
  function $\eta g$, where  
  \begin{equation}
    \eta=\frac{r_2-r_1}{|h|^\gamma},\label{choice-lambda}
  \end{equation}
yields the bound 

  \begin{align}\label{choice-w-3}
      \int_{B_{r_2+|h|}}\phi\left(\frac{\eta|\tau_{-h}\nabla w|}{|h|}\right)\,\dx
        \leq c\int_{B_{r_2}}\left[\phi\left(\eta|\nabla g|\right)+\phi\left(\frac{\eta|g|}{{r_2}}\right)\right]\,\dx.
  \end{align}
Now, set 
\begin{equation}\label{def-phi}
    \varphi=\tau_{-h}(\zeta^2\tau_hu)-\tau_{-h}w.
  \end{equation}
As $\spt\zeta\subset B_{r_2}$ and $r_2+|h|\le
  r_2+h_0\le R_2$, from the very definition of $g$ and Equation \eqref{choice-w-1} we infer that
  \begin{equation}\label{def-phi'}
    \varphi \in W^{1,\phi}_0(B_{R_2},\R^n) \quad \text{and } \quad \dive\varphi=0.
  \end{equation}
Hence,  $\varphi$ is an admissible 
  test function in \eqref{equa-div-free}. Therefore, the following identity holds:
  \begin{equation}\label{test-1}
    \int_{B_R} a(x,\E u) \cdot \E(\tau_{-h}(\zeta^2\tau_hu))\,\dx
    =
    \int_{B_R}  a(x,\E u)\cdot\E
    (\tau_{-h}w)\,\dx
    +
    \int_{B_R}f\cdot\varphi\,\dx.
  \end{equation}
 By setting, for brevity,
  \begin{equation}\label{def-G}
    G=\zeta\nabla\zeta\otimes \tau_hu,
  \end{equation}
  one has that
  \begin{align}\label{test-lhs}
    \int_{B_R}  a(x,\E u)\cdot \E(\tau_{-h}(\zeta^2\tau_hu)) \,\dx & =
      \int_{B_R}  a(x,\E u)\cdot\tau_{-h}(\zeta^2\tau_h\E u)\,\dx
      +
      \int_{B_R}  a(x,\E u)\cdot\tau_{-h}(G+G^T) \,\dx\\\nonumber
    & =
      \int_{B_R}\zeta^2  \tau_ha(x,\E u)\cdot \tau_h\E u \,\dx
      +
      2\int_{B_R} a(x,\E u)\cdot\tau_{-h}G \,\dx,
  \end{align}
  where the last equality holds thanks to \eqref{july19} and the fact that  $a(x, \mathcal E u)\in \Sym$ and hence it is orthogonal to any antisymmetric matrix. 
   Notice that
  \begin{align}
    \label{def-A-B}
    \tau_ha(x,\E u)
    & =\big( a(x+he_i,\E u(x+h e_i))-a(x+h e_i,\E u(x))\big) +\big( a(x+he_i,\E u(x))-a(x,\E u(x))\big)\\\nonumber
    &=\mathcal{A}_h+\mathcal{B}_h.
  \end{align}
 From \eqref{test-1}, ~\eqref{test-lhs}, \eqref{def-A-B}, and the definition of $\varphi$ we deduce that
  \begin{align}\label{def-I-1-5}
    \int_{B_R}\zeta^2\ \mathcal{A}_h \cdot \tau_h\E
    u \,\dx
    &=
    -\int_{B_R}\zeta^2 \mathcal{B}_h \cdot \tau_h\E
      u \,\dx
    -2\int_{B_R}  a(x,\E u) \cdot \tau_{-h}G \,\dx +\int_{B_R}  a(x,\E u) \cdot 
    \tau_{-h}\E w \,\dx\\\nonumber
    &\qquad+ \int_{B_R}f\cdot\tau_{-h}(\zeta^2\tau_hu)\,\dx
          -\int_{B_R} f\cdot \tau_{-h}w\,\dx =
    \mathrm{I}_1+\mathrm{I}_2+\mathrm{I}_3+\mathrm{I}_4+\mathrm{I}_5.
  \end{align}
  Owing  to the inequalities~\eqref{a-coercivity} and \eqref{a-coercivity-split}, 
  \begin{equation}
    \label{lower-bound}
    \int_{B_R}\zeta^2 \mathcal{A}_h \cdot \tau_h\E
    u \,\dx
    \ge
    c\int_{B_{r_1}}|\tau_h V(\E u)|^2\,\dx.
  \end{equation}
  We now provide upper bounds for the  terms $\mathrm{I}_i$, $i=1, \dots , 5$. 
\\ Let us begin with 
 $\mathrm{I}_1$.   Recall that  $0<\sigma \leq \alpha$ and let 
  $m_\sigma$ be an exponent obeying \eqref{oct5}, with $\beta$ replaced with $\sigma$. Observe that we may choose $m_\sigma$ in such a way that $m_\sigma > \frac 
  {n}\sigma$.
 Thanks to Lemma \ref{lem:ipholder-beta},  given $l>1$ such that $B_{l R} \Subset \Omega$, there exists  a function $k_\sigma \in L^{m_\sigma}(B_R)$  such that 
  \begin{equation}
    \label{oct2}
    |a(x,P)-a(y,P)|
    \le
    (k_\sigma(x)+k_\sigma(y))|x-y|^\sigma \phi'(|P|)
\end{equation}
 for a.e. $x,y\in B_R$ and every $P\in\Sym$. Moreover,
 \begin{equation}\label{oct3}
  \|k_\sigma\|_{L^{m_\sigma}(B_R)}\le c(\|k\|_{L^{
  {m}}(B_{l R})}+1),
  \end{equation}
for some constant $c$ depending on $n,
m,\alpha,\sigma,L, R$ and $l$.
 Set
  \begin{align*}
    \E(h)=\max\{|\E u(x)|,|\E u(x+he_i)|\}.
  \end{align*}
 Fix $\kappa\in(0,1]$ to be chosen later.
 By \eqref{oct2},  
 the monotonicity of 
  $\phi'$,  Young's inequality, \eqref{july21},  and~\eqref{a-coercivity-split}, the following chain holds:
  \begin{align}\label{pre-I-1-bound}
    |\mathrm{I}_1|
    &\le
      \int_{B_{R}}\zeta^2|\mathcal{B}_h|\,|\tau_h\E u|\,\dx\le
    |h|^{\sigma}\int_{B_{R}}\zeta^2(k_\sigma(x)+k_\sigma(x+he_i))\phi'(|\E u|)|\tau_h\E u|\,\dx\\\nonumber
    &\le
      |h|^{\sigma}\int_{B_{R}}\zeta^2(k_\sigma(x)+k_\sigma(x+he_i))\frac{\phi'(\E(h))}{\E(h)}\,\E(h)\,|\tau_h\E
      u|\,\dx\\\nonumber
    &\le
      \kappa \int_{B_{r_2}}\zeta^2\frac{\phi'(\E(h))}{\E(h)}\,|\tau_h\E
      u|^2\,\dx+
      \frac 1\kappa|h|^{2\sigma}\int_{B_{r_2}}\zeta^2(k_\sigma(x)+k_\sigma(x+he_i))^2\phi'(\E(h))\,\E(h)\,\dx
      \\ \nonumber & \leq 
 c\kappa \int_{B_{r_2}}|\tau_hV(\E u)|^2\,\dx      
+\frac{s_\phi}{\kappa}|h|^{2\sigma}\int_{B_{r_2}}\zeta^2(k_\sigma(x)+k_\sigma(x+he_i))^2\phi(\E(h))\,\dx.  
  \end{align}
  From the inequality
  $\phi(\E(h))\le \phi(|\E u(x)|)+\phi(|\E u(x+he_i)|)$  and Equations \eqref{est-translation}  and \eqref{seconda-split} we have that
  \begin{align}
      \label{july30}
      &\int_{B_{r_2}}\zeta^2(k_\sigma(x)+k_\sigma(x+he_i))^2\phi(\E(h))\,\dx\leq \left(\int_{B_{r_2}}(k_\sigma(x)+k_\sigma(x+he_i))^{ {{m_\sigma}}}\,\dx\right)^{\frac{2}{ {{m_\sigma}}}}\left(\int_{B_{r_2}}\Big(\zeta^2\phi(\E(h))\Big)^{\frac{ {{m_\sigma}}}{ {{m_\sigma}}-2}}\,\dx\right)^{\frac{ {{m_\sigma}}-2}{ {{m_\sigma}}}}
      \\ \nonumber &\leq \left(2^{{{m_\sigma}-1}}\int_{B_{R_2}}k_\sigma^{ {{m_\sigma}}}\,\dx\right)^{\frac{2}{ {{m_\sigma}}}}\left(\int_{B_{r_2}}\Big(\zeta^2(x)\phi(\E(u(x)))\Big)^{\frac{ {{m_\sigma}}}{ {{m_\sigma}}-2}}\,\dx + \int_{B_{r_2}}\Big(\zeta^2(x)\phi(\E(u(x+he_i)))\Big)^{\frac{ {{m_\sigma}}}{ {{m_\sigma}}-2}}\,\dx
      \right)^{\frac{ {{m_\sigma}}-2}{ {{m_\sigma}}}}
      \\ \nonumber &= \left( 2^{{{m_\sigma}-1}}\int_{B_{R_2}}k_\sigma^{ {{m_\sigma}}}\,\dx\right)^{\frac{2}{ {{m_\sigma}}}}\left(\int_{B_{r_2}}\Big(\zeta^2(x)\phi(\E(u(x)))\Big)^{\frac{ {{m_\sigma}}}{ {{m_\sigma}}-2}}\,\dx + \int_{B_{r_2}(x_0+he_i)}\Big(\zeta^2(x-he_i)\phi(\E(u(x)))\Big)^{\frac{ {{m_\sigma}}}{ {{m_\sigma}}-2}}\,\dx
      \right)^{\frac{ {{m_\sigma}}-2}{ {{m_\sigma}}}}
      \\ \nonumber &\leq \left( 2^{{{m_\sigma}-1}}
      \int_{B_{R_2}}k_\sigma^{ {{m_\sigma}}}\dx\right)^{\frac{2}{ {{m_\sigma}}}}\left(\int_{B_{r_2}}|V(\mathcal Eu)|^{\frac{ {2m_\sigma}}{ {{m_\sigma}}-2}}\,\dx + \int_{B_{r_2}(x_0+he_i)}|V(\mathcal Eu)|^{\frac{ {2m_\sigma}}{ {{m_\sigma}}-2}}\,\dx
      \right)^{\frac{ {{m_\sigma}}-2}{ {{m_\sigma}}}}.
\end{align}
Since ${{m_\sigma}}>\frac{n}{\sigma}$, given any   $\beta \in (\frac{n}{m_\sigma,  }\sigma)$, we have that $\frac{2{{m_\sigma}}}{{{m_\sigma}}-2}<\frac{2n}{n-2\beta}$.
Via H\"older's inequality with suitable exponents, one can hence  deduce that
\begin{align}
    \label{july33}
\left(\int_{B_{r_2}}|V(\mathcal Eu)|^{\frac{ {2m_\sigma}}{ {{m_\sigma}}-2}}\,\dx  
      \right)^{\frac{ {{m_\sigma}}-2}{ {{m_\sigma}}}} \leq \left(\int_{B_{r_2}}|V(\E u)|^2\,\dx\right)^{{1-\frac{n}{{m_\sigma}\beta}}}\left(\int_{B_{r_2}}|V(\E u)|^{\frac{2n}{n-2\beta}}\,\dx\right)^{{\frac{n-2\beta}{{m_\sigma}\beta}}}.
\end{align}
Analogously,
\begin{align}
    \label{july34}
\left(\int_{B_{r_2}(x_0+he_i)}|V(\mathcal Eu)|^{\frac{ {2m}}{ {{m_\sigma}}-2}}\,\dx  
      \right)^{\frac{ {{m_\sigma}}-2}{ {{m_\sigma}}}} \leq \left(\int_{B_{r_2}(x_0+he_i)}|V(\E u)|^2\,\dx\right)^{{1-\frac{n}{{m_\sigma}\beta}}}\left(\int_{B_{r_2}(x_0+he_i)}|V(\E u)|^{\frac{2n}{n-2\beta}}\,\dx\right)^{{\frac{n-2\beta}{{m_\sigma}\beta}}}.
\end{align}
From the inequality \eqref{besov-emb} one can infer that
\begin{align}   \label{july31}\nonumber
&\left(\int_{B_{r_2}}|V(\mathcal Eu)|^{\frac{ {2n}}{ n-2\beta}}\,\dx 
      \right)^{{\frac{n-2\beta}{{m_\sigma}\beta}}} \\
      &\qquad\leq  
      c\bigg((R_2-r_2)^{2(\sigma-\beta)}\sup_{|h|\leq h_0}\int_{B_{R_2}}\frac{|\tau_hV(\E u)|^2}{|h|^{2\sigma}}\,\dx+\frac{1}{(R_2-r_2)^{2\beta}}\int_{{B_{{ R_2}}}}|V(\E u)|^2\,\dx\bigg)^{{\frac{n}{{m_\sigma}\beta}}},
\end{align}
and
\begin{align}   \label{july32}
&\left(\int_{B_{r_2}(x_0+he_i)}|V(\mathcal Eu)|^{\frac{ {2n}}{ n-2\beta}}\,\dx  \right)^{{\frac{n-2\beta}{{m_\sigma}\beta}}}
\leq 
\left(\int_{B_{r_2+h_0}}|V(\mathcal Eu)|^{\frac{ {2n}}{ n-2\beta}}\,\dx \right)^{{\frac{n-2\beta}{{m_\sigma}\beta}}}
\\ \nonumber &\leq \left(\int_{B_{\frac 12 (r_2+R_2)}}|V(\mathcal Eu)|^{\frac{ {2n}}{ n-2\beta}}\,\dx  \right)^{{\frac{n-2\beta}{{m_\sigma}\beta}}}
\\ \nonumber
  & \leq  c
  {
  \bigg((R_2-r_2)^{2(\sigma-\beta)}
  }
  \sup_{{|h|\leq h_0}}\int_{B_{R_2}}\frac{|\tau_hV(\E u)|^2}{|h|^{2\sigma}}\,\dx+
  {
  \frac{1}{(R_2-r_2)^{2\beta}}
  }
  \int_{B_{R_2}}|V(\E u)|^2\,\dx\bigg)^{ {\frac{n}{{m_\sigma}\beta}}},
\end{align}
Combining the inequalities \eqref{july30}--\eqref{july32} and making use of Young's inequality enable us to deduce that
\begin{align}\label{est:fractional0}
\int_{B_{r_2}}&\zeta^2({k_\sigma}+{k_\sigma}(x+h))^2\phi(\E(h))\,\dx
  \\  \nonumber \le& c\left((R_2-r_2)^{ {\frac{n(\sigma-\beta)}{\beta}}}\int_{B_{R}}k_\sigma^{ {{m_\sigma}}} \,\dx\right)^{\frac{2}{ {{m_\sigma}}}}\left(\int_{B_{R}}|V(\E u)|^2\,\dx\right)^{ {\frac{{m_\sigma}\beta-n}{{m_\sigma}\beta}}}\left(\sup_{|h|\le h_0}\int_{B_{R_2}}\frac{|\tau_hV(\E u)|^2}{|h|^{2\sigma}}\,\dx\right)^{ {\frac{n}{{m_\sigma}\beta}}}
  \\ \nonumber  & \quad+\frac{c}{(R_2-r_2)^{ {\frac{2n}{{m_\sigma}}}}}
  \left(\int_{B_{R}}k_\sigma^{{m_\sigma}} \,\dx\right)^{\frac{2}{{m_\sigma}}}\int_{B_{R}}|V(\E u)|^2\,\dx
 \\ \nonumber & \leq 
  \varepsilon \left(\sup_{|h|\leq h_0}\int_{B_{R_2}}\frac{|\tau_hV(\E u)|^2}{|h|^{2\sigma}}\,\dx\right) +c_\varepsilon\left((R_2-r_2)^{ {\frac{n(\sigma-\beta)}{\beta}}}\int_{B_{R}}k_\sigma^{{m_\sigma}}(x)\,\dx\right)^{ {\frac{2}{{m_\sigma}-\frac{n}{\beta}}}}\int_{B_{R}}|V(\E u)|^2\,\dx\\ \nonumber & \quad  + { \frac{c}{(R_2-r_2)^{ {\frac{2n}{{m_\sigma}}}}}\left(\int_{B_{R}}k_\sigma^{{m_\sigma}}\,\dx\right)^{\frac{2}{{m_\sigma}}}\int_{B_{R}}|V(\E u)|^2\,\dx}\\
 \nonumber & \leq 
  \varepsilon \left(\sup_{|h|\leq h_0}\int_{B_{R_2}}\frac{|\tau_hV(\E u)|^2}{|h|^{2\sigma}}\,\dx\right) \\ \nonumber & \quad +
   {
  \frac{c_\varepsilon}{(R_2-r_2)^{2\sigma}}
  }\left({ 1}+(R_2-r_2)^{
 {
 \sigma {m_\sigma}-n
 }
 }
 \int_{B_{R}}k_\sigma^{{m_\sigma}}\,\dx\right)^{ {\frac{2}{{m_\sigma}-\frac{n}{\beta}}}}\int_{B_{R}}|V(\E u)|^2\,\dx
 \end{align}
 Choosing $\varepsilon$ so small that $c_\kappa \varepsilon \leq c \kappa$ and coupling  \eqref{pre-I-1-bound} with \eqref{est:fractional0} yield:
 \begin{align}\label{I-1-bound}
    |\mathrm{I}_1|
    &\le
      c\kappa \int_{B_{r_2}}|\tau_hV(\E u)|^2\,\dx+
     c\kappa |h|^{2\sigma} \left(\sup_{{|h|\leq h_0}}\int_{B_{R_2}}\frac{|\tau_hV(\E u)|^2}{|h|^{2\sigma}}\,\dx\right)\\\nonumber 
  &+ \frac{c_{\kappa}'|h|^{2\sigma}}{(R_2-R_1)^{ {2\sigma}}}
  \left(1+ {R^{\sigma {m_\sigma}-n}}\int_{B_{R}}k_\sigma^{{m_\sigma}}\,\dx\right)^{{\frac{2}{{m_\sigma}-\frac{n}{\beta}}}}\int_{B_{R}}|V(\E u)|^2\,\dx\\\nonumber
  &\le
      c\kappa \int_{B_{r_2}}|\tau_hV(\E u)|^2\,\dx+
     c_\kappa |h|^{2\sigma} \left(\sup_{{|h|\leq h_0}}\int_{B_{R_2}}\frac{|\tau_hV(\E u)|^2}{|h|^{2\sigma}}\,\dx\right)\\\nonumber 
  &+ \frac{c_{\kappa}'|h|^{2\sigma}}{(R_2-R_1)^{ {2\sigma}}}
  \left(1+
  {R^{\sigma {m_\sigma}-n}}
 \int_{B_{R}}k_\sigma^{{m_\sigma}}\,\dx\right)^{{\frac{2}{{m_\sigma}-\frac{n}{\beta}}}}\left(\int_{B_{R}}\phi(|\nabla u|)\,\dx\right).
  \end{align}
  In the last step, we applied \eqref{seconda-split}. 
    \\ To bound $\mathrm{I}_2$, 
one can exploit the assumption \eqref{ip2b},  
  the inclusion $\spt G\subset B_{r_2}$, and the inequality $r_2+|h|\le R_2$ to obtain:
  \begin{align*}
    |\mathrm{I}_2|
    &\le
      2\int_{B_R}\big|a(x,\E u)\big|\,|\tau_{-h}G\big|\,\dx
      \le
    c\int_{B_{R_2}}\phi'(|\E u|)\,|\tau_{-h}G|\,\dx.
  \end{align*}
  Let 
  $\gamma$ be the exponent defined in \eqref{gamma}. From  Young's
  inequality~\eqref{eq:young}, the inequality
  \eqref{est-diff-quotient} with $v=(r_2-r_1)G/|h|^\gamma$,
  and the inequalities~\eqref{homogeneity-phi-star}   and \eqref{july59}  we hence infer that
 \begin{align}\label{first-est-I2}
    |\mathrm{I}_2|
    &\le
      c|h|^{\min\{s_\phi',2\}}\int_{B_{R_2}}\phi'(|\E u|)\,\frac{|\tau_{-h}G|}{|h|^{\gamma+1}}\,\dx\\\nonumber
    &\le
      c|h|^{\min\{s_\phi',2\}}
      \int_{B_{R_2}}\left[
      \kappa
      \phi\left(\frac{r_2-r_1}{|h|^\gamma}\frac{|\tau_{-h}G|}{|h|}\right)
      +
      c_\kappa \widetilde \phi\left(\frac1{r_2-r_1}\phi'(|\E u|)\right)\right]\,\dx\\\nonumber
    &\le
      c|h|^{\min\{s_\phi',2\}}\left[\kappa\int_{B_{r_2}}
      \phi\left(\frac{r_2-r_1}{|h|^\gamma}|\nabla G|\right)\,\dx
      +
      c_\kappa \Bigg(\frac{1}{(r_2-r_1)^{i_\phi'}}
     +\frac{1}{(r_2-r_1)^{s_\phi'}}
      \Bigg)\int_{B_{R_2}}\phi(|\E u|)
      \,\dx\right].
  \end{align}
  Note that the last inequality also rests on the inclusion $\spt G\subset B_{r_2}$.
  Thanks to \eqref{properties-zeta},
  \begin{align*}
    |\nabla G|
    &\le
      |\nabla\zeta|\,|\zeta \nabla \tau_hu|
      +
      (|\nabla ^2\zeta|+|\nabla\zeta|^2)\,|\tau_hu|\le
      |\nabla\zeta|\,|\nabla(\zeta\tau_hu)|+(|\nabla^2\zeta|+2|\nabla\zeta|^2)\,|\tau_hu|\\
    &\le
    \frac{c}{r_2-r_1}|\nabla(\zeta\tau_hu)|+\frac{c}{(r_2-r_1)^2}\,|\tau_hu|.
\end{align*}
Hence, 
\begin{align}\label{bound-phi-DG}
    \phi\left(\frac{r_2-r_1}{|h|^\gamma}|\nabla G|\right)
    &\le
      c\phi\left(\frac{|\nabla(\zeta\tau_hu)|}{|h|^\gamma}\right)
      +
      c\phi\left( {R^{-\gamma}\frac{R}{r_2-r_1}}\frac{|\tau_hu|}{|h|}\right),\\\nonumber
    &\le
      c\phi\left(\frac{|\nabla (\zeta\tau_hu)|}{|h|^\gamma}\right)
      +
      {c\left(\frac{1}{R^{\gamma s_\phi}}+\frac{1}{R^{\gamma i_\phi}}\right)\frac{R^{s_\phi}}{(r_2-r_1)^{s_\phi}}}\phi\left(\frac{|\tau_hu|}{|h|}\right).
  \end{align}
  Here, we have made use of the monotonicity of $\phi$,
the inequality~\eqref{sub-additivity-phi},  the bound $|h|\le R$, Equation \eqref{homogeneity-phi}, and the inequality
  $r_2-r_1\le R $. From 
 ~\eqref{first-est-I2} and \eqref{bound-phi-DG} we deduce, 
 via ~\eqref{est-diff-quotient}, that
  \begin{align}\label{second-est-I2}
    |\mathrm{I}_2|
    &\le
      c\kappa |h|^{\min\{s_\phi',2\}}\int_{B_{r_2}}
      \phi\left(\frac{|\nabla (\zeta\tau_hu)|}{|h|^\gamma}\right) \dx\\\nonumber
    &\qquad
      +
      c_\kappa |h|^{\min\{s_\phi',2\}}\left( {\left(\frac{1}{R^{\gamma s_\phi}}+\frac{1}{R^{\gamma i_\phi}}\right)\frac{R^{s_\phi}}{(r_2-r_1)^{s_\phi}}}
      +
      \frac{1}{(r_2-r_1)^{i_\phi'}}
      +\frac{1}{(r_2-r_1)^{s_\phi'}}
    \right)\int_{B_{R_2}}\phi(|\nabla u|)
      \,\dx.
  \end{align}
Owing to Korn's
  inequality~\eqref{korn-0} 
  and~\eqref{july36}, 
  \begin{align}\label{bound-higher-deriv}
    \int_{B_{r_2}}
    \phi\left(\frac{|\nabla(\zeta\tau_hu)|}{|h|^\gamma}\right)\,\dx
    &\le
    c\int_{B_{r_2}}
      \phi\left(\frac{|\E(\zeta\tau_hu)|}{|h|^\gamma}\right)\,\dx\le
      c'\int_{B_{r_2}}
      \left[
      \phi\left(\frac{\zeta|\tau_h\E u|}{|h|^\gamma}\right)
      +
      \phi\left(\frac{|\nabla\zeta||\tau_hu|}{|h|^\gamma}\right)\right]\dx.
  \end{align}
  The use of Equations ~\eqref{add-shift}, \eqref{homogeneity-phi-shifted}, \eqref{a-coercivity-split}, and 
  the equality
  $\gamma\max\{s_\phi,2\}=\min\{s_\phi',2\}$ result in the chain:
  \begin{align}\label{bound-hot}
    \int_{B_{r_2}}
    \phi\left(\frac{\zeta|\tau_h\E u|}{|h|^\gamma}\right)
    \dx
    &\le
      c\int_{B_{r_2}}
      \left(
      \phi_{|\E u|}\left(\frac{|\tau_h\E u|}{|h|^\gamma}\right)
      +
      \phi(|\E u|)
      \right)\, \dx \le
      c'\int_{B_{r_2}}
      \left(
      \frac{\phi_{|\E u|}(|\tau_h\E u|)}{|h|^{\min\{s_\phi',2\}}}
      +
      \phi(|\E u|)
      \right)\, \dx\\\nonumber
    &\le
      c''\int_{B_{r_2}}
      \left(
      \frac{|\tau_hV(\E u)|^2}{|h|^{\min\{s_\phi',2\}}}
      +
      \phi(|\E u|)
      \right)\, \dx.
  \end{align}
Next,  we have that
  \begin{align}\label{bound-lot}
    \int_{B_{r_2}}
    \phi\left(\frac{|\nabla\zeta||\tau_hu|}{|h|^\gamma}\right)\dx
    &\le
      \int_{B_{r_2}}
      \phi\left(\frac{cR^{1-\gamma}}{r_2-r_1}
      \frac{|\tau_hu|}{|h|}\right)\dx\le
    {c\left(\frac{1}{R^{\gamma s_\phi}}+\frac{1}{R^{\gamma i_\phi}}\right)\frac{R^{s_\phi}}{(r_2-r_1)^{s_\phi}}}\int_{B_{r_2}}
      \phi\left(\frac{|\tau_hu|}{|h|}\right)\dx\\\nonumber
    &\le
      {c\left(\frac{1}{R^{\gamma s_\phi}}+\frac{1}{R^{\gamma i_\phi}}\right)\frac{R^{s_\phi}}{(r_2-r_1)^{s_\phi}}}\int_{B_{R_2}}
      \phi(|\nabla u|)\dx,
  \end{align}
  where the first inequality holds by \eqref{properties-zeta} and the inequality $|h|\le R$, the second one by
\eqref{homogeneity-phi}, and the last one by \eqref{est-diff-quotient}.
The inequalities \eqref{bound-higher-deriv}--\eqref{bound-lot} yield:
  \begin{align}\label{key-estimate}
    \int_{B_{r_2}}
    \phi\left(\frac{|\nabla (\zeta\tau_hu)|}{|h|^\gamma}\right) \dx
    &\le
      \frac{c}{|h|^{\min\{s_\phi',2\}}}\int_{B_{r_2}}
      |\tau_hV(\E u)|^2\, \dx
      \\\nonumber
      &+
    c\left({1}+\bigg(\frac{1}{R^{\gamma s_\phi}}+\frac{1}{R^{\gamma i_\phi}}\bigg)\frac{R^{s_\phi}}{(r_2-r_1)^{s_\phi}}\right)\int_{B_{R_2}}
      \phi(|\nabla u|)\dx.
  \end{align}
  Coupling the latter estimate with~\eqref{second-est-I2} tells us that
  \begin{align}\label{est-I2}
    |\mathrm{I}_2|
    &\le
      c\kappa\int_{B_{r_2}}
      |\tau_hV(\E u)|^2\, \dx
      \\\nonumber
      &+
      c_\kappa |h|^{\min\{s_\phi',2\}}\left( 1+{\left(\frac{1}{R^{\gamma s_\phi}}+\frac{1}{R^{\gamma i_\phi}}\right)\frac{R^{s_\phi}}{(r_2-r_1)^{s_\phi}}}
      +
      \frac{1}{(r_2-r_1)^{i_\phi'}}
      +\frac{1}{(r_2-r_1)^{s_\phi'}}
    \right)\int_{B_{R_2}}\phi(|\nabla u|)
      \,\dx,
  \end{align}
  for every $\kappa\in(0,1]$.
  \\
  Our task is now to derive an estimate for $\mathrm{I}_3$.
  The assumption~\eqref{ip2b} implies:
  \begin{align*}
    |\mathrm{I}_3|
    &\le
    \int_{B_{R_2}}|a(x,\E u)|\,|\tau_{-h}\E w|\,\dx\le
    { L}|h|^{\min\{s_\phi',2\}}\int_{B_{R_2}}\phi'(|\E
      u|)\frac{|\tau_{-h}\nabla  w|}{|h|^{\gamma+1}}\,\dx\\
    &=
      L|h|^{\min\{s_\phi',2\}}\int_{B_{R_2}}\frac{1}{r_2-r_1}\phi'(|\E
      u|)\frac{\eta |\tau_{-h}\nabla w|}{|h|}\,\dx,
  \end{align*}
  where $\eta$ is defined by~\eqref{choice-lambda}.
  The rightmost side of the latter chain can be bounded via Young's
  inequality~\eqref{eq:young} and Equations \eqref{choice-w-3},
  \eqref{homogeneity-phi-star},
  and \eqref{july59} to deduce that
  \begin{align}\label{first-est-I3}
    |\mathrm{I}_3|
    &\le
      c|h|^{\min\{s_\phi',2\}}\int_{B_{R_2}}
      \left[\kappa\phi\left(\frac{\eta |\tau_{-h}\nabla w|}{|h|}\right)
      +c_\kappa\widetilde \phi\left(\frac1{r_2-r_1}\phi'(|\E
      u|)\right)\right]\dx\\\nonumber
    &\le
      c|h|^{\min\{s_\phi',2\}}
      \left[\kappa
      \int_{B_{r_2}}\left[\phi\left(\eta|\nabla g|\right)+\phi\left(\frac{\eta|g|}{{r_2}}\right)\right]\dx
       +
       c_\kappa\Bigg(
       \frac{1}{(r_2-r_1)^{i_\phi'}}
      +\frac{1}{(r_2-r_1)^{s_\phi'}}
      \Bigg)\int_{B_{R_2}}
      \phi(|\E u|)\dx\right].
  \end{align}
  From Equation \eqref{properties-zeta} and the inequality $|h|^{1-\gamma}\le R^{1-\gamma}$ we
  obtain:
  \begin{align*}
    \eta |\nabla g|
   & \le
      2\eta |\nabla\zeta|\,|\zeta \nabla \tau_hu|
      +
      2\eta(|\nabla ^2\zeta|+|\nabla\zeta|^2)\,|\tau_hu|\\ \nonumber &\le
      2\eta|\nabla\zeta|\,|\nabla (\zeta\tau_hu)|+2\eta(|\nabla ^2\zeta|+2|\nabla\zeta|^2)\,|\tau_hu| \le
      \frac{c|\nabla (\zeta\tau_hu)|}{|h|^{\gamma}}
      +
      \frac{cR^{1-\gamma}|\tau_hu|}{(r_2-r_1)|h|}\,,
\end{align*}
and
\begin{align*}
  \frac{\eta|g|}{{r_2}}
  \le
  \frac{c|\tau_hu|}{{r_2}|h|^\gamma}
  \le
  \frac{cR^{1-\gamma}|\tau_hu|}{(r_2-r_1)|h|}.
\end{align*}
Hence,  via \eqref{sub-additivity-phi} and~\eqref{homogeneity-phi}, we deduce:
\begin{align*}
  \int_{B_{r_2}}\left[\phi\left(\eta|\nabla g|\right)+\phi\left(\frac{\eta|g|}{{r_2}}\right)\right]\dx
  &\le
    c\int_{B_{r_2}}
    \phi\left(\frac{|\nabla (\zeta\tau_hu)|}{|h|^\gamma}\right) \dx\\\nonumber
    &+
     {c\left(\frac{1}{R^{\gamma s_\phi}}+\frac{1}{R^{\gamma i_\phi}}\right)\frac{R^{s_\phi}}{(r_2-r_1)^{s_\phi}}}\int_{B_{r_2}}
    \phi\left(\frac{|\tau_hu|}{|h|}\right) \dx.
\end{align*}
The last integral can be estimated through the inequality ~\eqref{est-diff-quotient}. Combining the resultant estimate
with~\eqref{first-est-I3} and making use of the
inequality~\eqref{key-estimate} yield:
  \begin{align}\label{est-I3}
    |\mathrm{I}_3|
    &\le
      c|h|^{\min\{s_\phi',2\}}
      \Bigg[\kappa
      \int_{B_{r_2}}
      \phi\left(\frac{|\nabla (\zeta\tau_hu)|}{|h|^\gamma}\right)\dx
     \\\nonumber &\qquad+
      c_\kappa\left( {\left(\frac{1}{R^{\gamma s_\phi}}+\frac{1}{R^{\gamma i_\phi}}\right)\frac{R^{s_\phi}}{(r_2-r_1)^{s_\phi}}}
      +
      \frac{1}{(r_2-r_1)^{i_\phi'}}
      +\frac{1}{(r_2-r_1)^{s_\phi'}}
      \right)
      \int_{B_{R_2}}\phi(|\nabla u|)\dx\Bigg]\\\nonumber
    &\le
      c\kappa\int_{B_{r_2}} |\tau_hV(\E u)|^2 \dx
      \\\nonumber &\qquad+
      c_\kappa|h|^{\min\{s_\phi',2\}}\left( {\left(\frac{1}{R^{\gamma s_\phi}}+\frac{1}{R^{\gamma i_\phi}}\right)\frac{R^{s_\phi}}{(r_2-r_1)^{s_\phi}}}
      +
      \frac{1}{(r_2-r_1)^{i_\phi'}}
      +\frac{1}{(r_2-r_1)^{s_\phi'}}
      \right)
      \int_{B_{R_2}}\phi(|\nabla u|)\dx.
  \end{align}
A bound for the term $\mathrm{I}_4$ calls into play assumption \eqref{psi-assumption-main} and the Poincar\'e-Sobolev inequality \eqref{SPnew1}. Thanks to \eqref{june10}, the former amounts to requiring that
\begin{equation}\label{psi-assumption''}
  \phi^{-1}(t)^{\theta}  \phi_{0,n}^{\,-1}(t)^{1-\theta} \psi^{-1}(t)\le  c(t + 1)
  \qquad\mbox{for}\,\, t\geq 0.
\end{equation}
This inequality in turn ensures that
\begin{align}
    \label{june9}
    r^{{\theta}}s^{{1-\theta}}t \leq c\big( \phi (r) +  \phi_{0,n} (s) + \psi (t) + 1\big) \quad \text{for $r,s,t \geq 0$.}
\end{align}
Let $\lambda>0$ to be chosen later and 
set
\begin{equation*}
  \Phi_\lambda =\int_{B_{r_2}} \phi\left(\frac{|\nabla (\zeta^2\tau_h u)|}{\lambda|h|^{ {\theta\gamma}}}\right)\dx.
\end{equation*}
Thanks to the inequalities  \eqref{june9}, \eqref{est-translation}, and \eqref{est-diff-quotient},
\begin{align}\label{est-RHS-psi}
|\mathrm{I}_4|&=\int_{B_{r_2}}|\tau_{-h}(\zeta^2\tau_hu)|\,|f|\,\dx\\ \nonumber
    &=
      |h|^{\theta(1+\gamma)}\lambda (|B_{r_2}| +\Phi_\lambda)^{\frac{{1-\theta}}{n}}
      \int_{B_{r_2}}\Big(\frac{|\tau_{-h}(\zeta^2\tau_hu)|}{\lambda|h|^{ {\theta\gamma+1}}}\Big)^{\theta}\Big(\frac{|\tau_{-h}(\zeta^2\tau_hu)|}{\lambda|h|^{ {\theta\gamma}}(|B_{r_2}|  +\Phi_\lambda)^{\frac{1}{n}}}\Big)^{1-\theta}\,|f|\,\dx\\ \nonumber  
    &\le c
     |h|^{\theta(1+\gamma)}\lambda (|B_{r_2}| +\Phi_\lambda)^{\frac{{1-\theta}}{n}}
      \int_{B_{r_2}}
      \left[\phi\left(\frac{|\tau_{-h}(\zeta^2\tau_hu)|}{\lambda|h|^{ {\theta\gamma+1}}}\right)
+ \phi_{0,n}\left(\frac{|\tau_{-h}(\zeta^2\tau_hu)|}{\lambda|h|^{ {{\theta\gamma}}}(|B_{r_2}| +\Phi_\lambda)^{\frac{1}{n}}}\right)
     +
     \psi(|f|) + 1\right]\,\dx \\
     \nonumber
 &\le
    c|h|^{\theta(1+\gamma)}\lambda (|B_{r_2}| +\Phi_\lambda)^{\frac{{1-\theta}}{n}} 
    \int_{B_{r_2}}
\left[\phi\left(\frac{|\nabla(\zeta^2\tau_hu)|}{\lambda|h|^{ {\theta\gamma}}}\right)
+ \phi_{0,n}\left(\frac{{ 2}|\zeta^2\tau_hu|}{\lambda|h|^{ {\theta\gamma}}(|B_{r_2}| +\Phi_\lambda)^{\frac{1}{n}}}\right)   +
    \psi(|f|) + 1\right]\,\dx.
  \end{align}
  Note that, when making use of the inequalities \eqref{est-translation} and \eqref{est-diff-quotient} with $v=\zeta^2\tau_hu$, we have exploited the fact that the support of the latter function is contained in $B_{r_2}$.
 \\
  An application of the Sobolev-Poincar\'e
inequality~\eqref{SPnew1} to the function $\frac{\zeta^2\tau_h u}{\lambda|h|^{{{\theta}\gamma}}}$ and the $\Delta_2$ property of $\phi$ yield:
  \begin{equation}\label{july39}
\int_{B_{r_2}} \phi_{0,n}\left(\frac{{ 2}|\zeta^2\tau_hu|}{\lambda|h|^{ {\theta\gamma}}(|B_{r_2}| +\Phi_\lambda)^{\frac{1}{n}}}\right)\dx
    \le
c\int_{B_{r_2}}\ \phi\left(\frac{|\nabla(\zeta^2\tau_hu)|}{\lambda|h|^{ {\theta\gamma}}}\right)\dx
    =
    c\Phi_\lambda.
  \end{equation}
  From the inequalities \eqref{est-RHS-psi} and \eqref{july39} we deduce that
\begin{align}\label{bound-RHS}
    |\mathrm{I}_4|
    & \le
      c|h|^{\theta(1+\gamma)}\lambda|B_{r_2}|^{\frac{{1-\theta}}{n}}
      \Phi_\lambda
      + c|h|^{ {\theta(1+\gamma)}}\lambda|B_{r_2}|^{\frac{{1-\theta}}{n}}
  \int_{B_{r_2}}(\psi(|f|)+1)\,\dx
  \\ \nonumber & \quad
+ 
 c|h|^{ \theta(1+\gamma)}\lambda \Phi_\lambda^{\frac{{1-\theta}}{n}+1} 
 + c|h|^{ {\theta(1+\gamma)}}\lambda \Phi_\lambda^{\frac{{1-\theta}}{n}}
  \int_{B_{r_2}}(\psi(|f|)+1)\,\dx
  = \mathrm{J}+ \mathrm{JJ} + 
\mathrm{I}+\mathrm{II}.
\end{align}
Now, set
\begin{align}\label{july40}
  \lambda =\eps^{-\frac{1}{i_\phi-1}}\max\Big\{\Phi_1^\frac{{1-\theta}}{i_\phi(n+{1-\theta})-n},\Phi_1^\frac{{1-\theta}}{{s_\phi}(n+{1-\theta})-n}\Big\},
\end{align}
where $\eps\in(0,1)$ will be chosen later.
The following chain can be verified by
distinguishing the cases when
$\Phi_1\ge1$ or $\Phi_1<1$ and using Equation \eqref{homogeneity-phi}:
\begin{align*}
  \Phi_\lambda
  &=
    \int_{B_{r_2}}\phi\left(\frac{|\nabla(\zeta^2\tau_h u)|}{\lambda|h|^{ \theta \gamma}}\right)\dx\le
  \min\Big\{\Phi_1^{-\frac{{i_\phi}({1-\theta})}{{i_\phi}(n+{1-\theta})-n}},\Phi_1^{-\frac{{s_\phi}({1-\theta})}{{s_\phi}(n+{1-\theta})-n}}\Big\}
  \int_{B_{r_2}}\phi\left(\eps^{\frac{1}{{i_\phi}-1}}\frac{|\nabla(\zeta^2\tau_h
    u)|}{|h|^{ {\theta\gamma}}}\right)\dx\\\nonumber
  &\le
   \eps^{i_\phi'}\min\Big\{\Phi_1^{-\frac{{i_\phi}({1-\theta})}{{i_\phi}(n+{1-\theta})-n}},\Phi_1^{-\frac{{s_\phi}({1-\theta})}{{s_\phi}(n+{1-\theta})-n}}\Big\}
  \int_{B_{r_2}}\phi\left(\frac{|\nabla(\zeta^2\tau_h
    u)|}{|h|^{ {\theta\gamma}}}\right)\dx  =
    \eps^{i_\phi'}\min\Big\{\Phi_1^{\frac{n({i_\phi}-1)}{{i_\phi}(n+{1-\theta})-n}},\Phi_1^{\frac{n({s_\phi}-1)}{{s_\phi}(n+{1-\theta})-n}}\Big\}.
\end{align*}
Hence,
\begin{align*}
  \lambda\Phi_\lambda^{\frac{{1-\theta}}{n}}
  \le
  \eps^{-\frac{1}{{i_\phi}-1}+\frac{({1-\theta}) {i_\phi'}}{n}}\max\Big\{\Phi_1^{\frac{({1-\theta})
  {i_\phi}}{{i_\phi}(n+{1-\theta})-n}},\Phi_1^{\frac{({1-\theta}) {s_\phi}}{{s_\phi}(n+{1-\theta})-n}}\Big\}
  \le
  \eps^{-\frac{1}{{i_\phi}-1}}\max\Big\{\Phi_1^{\frac{ ({1-\theta})
  {i_\phi'}}{n+({1-\theta}) {i_\phi'}}},\Phi_1^{\frac{({1-\theta}) {s_\phi'}}{n+({1-\theta}) {s_\phi'}}}\Big\}.
\end{align*}
Therefore, 
\begin{align}\label{june5}
  \mathrm{I}
  =
  c|h|^{ {\theta(1+\gamma)}}\lambda \Phi_\lambda^{\frac{{1-\theta}+n}{n}}
  \le
  c|h|^{{\theta(1+\gamma)}}\eps \Phi_1
\end{align}
and
\begin{align}\label{june6}
  \mathrm{II}
  &=
  c|h|^{ {\theta(1+\gamma)}}\lambda \Phi_\lambda^{\frac{{1-\theta}}{n}}\int_{B_{r_2}}(\psi(|f|)+1)\,\dx\\
  & \nonumber \le
  c|h|^{ {\theta(1+\gamma)}}\eps^{-\frac{1}{{i_\phi}-1}}
    \max\Big\{\Phi_1^{\frac{({1-\theta}) {i_\phi'}}{n+({1-\theta})
    {i_\phi'}}},\Phi_1^{\frac{ ({1-\theta}) {s_\phi'}}{n+({1-\theta})
    {s_\phi'}}}\Big\}\int_{B_{r_2}}(\psi(|f|)+1)\,\dx\\
  & \nonumber \le
    c|h|^{ {\theta(1+\gamma)}}\left[\eps\Phi_1+c_\eps\bigg(\int_{B_{r_2}}(\psi(|f|)+1)\,\dx\bigg)^{1+\frac{({1-\theta}) {i_\phi'}}{n}}+c_\eps\bigg(\int_{B_{r_2}}(\psi(|f|)+1)\,\dx\bigg)^{1+\frac{({1-\theta}){s_\phi'}}{n}}\right],
\end{align}
where the last inequality 
rests on Young's inequality. 
Next, 
\begin{align}
    \label{june3}
    \mathrm{J} & = c |h|^{ {\theta(1+\gamma)}}\lambda|B_{r_2}|^{\frac{{1-\theta}}{n}}
     \Phi_\lambda \\ \nonumber &\leq c |h|^{{\theta(1+\gamma)}}|B_{r_2}|^{\frac{{1-\theta}}{n}} \eps^{-\frac{1}{{i_\phi}-1}}\max\Big\{\Phi_1^\frac{{1-\theta}}{{i_\phi}(n+{1-\theta})-n},\Phi_1^\frac{{1-\theta}}{{s_\phi}(n+{1-\theta})-n}\Big\}\eps^{{i_\phi'}}\min\Big\{\Phi_1^{\frac{n({i_\phi}-1)}{{i_\phi}(n+{1-\theta})-n}},\Phi_1^{\frac{n({s_\phi}-1)}{{s_\phi}(n+{1-\theta})-n}}\Big\}
     \\ \nonumber &\leq  c |h|^{ {\theta(1+\gamma)}}|B_{r_2}|^{\frac{{1-\theta}}{n}} \eps \max\Big\{\Phi_1^\frac{n({i_\phi}-1)+{1-\theta}}{n({i_\phi}-1)+({1-\theta}){i_\phi}},\Phi_1^\frac{n({s_\phi}-1)+{1-\theta}}{n({s_\phi}-1)+({1-\theta}) {s_\phi}}\Big\}
     \\ \nonumber
     &\leq  
     c |h|^{\theta(1+\gamma)} \eps \Big(\Phi_1 + |B_{r_2}|^{1+\frac{({1-\theta}) i_\phi'}{n}}+|B_{r_2}|^{1+\frac{({1-\theta}) s_\phi'}{n}}\Big).     
      \end{align}
Note that the last inequality follows via two Young's inequalities with exponents $\frac{n(i_\phi-1)+({1-\theta}) i_\phi}{n(i_\phi-1)+{1-\theta}}$ and $\frac{n(i_\phi-1)+({1-\theta}) i_\phi}{({1-\theta})(i_\phi-1)}$, and  with similar  exponents where $i_\phi$ is replaced with
 $s_\phi$.
 Moreover,
\begin{align*}
    \mathrm{JJ} & = 
    c|h|^{\theta(1+\gamma)}\lambda|B_{r_2}|^{\frac{{1-\theta}}{n}}
  \int_{B_{r_2}}(\psi(|f|)+1)\,\dx
  \\ \nonumber & \leq  c|h|^{ {\theta(1+\gamma)}} |B_{r_2}|^{\frac{{1-\theta}}{n}}\eps^{-\frac{1}{{i_\phi}-1}}\Big(\Phi_1^\frac{{1-\theta}}{{i_\phi}(n+{1-\theta})-n} +\Phi_1^\frac{{1-\theta}}{{s_\phi}(n+{1-\theta})-n}\Big)\int_{B_{r_2}}(\psi(|f|)+1)\,\dx.
  \end{align*}
  Young's inequality with exponents $\frac{n+{1-\theta}}{{1-\theta}}$ and $\frac{n+{1-\theta}}{n}$ yields:
  \begin{align*}
    |B_{r_2}|^{\frac{{1-\theta}}{n}}\int_{B_{r_2}}(\psi(|f|)+1)\,\dx
    &\le
    \tfrac{{1-\theta}}{n+{1-\theta}}|B_{r_2}|^{\frac{n+{1-\theta}}{n}}+\tfrac{n}{n+{1-\theta}}\bigg(\int_{B_{r_2}}\psi(|f|)+1\,\dx\bigg)^{\frac{n+{1-\theta}}{n}} \le
\bigg(\int_{B_{r_2}}(\psi(|f|)+1)\,\dx\bigg)^{\frac{n+{1-\theta}}{n}}.
  \end{align*}
 Coupling the last two inequalities tells us that
  \begin{align}\label{june4}
 \mathrm{JJ} & \leq c|h|^{{\theta(1+\gamma)}} \bigg[\eps \Phi_1+ c_\eps\bigg(
\int_{B_{r_2}}(\psi(|f|)+1)\,\dx\bigg)^{
 1+\frac{({1-\theta})i_\phi'}{n}
} +
c_\eps\bigg(\int_{B_{r_2}}
(\psi(|f|)+1)\,\dx\bigg)^{
 1+\frac{({1-\theta}) s_\phi'}{n}
}
\bigg].
\end{align}
From the inequalities \eqref{bound-RHS}--\eqref{june4} we deduce that
\begin{align}\label{bound-RHS-2}
    |\mathrm{I}_4|
& \le
     |h|^{ {\theta(1+\gamma)}}\bigg[ c\eps\Phi_1+c_\eps\bigg(\int_{B_{r_2}}(\psi(|f|)+1)\,\dx\bigg)^{1+\frac{(1-\theta) {i_\phi'}}{n}}+ c_\eps\bigg(\int_{B_{r_2}}(\psi(|f|)+1)\,\dx\bigg)^{1+\frac{(1-\theta) {s_\phi'}}{n}}\bigg].
\end{align}
Thanks to {Korn's inequality \eqref{korn-0} and }
\eqref{july36},  
\begin{align}\label{july41}
  \Phi_1
  &\le
  {
  c\int_{B_{r_2}} \phi\left(\frac{|\E (\zeta^2\tau_h u)|}{|h|^{ {\theta\gamma}}}\right)\dx
  }
 \le
  c'\int_{B_{r_2}} \phi\left(\frac{\zeta^2|\tau_h\E u|}{|h|^{ {\theta\gamma}}}\right)\dx
  +
  c' \int_{B_{r_2}}\phi\left(\frac{|\nabla\zeta|\,|\tau_hu|}{|h|^{ {\theta\gamma}}}\right)\dx.
\end{align}
We bound the last two integrals separately. As for the former,
owing to Equations \eqref{add-shift}, \eqref{homogeneity-phi-shifted},  \eqref{a-coercivity-split}, and to the equality ${ {\gamma}}\max\{{s_\phi},2\}= {1+\gamma}$, we have that
\begin{align*}
  \int_{B_{r_2}}\phi\left(\frac{\zeta^2|\tau_h{\E u}|}{|h|^{ {\theta\gamma}}}\right)
    \dx
    &
\le
c\int_{B_{r_2}}\zeta^2\phi_{|{\E u}|}\left(\frac{|\tau_h{\E u}|}{|h|^{ {\theta\gamma}}}\right)+\phi(|{\E u}|) \,\dx\nonumber\\
  & \le
    \frac{c}{|h|^{{ {\theta\gamma}}\max\{{s_\phi},2\}}}\int_{B_{r_2}}\zeta^2\phi_{|{\E u}|}(|\tau_h{\E u}|)\,\dx
    +
    c\int_{B_{r_2}}\phi(|{\E u}|) \,\dx\nonumber\\
  & \le
    \frac{c}{|h|^{ {\theta(1+\gamma)}}}\int_{B_{r_2}}\zeta^2|\tau_h[V({\E u})]|^2\,\dx
   +
   c\int_{B_{r_2}}\phi(|{\E u}|) \,\dx.
\end{align*}
Since $|h|\le  \min\{1,R\}$ and $|\nabla\zeta|\le\frac c{r_2-r_1}$, by
the property~\eqref{homogeneity-phi} and the inequality \eqref{est-diff-quotient} the latter integral on the right-hand side of \eqref{july41} can be estimated as follows:
\begin{align}\label{july50}
  \int_{B_{r_2}}  \phi\Big(\frac{|\nabla\zeta|\,|\tau_hu|}{|h|^{ {\theta\gamma}}}\Big)\dx
  &\le
   c{\left(\frac{1}{R^{\gamma s_\phi}}+\frac{1}{R^{\gamma i_\phi}}\right)\frac{R^{s_\phi}}{(r_2-r_1)^{s_\phi}}}\int_{B_{r_2}}  \phi\Big(\frac{|\tau_hu|}{|h|}\Big)\dx \\\nonumber
  &\le
  {c\left(\frac{1}{R^{\gamma s_\phi}}+\frac{1}{R^{\gamma i_\phi}}\right)\frac{R^{s_\phi}}{(r_2-r_1)^{s_\phi}}}\int_{B_{R_2}} \phi(|\nabla u|)\dx.
\end{align} 
Altogether, we deduce that
\begin{align*}
  \Phi_1
  \le
  \frac{c}{|h|^{{\theta(1+\gamma)}}}\int_{B_{{r_2}}}\zeta^2|\tau_h[V(\E u)]|^2\,\dx
   +
c\left( {1+}\left(\frac{1}{R^{\gamma s_\phi}}+\frac{1}{R^{\gamma i_\phi}}\right)\frac{R^{s_\phi}}{(r_2-r_1)^{s_\phi}}\right)\int_{B_{R_2}} \phi(|\nabla u|)\dx.
\end{align*}
Coupling the latter inequality with~\eqref{bound-RHS-2} results in the following bound:
\begin{align}\label{est-I4}
 |\mathrm{I}_4|& \le
   c\eps\int_{B_{r_2}} |\tau_h[V(\E u)]|^2\dx\
  + c_\eps
     |h|^{ {\theta(1+\gamma)}}\Bigg[\bigg( {1+ }{\left(\frac{1}{R^{\gamma s_\phi}}+\frac{1}{R^{\gamma i_\phi}}\right)\frac{R^{s_\phi}}{(r_2-r_1)^{s_\phi}}} \bigg)\int_{B_{R_2}} \phi(|\nabla u|)\dx\\\nonumber
     &
     \qquad\qquad\qquad\qquad\quad+\bigg(\int_{B_{r_2}}(\psi(|f|)+1)\,\dx\bigg)^{1+\frac{(1-\theta) i_\phi'}{n}}
     +
     \bigg(\int_{B_{r_2}}(\psi(|f|)+1)\,\dx\bigg)^{1+\frac{(1-\theta) s_\phi'}{n}}
 \Bigg].
\end{align}
We finally focus on $\mathrm{I}_5$.  Set
\begin{align}
    \label{Psi}
    \Psi_\lambda= \int_{B_{R_2}}\phi\bigg(\frac{|\nabla w|}{\lambda|h|^{{\theta\gamma}}}\bigg)\, \d x
\end{align}
for $\lambda >0$.
Similarly to \eqref{est-RHS-psi}, one has that
\begin{align}\label{july45}
|\mathrm{I}_5| &\leq \int_{B_{R_2}} |f| |\tau_{-h}w|\,\dx    
\\ 
\nonumber
    &=
      |h|^{{\theta(1+\gamma)}} \lambda (|B_{R_2}| +\Psi_\lambda)^{\frac{1-\theta}{n}}
\int_{B_{R_2}}\Big(\frac{|\tau_{-h}w|}{\lambda|h|^{ {\theta\gamma}+1}}\Big)^{\theta}\Big(\frac{|\tau_{-h}w|}{\lambda|h|^{{\theta\gamma}}(|B_{R_2}|  +\Psi_\lambda)^{\frac{1}{n}}}\Big)^{1-\theta}\,|f|\,\dx\\ \nonumber  
    &\le c
     |h|^{{\theta(1+\gamma)}}\lambda (|B_{R_2}| +\Psi_\lambda)^{\frac{{1-\theta}}{n}}
      \int_{B_{R_2}}
     \left[ \phi\left(\frac{|\tau_{-h}w|}{ \lambda|h|^{{\theta\gamma}+1}}\right)
  + \phi_{0,n}\left(\frac{|\tau_{-h}w|}{ \lambda|h|^{{\theta\gamma}}(|B_{R_2}| +\Psi_\lambda)^{\frac{1}{n}}}\right)
     +
     \psi(|f|) + 1\right]\,\dx \\
     \nonumber
 &\le
    c|h|^{{\theta(1+\gamma)}} \lambda (|B_{R_2}| +\Psi_\lambda)^{\frac{{1-\theta}}{n}} 
    \int_{B_{R_2}}\left[
\phi\left(\frac{|\nabla w|}{ \lambda|h|^{{\theta\gamma}}}\right)
+ \phi_{0,n}\left(\frac{2|w|}{\lambda |h|^{{\theta\gamma}}(|B_{R_2}| +\Psi_\lambda)^{\frac{1}{n}}}\right)   +
    \psi(|f|) + 1\right]\,\dx.
\end{align}
The Sobolev-Poincar\'e inequality \eqref{SPnew1} now tells us that
\begin{align}
    \label{july46}
    \int_{B_{R_2}} \phi_{0,n}\left(\frac{2|w|}{ \lambda|h|^{{\theta\gamma}}(|B_{R_2}| +\Psi_\lambda)^{\frac{1}{n}}}\right)\, \d x \leq c \int_{B_{R_2}}\phi\bigg(\frac{|\nabla w|}{\lambda|h|^{{\theta\gamma}}}\bigg)\, \d x= c \Psi_\lambda.
\end{align}
Coupling \eqref{july45} with \eqref{july46} implies that
\begin{align}
    \label{july47}
    |\mathrm{I}_5| & \leq  c|h|^{{\theta(1+\gamma)}} \lambda |B_{R_2}|^{\frac{{1-\theta}}{n}} \bigg(\Psi_\lambda+ \int_{B_{R_2}}(\psi(|f|) +1 )\, \d x\bigg)
    +
    c|h|^{{\theta(1+\gamma)}} \lambda\Psi_\lambda^{\frac{{1-\theta}}{n}}\bigg(\Psi_\lambda+ \int_{B_{R_2}}(\psi(|f|) +1) \, \d x\bigg).
\end{align}
Starting from \eqref{july47} instead of \eqref{bound-RHS}, the same argument which yields
\eqref{bound-RHS-2}, with $\varepsilon =1$, entails that
\begin{align}
    \label{july48}
     |\mathrm{I}_5| & \leq
     c|h|^{{\theta(1+\gamma)}}\bigg[ \Psi_1+\bigg(\int_{B_{R_2}}(\psi(|f|)+1)\,\dx\bigg)^{1+\frac{({1-\theta}) {i_\phi'}}{n}}+  \bigg(\int_{B_{R_2}}(\psi(|f|)+1)\,\dx\bigg)^{1+\frac{({1-\theta}) {s_\phi'}}{n}}\bigg].
\end{align}
On the other hand,  thanks to \eqref{Bogovski-gradient} 
and \eqref{july50} one has that
\begin{align}
    \label{july49}
    \Psi_1& = \int_{B_{R_2}}\phi\bigg(\frac{|\nabla w|}{|h|^{{\theta\gamma}}}\bigg)\, \d x=
\int_{B_{R_2}}\phi\bigg(\frac{|\nabla \mathcal Bg|}{|h|^{{\theta\gamma}}}\bigg)\, \d x \leq c\int_{B_{R_2}}\phi\bigg(\frac{|g|}{|h|^{{\theta\gamma}}}\bigg)\, \d x
\\ \nonumber & = c \int_{B_{R_2}}\phi\bigg(\frac{|2\zeta\nabla\zeta\cdot\tau_hu|}{|h|^{{\theta\gamma}}}\bigg)\, \d x
\leq  c\int_{B_{r_2}}  \phi\Big(\frac{|\nabla\zeta|\,|\tau_hu|}{|h|^{{\theta\gamma}}}\Big)\dx\\\nonumber
&\le
{c\left(\frac{1}{R^{\gamma s_\phi}}+\frac{1}{R^{\gamma i_\phi}}\right)\frac{R^{s_\phi}}{(r_2-r_1)^{s_\phi}}}\int_{B_{R_2}} \phi(|\nabla u|)\,\dx.
\end{align} 
The inequalities \eqref{july48} and \eqref{july49} ensure that
\begin{align}
    \label{july51}
     |\mathrm{I}_5|  \leq
     |h|^{{\theta(1+\gamma)}}&\bigg[ {c\left(\frac{1}{R^{\gamma s_\phi}}+\frac{1}{R^{\gamma i_\phi}}\right)\frac{R^{s_\phi}}{(r_2-r_1)^{s_\phi}}}\int_{B_{R_2}} \phi(|\nabla u|)\,\dx
\\ \nonumber & \quad +\bigg(\int_{B_{R_2}}(\psi(|f|)+1)\,\dx\bigg)^{1+\frac{({1-\theta}) {i_\phi'}}{n}}+  \bigg(\int_{B_{R_2}}(\psi(|f|)+1)\,\dx\bigg)^{1+\frac{({1-\theta}){s_\phi'}}{n}}\bigg].
\end{align}
 Combining Equations \eqref{def-I-1-5}, \eqref{lower-bound}, \eqref{I-1-bound},
  \eqref{est-I2}, \eqref{est-I3}, \eqref{est-I4}   with $\varepsilon =\kappa$, and \eqref{july51}
  yields:
  \begin{align}\label{july53}
    \int_{B_{r_1}}&|\tau_h V(\E u)|^2\,\dx
    \le
      c\kappa\int_{B_{r_2}}
      |\tau_hV(\E u)|^2\, \dx
      +
     {
     c \kappa |h|^{2\sigma} \left(\sup_{{|h|\leq h_0}}\int_{B_{R_2}}\frac{|\tau_hV(\E u)|^2}{|h|^{2\sigma}}\,\dx\right)}\\\nonumber 
  & \quad +  \frac{c_{\kappa}|h|^{2\sigma}}{(r_2-r_1)^{ {2\sigma}}}
  \left(1+ {R^{\sigma {m_\sigma}-n}}\int_{B_{R}}k_\sigma^{{m_\sigma}}\,\dx\right)^{ {\frac{2}{{m_\sigma}-\frac{n}{\beta}}}}\int_{B_{R}}\phi(|\nabla u|)\,\dx\\  \nonumber
    &\quad
      +
      c_\kappa|h|^{\theta\min\{{s_\phi'},2\}}
      \Bigg[\left(  {1+}{\left(\frac{1}{R^{\gamma s_\phi}}+\frac{1}{R^{\gamma i_\phi}}\right)\frac{R^{s_\phi}}{(r_2-r_1)^{s_\phi}}}
      +
      \frac{1}{(r_2-r_1)^{i_\phi'}}
      +
      \frac{1}{(r_2-r_1)^{s_\phi'}}
      \right)
       \int_{B_{R}} \phi(|\nabla u|)\,\dx
       \\ \nonumber & \qquad \qquad \ \qquad \qquad \qquad+ \bigg(\int_{B_{R}}(\psi(|f|)+1)\,\dx\bigg)^{1+\frac{({1-\theta}) i_\phi'}{n}}+\bigg(\int_{B_{R}}(\psi(|f|)+1)\,\dx\bigg)^{1+\frac{({1-\theta}) s_\phi'}{n}}\Bigg],\\\nonumber
  \end{align}
for  $\kappa\in(0,1]$,   $r_1,r_2$ with   $R_1\le r_1<r_2\le \frac12(R_1+R_2)
\leq R$, and $|h|\le h_0$.
 Now, 
  choose $\kappa=\frac1{2c}$  and apply the iteration Lemma~\ref{lem:Giaq} to obtain:
  \begin{align}\label{est-tau-h-V-1}
    \int_{B_{R_1}}&|\tau_h V(\E u)|^2\,\dx
    \le      
     {
     \frac 12 |h|^{2\sigma} \left(\sup_{{|h|\leq h_0}}\int_{B_{R_2}}\frac{|\tau_hV(\E u)|^2}{|h|^{2\sigma}}\,\dx\right)}\\\nonumber 
  & \quad +  \frac{c|h|^{2\sigma}}{(R_2-R_1)^{ {2\sigma}}}
  \left(1+ {R^{\sigma {m_\sigma}-n}}\int_{B_{R}}k_\sigma^{{m_\sigma}}\,\dx\right)^{{\frac{2}{{m_\sigma}-\frac{n}{\beta}}}}\int_{B_{R}}\phi(|\nabla u|)\,\dx\\  \nonumber
    &\quad
      +
      c|h|^{\theta\min\{{s_\phi'},2\}}
      \Bigg[\left(1+ {\Big(\frac{1}{R^{\gamma s_\phi}}+\frac{1}{R^{\gamma i_\phi}}}\Big)\frac{R^{s_\phi}}{(R_2-R_1)^{s_\phi}}
      +
      \frac{1}{(R_2-R_1)^{i_\phi'}}
      {
      +\frac{1}{(R_2-R_1)^{s_\phi'}}
      }\right)
       \int_{B_{R}} \phi(|\nabla u|)\,\dx
       \\ \nonumber & \qquad \qquad \ \qquad \qquad \qquad 
       +
       \bigg(\int_{B_{R}}(\psi(|f|)+1)\,\dx\bigg)^{1+\frac{({1-\theta}) i_\phi'}{n}}
       +
       {
       \bigg(\int_{B_{R}}(\psi(|f|)+1)\,\dx\bigg)^{1+\frac{({1-\theta}) s_\phi'}{n}}
       }\Bigg]\\\nonumber
  \end{align}
  for  {  $0< R_1<R_2\le R$}
   and   $|h|\le h_0$.  
 Dividing through by $|h|^{2\sigma}$ in the inequality \eqref{est-tau-h-V-1}, taking the supremum  over $h\in (0, h_0]$, { and using the fact that $h_0\leq 1$} tell us that
\begin{align}\label{est-tau-h-V-2}
 \sup_{0<|h|\leq h_0}&\int_{B_{R_1}}\frac{|\tau_hV(\E u)|^2}{|h|^{2\sigma}}\,\dx
    \le      
     {
     \frac 12  \left(\sup_{0<|h|\leq h_0}\int_{B_{R_2}}\frac{|\tau_hV(\E u)|^2}{|h|^{2\sigma}}\,\dx\right)}\\\nonumber 
  & \quad +  \frac{c}{(R_2-R_1)^{ {2\sigma}}}\left(1+ {R^{\sigma {m_\sigma}-n}}\int_{B_{R}}k_\sigma^{{m_\sigma}}\,\dx\right)^{\frac{2}{m_\sigma-\frac{n}{\beta}}}\int_{B_{R}}\phi(|\nabla u|)\,\dx\\  \nonumber
    &\quad
      +
      c 
      \Bigg[\left(1+\Big( {\frac{1}{R^{\gamma s_\phi}}+\frac{1}{R^{\gamma i_\phi}}}\Big)\frac{R^{s\phi}}{(R_2-R_1)^{s_\phi}}
      +
      \frac{1}{(R_2-R_1)^{i_\phi'}}
      +
      \frac{1}{(R_2-R_1)^{s_\phi'}}
      \right)
       \int_{B_{R}} \phi(|\nabla u|)\,\dx
       \\ \nonumber & \qquad \qquad + \bigg(\int_{B_{R}}(\psi(|f|)+1)\,\dx\bigg)^{1+\frac{({1-\theta}) i_\phi'}{n}}
       +
       {
       \bigg(\int_{B_{R}}(\psi(|f|)+1)\,\dx\bigg)^{1+\frac{({1-\theta}) s_\phi'}{n}}
       }
       \Bigg].
  \end{align}
 The inequality \eqref{est-tau-h-V-2} continues to hold if the supremum over $h$ is extended to the interval $(0,{\frac{R}{2}}]$. 
Indeed, 
if $h_0<|h|\le\frac{R}{2}$, then
from \eqref{est-translation}, the very definition of the function $V$, and \eqref{indices1} we deduce that
\begin{align}\label{est:large-h}
  \int_{B_{R_1}}\frac{|\tau_hV(\E u)|^2}{|h|^{2\sigma}}\,\dx
  \le
  \frac{c}{h_0^{2\sigma}}\int_{B_{\frac{3}{2}R}}|V(\E u)|^2\,\dx
  \le
  c'\bigg(\frac{1}{(R_2-R_1)^{2\sigma}}+1\bigg)\int_{B_{\frac{3}{2}R}}\phi(|\E u|)\,\dx. 
\end{align}
Consequently,
\begin{align}\label{est-tau-h-V-2''}
 \sup_{0<|h|\leq {\frac{R}{2}}}\int_{B_{R_1}}\frac{|\tau_hV(\E u)|^2}{|h|^{2\sigma}}\,\dx
    &\le      
     {
     \frac 12  \left(\sup_{0<|h|\leq \frac{R}{2}}\int_{B_{R_2}}\frac{|\tau_hV(\E u)|^2}{|h|^{2\sigma}}\,\dx\right)}\\\nonumber 
  & \quad +  {\frac{c  }{(R_2-R_1)^{ {2\sigma}}}\left(1+ {R^{\sigma {m_\sigma}-n}}\int_{B_{R}}k_\sigma^{{m_\sigma}}\,\dx\right)^{{\frac{2}{{m_\sigma}-\frac{n}{\beta}}}}\int_{B_{{\frac{3}{2}R}}}\phi(|\nabla u|)\,\dx}\\  \nonumber
    &\quad
      +
      c 
      \Bigg[\left(1+\Big({\frac{1}{R^{\gamma s_\phi}}+\frac{1}{R^{\gamma i_\phi}}}\Big)\frac{R^{s_\phi}}{(R_2-R_1)^{s_\phi}}
      +
      \frac{1}{(R_2-R_1)^{i_\phi'}}
      +
      {
      \frac{1}{(R_2-R_1)^{s_\phi'}}
      }\right)
       \int_{B_{\frac{3}{2}R}} \phi(|\nabla u|)\,\dx
       \\ \nonumber & \quad+ \bigg(\int_{B_{R}}(\psi(|f|)+1)\,\dx\bigg)^{1+\frac{({1-\theta}) i_\phi'}{n}}
       +
       {
       \bigg(\int_{B_{R}}(\psi(|f|)+1)\,\dx\bigg)^{1+\frac{({1-\theta}) s_\phi'}{n}}
       }\Bigg].\\\nonumber
  \end{align}
 Hence, an application of the iteration Lemma~\ref{lem:Giaq} again and the  inequality \eqref{oct3} enable us to infer that
  \begin{align}\label{est-tau-h-V-3}
    \sup_{0<|h|\leq  {\frac{1}{2}R}}\int_{B_{ {\frac 34R}}}\frac{|\tau_hV(\E u)|^2}{|h|^{2\sigma}}\,\dx
    &\le \frac{c }{R^{ {2\sigma}}}\left(1+ {R^{\sigma {m_\sigma}-n}}\int_{B_{ {R}}}k_\sigma^{{m_\sigma}}\,\dx\right)^{{\frac{2}{{m_\sigma}-\frac{n}{\beta}}}} \int_{B_{\frac{3}{2}R}}\phi(|\nabla u|)\,\dx\\\nonumber 
      &\quad+
  c
    \Bigg[\left({1+}{\frac{1}{R^{\gamma s_\phi}}+\frac{1}{R^{\gamma i_\phi}}}+\frac{1}{R^{i_\phi'}}+
    \frac{1}{R^{s_\phi'}}
    \right)\int_{B_{\frac{3}{2}R}}\phi(|\nabla u|)\,\dx
    \\\nonumber
    &\quad+{\bigg(\int_{B_R}\psi(|f|)+1\,\dx\bigg)^{1+\frac{({1-\theta}) i_\phi'}{n}}
    +
    \bigg(\int_{B_{R}}(\psi(|f|)+1)\,\dx\bigg)^{1+\frac{({1-\theta}) s_\phi'}{n}}
       }\Bigg] \nonumber
       \\  \nonumber & \le {c_R}
    \left(1+
       \bigg(\int_{B_{ \frac 54 R}}k^{{m}}(x)\,\dx\bigg)^{{\frac{2m_\sigma}{{m(m_\sigma}-\frac{n}{\beta})}}}  \right)  \int_{B_{\frac{3}{2}R}}\phi(|\nabla u|)\,\dx\\\nonumber 
      &\quad+
   c
{\bigg(\int_{B_R}(\psi(|f|)+1)\,\dx\bigg)^{1+\frac{({1-\theta}) i_\phi'}{n}}
+c\bigg(\int_{B_{R}}(\psi(|f|)+1)\,\dx\bigg)^{1+\frac{({1-\theta}) s_\phi'}{n}}
       }.
  \end{align}
  A chain analogous to \eqref{est:large-h} then shows that the inequality \eqref{est-tau-h-V-3} also holds if the supremum on the leftmost side is extended to $0<|h|\leq \frac 34R$. Namely,
  \begin{align}\label{est-tau-h-V-4}
    \sup_{0<|h|\leq \frac 34R}\int_{B_{ {\frac 34R}}}\frac{|\tau_hV(\E u)|^2}{|h|^{2\sigma}}\,\dx
     & \le {c_R}
    \left(1+
       \bigg(\int_{B_{ {{\frac 54 R}}}}k^{{m}}(x)\,\dx\bigg)^{{\frac{2m_\sigma}{{m(m_\sigma}-\frac{n}{\beta})}}}  \right)  \int_{B_{\frac{3}{2}R}}\phi(|\nabla u|)\,\dx\\\nonumber 
      &\quad+
   c
{\bigg(\int_{B_R}(\psi(|f|)+1)\,\dx\bigg)^{1+\frac{({1-\theta}) i_\phi'}{n}}
+c\bigg(\int_{B_{R}}(\psi(|f|)+1)\,\dx\bigg)^{1+\frac{({1-\theta}) s_\phi'}{n}}
       }.
  \end{align}
  \end{proof}
 
  \section{Proof of Theorem \ref{main}, Part (i), completed}\label{sec:proof2}
With the a priori estimates established in the previous section at our disposal, we are now in a position to accomplish the proof of Part (i) of Theorem \ref{main}, via a regularization argument.

\begin{proof}[Proof of Theorem \ref{main}, Part (i), completed]
Let $\{\varrho_\ell\}$ be a standard family of smooth mollifiers in $\rn$, for $\ell\in \mathbb N$. Define the  functions $a_\ell : \Omega\times \mathbb R^{n\times n}_{\rm sym} \to \mathbb R^{n\times n}_{\rm sym}$ as 
$${a}_\ell(x,P)=\int_{\Omega}\varrho_\ell(x-y)a(y,P)\,\d y \quad \text{for $(x,P)\in \Omega\times \mathbb R^{n\times n}_{\rm sym}$,}$$
and $f_\ell : \Omega \to \rn$ as
$${f}_\ell(x)=\int_{\Omega}\varrho_\ell(x-y)f(y)\,\d y \quad \text{for $x\in \Omega$.}$$
Let $u \in W^{1,\phi}_{\rm loc}(\Omega, \real^n)$ be a local weak solution to the system \eqref{equa}. Given a ball $B_R$ such that $B_{\frac 32R} \Subset \Omega$,  let $u_\ell \in u+ W^{1,\phi}_0(B_{\frac32 R}, \real^n)$ be such that $\dive u_\ell =0$ and
  \begin{equation}
    \label{equa-div-free-eps}
    \int_{B_{\frac{3}{2}R}}  a_\ell(x,\E u_\ell) \cdot \E\varphi \,\dx
    =
    \int_{B_{\frac{3}{2}R}} f_\ell\cdot\varphi\,\dx
  \end{equation}
  for every $\varphi\in W^{1,\phi}_0(B_{\frac{3}{2}R},\R^n)$ such that $\dive\varphi=0$.   The existence  of such a function $u_\ell$ can be established via a classical theorem  on pseudomonotone operators -- see e.g. \cite[Theorem 27.A]{zeidler}. Some details on its application in the case when $f_\ell=0$ can be found in \cite{DiKa}. 
	\\   Notice that the assumptions \eqref{ip1}--\eqref{ip2b} on the function $a$ are transferred into parallel properties on each function $a_\ell$, namely:
	   \begin{equation}\label{B1}
	( a_\ell(x,P)-a_\ell(x,Q))\cdot (P-Q) \ge \nu\phi''(|P|+|Q|)|P-Q|^2
\end{equation}
	   
	  \begin{equation}\label{B2}
	| a_\ell(x,P)-a_\ell(x,Q)|\le L \phi''(|P|+|Q|)|P-Q|
\end{equation}

\begin{equation}\label{B2b}
	| a_\ell(x,P)|\le L \phi'(|P|)
\end{equation}
for a.e. $x\in B_{\frac{3}{2}R}$ and every $P, Q \in \mathbb R^{n\times n}_{\rm sym}$. In the other hand, the hypothesis
\eqref{ipholder} turns into:
\begin{equation}\label{B3}
	| a_\ell(x,P)-a_\ell(y,P)|\le |x-y|^{{\alpha}}(k_\ell(x)+k_\ell(y))\phi'(|P|)
\end{equation}
for a.e. $x, y\in B_{\frac{3}{2}R}$ and
every $P \in \mathbb R^{n\times n}_{\rm sym}$,
	 where	 the function $k_\ell : \Omega \to [0, \infty)$ is given by
     $${k}_\ell(x)=\int_{\Omega}\varrho_\ell(x-y)k(y)\,\d y \quad \text{for $x\in \Omega$.}$$
Since  $k_\ell \in L^\infty_{\rm loc} (\Omega)$ and $f_\ell \in   L^\infty_{\rm loc} (\Omega)$, the estimate \eqref{est-tau-h-V-4} can be applied with $u$, $k$ and $f$ replaced with $u_\ell$, $k_\ell$ and $f_\ell$. This results in:
\begin{align}\label{est-tau-h-Veps}
 \sup_{0<|h|\leq  {
\frac34R}}\int_{B_{ {\frac 34R}}}\frac{|\tau_hV(\E u_\ell)|^2}{|h|^{2\sigma}}\,\dx
     & \le {c_R}
    \bigg(1+
       \bigg(\int_{B_{ {{\frac 54 R}}}}{k}_\ell^{{m}}(x)\,\dx\bigg)^{{\frac{2m_\sigma}{{m(m_\sigma}-\frac{n}{\beta})}}}  \bigg)  \int_{B_{\frac32R}}\phi(|\nabla u_\ell|)\,\dx\\\nonumber 
      &\quad+
   c
{\bigg(\int_{B_R}(\psi(|f_\ell|)+1)\,\dx\bigg)^{1+\frac{({1-\theta}) i_\phi'}{n}}
+c\bigg(\int_{B_{R}}(\psi(|f_\ell|)+1)\,\dx\bigg)^{1+\frac{({1-\theta}) s_\phi'}{n}}
       }.
  \end{align}
Observe that, in principle, Equation \eqref{est-tau-h-V-4} can only be applied with $R$ replaced with a strictly smaller number.   The inequality \eqref{est-tau-h-Veps} then follows thanks to the monotone convergence theorem.
By Jensen's inequality, 
\begin{align}
    \label{july66}
    \int_{B_{ {{\frac 54R}}}}k^{m}_\ell\,\dx \leq \int_{B_{ {\frac 32R}}}k^{m}\,\dx \quad \text{and} \quad  \int_{B_{ {{R}}}}\psi(|f_\ell|) \,\dx \leq \int_{B_{\frac 32R }}\psi(|f|) \,\dx,
\end{align}
for sufficiently large $\ell\in\N$.
Thus, from \eqref{est-tau-h-Veps} we obtain:
{\begin{align}\label{est-tau-h-Veps'}
 \sup_{0<|h|\leq \frac 34R}\int_{B_{ {\frac 34R}}}\frac{|\tau_hV(\E u_\ell)|^2}{|h|^{2\sigma}}\,\dx
     & \le {c_R}
    \bigg(1+
       \bigg(\int_{B_{ \frac 32 R}}{k}^{{m}}(x)\,\dx\bigg)^{{\frac{2m_\sigma}{{m(m_\sigma}-\frac{n}{\beta})}}}  \bigg)  \int_{B_{\frac32R}}\phi(|\nabla u_\ell|)\,\dx\\\nonumber 
      &\quad+
   c
{\bigg(\int_{B_{\frac32 R}}\psi(|f|)+1\,\dx\bigg)^{1+\frac{({1-\theta}) i_\phi'}{n}}
+c\bigg(\int_{B_{\frac 32R}}\psi(|f|)+1\,\dx\bigg)^{1+\frac{({1-\theta}) s_\phi'}{n}}
       }.
  \end{align}}
Our purpose is now to prove that 
\begin{align}
    \label{july70}
\lim_{\ell\to \infty}\int_{B_{\frac32R}}\phi(|\E u_\ell-\E u|)\,\dx=0.
\end{align}
The properties of $u_\ell$ legitimate the  use of the test function $u_\ell-u$ in equations \eqref{equa-div-free-eps} and \eqref{equa}.  
By the assumption \eqref{B1}, we obtain that
\begin{align}\label{I1+I2}
	\nu\int_{B_{\frac{3}{2}R}}&|\E u_\ell-\E u|^2\phi''(|\E u_\ell|+|\E u |)\,\d x\le \int_{B_{\frac{3}{2}R}}( a_\ell(x,\E u_\ell)-a_\ell(x,\E u))\cdot (\E u_\ell-\E u)\, \d x \\
	 \nonumber &=\int_{B_{\frac{3}{2}R}}( a_\ell(x,\E u_\ell)-a(x,\E u)) \cdot (\E u_\ell-\E u)\, \d x+\int_{B_{\frac{3}{2}R}}( a(x,\E u)-a_\ell(x,\E u))\cdot (\E u_\ell-\E u)\, \d x\\
	 \nonumber &=\int_{B_{\frac{3}{2}R}} (f_\ell-f)\cdot ( u_\ell- u)\, \d x +\int_{B_{\frac{3}{2}R}} (a(x,\E u)-a_\ell(x,\E u))\cdot (\E u_\ell-\E u)\, \d x \\  
	&  \nonumber \le \int_{B_{\frac{3}{2}R}}| a(x,\E u)-a_\ell(x,\E u)|| \E u_\ell-\E u|\, \d x+\int_{B_{\frac{3}{2}R}} |f_\ell-f|\, | u_\ell- u|\, \d x = I_1^{k}+I_2^{k}.
	\end{align}
{ Thanks to \eqref{add-shift} and \eqref{shift4'}, the following chain holds:
\begin{align}
    \label{july75}
    \int_{B_{\frac{3}{2}R}}\phi (|\E u_\ell-\E u|)\, \d x & \leq  c_\delta \int_{B_{\frac{3}{2}R}}\phi_{|\E u|}(|\E u_\ell-\E u|)\, \d x + \delta \int_{B_{\frac{3}{2}R}}\phi (|\E u|)\, \d x 
    \\ \nonumber & \leq 
    c_\delta \int_{B_{\frac{3}{2}R}}|\E u_\ell-\E u|^2\phi''(|\E u_\ell|+|\E u |)\,\d x + \delta \int_{B_{\frac{3}{2}R}}\phi (|\E u|)\, \d x \\ \nonumber & \leq \frac {c_\delta}\nu (I_1^{k}+I_2^{k})+\delta \int_{B_{\frac{3}{2}R}}\phi (|\E u|)\, \d x.
\end{align}
We begin by estimating $I_2^{k}$.   Given $\lambda>0$,   set
$$\Theta_\lambda=\int_{B_{\frac{3}{2}R}} \phi\left(\frac{| \nabla u_\ell- \nabla u|}{\lambda}\right)\, \dx.$$
Owing to \eqref{july62} and \eqref{june10}, there exists a constant $c>0$ such that
\begin{align}
    \label{july63}
\widetilde{ \phi_{0,n}}(2ct) \leq \psi(t) +1 \quad \text{for $t \geq 0$.}
\end{align}
 Hence,  by Young's inequality,
\begin{align}\label{J2'}
        I_2^{k} & \le \frac{\lambda}{c}(|B_{\frac{3}{2}R}|+\Theta_\lambda)^{\frac{1}{n}}
        \left[ \int_{B_{\frac{3}{2}R}}\phi_{0,n}\left(\frac {| u_\ell- u|}{\lambda (|B_{\frac{3}{2}R}|+\Theta_\lambda)^{\frac{1}{n}}}\right)\,\dx+\int_{B_{\frac{3}{2}R}} \widetilde{ \phi_{0,n}}(c|f_\ell-f|)\,\dx\right].
        \end{align} 
The Poincar\'e-Sobolev inequality \eqref{SPnew1}
and the $\Delta_2$-property of $\phi$ ensure that there exists a constant $c$ such that
\begin{equation}\label{july64}
\int_{B_{\frac{3}{2}R}} \phi_{0,n}\left(\frac {| u_\ell- u|}{\lambda (|B_{\frac{3}{2}R}|+\Theta_\lambda)^{\frac{1}{n}}}\right)\,\dx
    \le
c\int_{B_{\frac{3}{2}R}} \phi\left(\frac{| \nabla(u_\ell- u)|}{\lambda}\right)\, \dx
    =
    c\, \Theta_\lambda.
  \end{equation}
Coupling \eqref{J2'} with \eqref{july64} yields:
\begin{align}\label{J2}
        I_2^{k}   &\leq c\lambda |B_{\frac{3}{2}R}|^{\frac{1}{n}} 
\Theta_\lambda+
c\lambda |B_{\frac{3}{2}R}|^{\frac{1}{n}}
\int_{B_{\frac{3}{2}R}} \widetilde{ \phi_{0,n}}(c|f_\ell-f|)\,\dx
     +c\lambda \Theta_\lambda^{\frac{1}{n}+1}  
     +c\lambda \Theta_\lambda^{\frac{1}{n}} 
     \int_{B_{\frac{3}{2}R}}  \widetilde{ \phi_{0,n}}(c|f_\ell-f|)\,\dx 
     \\ \nonumber & =\textrm{J}+ \textrm{JJ}+\textrm{I}+\textrm{II}.
    \end{align}
Given $\varepsilon \in (0, 1)$, 
one can choose
$$\lambda =\eps^{-\frac{1}{i_\phi-1}}\max\Big\{\Theta_1^\frac{1}{i_\phi(n+1)-n},\Theta_1^\frac{1}{{s_\phi}(n+1)-n}\Big\},$$
and
bound the quantities $\mathrm
{J}$, $\mathrm{I}$ and $\mathrm{II}$  similarly to the  parallel quantities appearing in \eqref{bound-RHS} with $\theta=0$. The quantity $\mathrm
{JJ}$ can be estimated as follows:
\begin{align}
    \label{oct11}
  \mathrm
{JJ}& =   c\lambda |B_{\frac{3}{2}R}|^{\frac{1}{n}}
\int_{B_{\frac{3}{2}R}} \widetilde{\phi_{0,n}}(c|f_\ell-f|)\,\dx \leq c
|B_{\frac{3}{2}R}|^{\frac{1}{n}}\eps^{-\frac{1}{{i_\phi}-1}}\Big(\Theta_1^\frac{1}{{i_\phi}(n+1)-n} +\Theta_1^\frac{1}{{s_\phi}(n+1)-n}\Big)\int_{B_{\frac{3}{2}R}} \widetilde{\phi_{0,n}}(c|f_\ell-f|)\,\dx 
\\ \nonumber &
\leq c \varepsilon \Theta_1 + c_\varepsilon \bigg(|B_{\frac{3}{2}R}|^{\frac 1n}\int_{B_{\frac{3}{2}R}} \widetilde{\phi_{0,n}}(c|f_\ell-f|)\,\dx\bigg)^{(1+\frac{{i_\phi'}}{n}) {\frac{n}{n+1}}} + c_\varepsilon \bigg(|B_{\frac{3}{2}R}|^{\frac 1n}\int_{B_{\frac{3}{2}R}} \widetilde{\phi_{0,n}}(c|f_\ell-f|)\,\dx\bigg)^{(1+\frac{ {s_\phi'}}{n}) {\frac{n}{n+1}}}\\\nonumber
& {
\le
c\varepsilon\Theta_1
+
\varepsilon \Big(|B_{\frac{3}{2}R}|^{1+\frac{i_\phi'}{n}}
+|B_{\frac{3}{2}R}|^{1+\frac{s_\phi'}{n}}\Big)
+
c_\eps' \bigg(\int_{B_{\frac{3}{2}R}} \widetilde{\phi_{0,n}}(c|f_\ell-f|)\,\dx\bigg)^{1+\frac{{i_\phi'}}{n}}
+
c_\eps '\bigg(\int_{B_{\frac{3}{2}R}} \widetilde{\phi_{0,n}}(c|f_\ell-f|)\,\dx\bigg)^{1+\frac{{i_\phi'}}{n}}
}.
\end{align} 
So doing, and making also use of the Korn inequality \eqref{korn-0} to bound $\Theta_1$ by 
$\int_{B_{\frac{3}{2}R}}\phi (|\E  u_\ell-\E u|)\, \d x $ times a constant, one can conclude that
\begin{align}\label{july77}
    |I_2^k|
    & \le
      c\eps \int_{B_{\frac{3}{2}R}}\phi (|\E  u_\ell-\E u|)\, \d x
     {
     +c\eps\big(|B_{\frac{3}{2}R}|^{1+\frac{{i_\phi'}}{n}}+|B_{\frac{3}{2}R}|^{1+\frac{{s_\phi'}}{n}}\big)
     }
     \\ \nonumber & \quad\quad 
     + 
     c_\eps\bigg( \int_{B_{\frac{3}{2}R}}\widetilde{\phi_{0,n}}(c|f_\ell-f|)\,\dx\bigg)^{1+\frac{ {i_\phi'}}{n}}
     +
     c_\eps\bigg( \int_{B_{\frac{3}{2}R}} \widetilde{\phi_{0,n}}(c|f_\ell-f|)\,\dx\bigg)^{1+\frac{{s_\phi'}}{n}}.
     \end{align}
     Choosing $\varepsilon \leq \frac{\nu}{2cc_\delta}$  and coupling \eqref{july75} with \eqref{july77} yield:
\begin{align}
    \label{july78}
     \int_{B_{\frac{3}{2}R}}&\phi (|\E u_\ell-\E u|)\, \d x  \leq \frac{2c_\delta}{\nu}I_1^k+
     2\delta \int_{B_{\frac{3}{2}R}}\phi (|\E u|)\, \d x
    {
     +c\eps\big(|B_{\frac{3}{2}R}|^{1+\frac{{i_\phi'}}{n}}+|B_{\frac{3}{2}R}|^{1+\frac{ {s_\phi'}}{n}}\big)
     }
     \\ \nonumber & \quad + 
     c_\eps\bigg( \int_{B_{\frac{3}{2}R}}\widetilde{\phi_{0,n}}(c|f_\varepsilon-f|)\,\dx\bigg)^{1+\frac{ {i_\phi'}}{n}}
     +
     c_\eps\bigg( \int_{B_{\frac{3}{2}R}} \widetilde{\phi_{0,n}}(c|f_\ell-f|)\,\dx\bigg)^{1+\frac{{s_\phi'}}{n}}.
\end{align}
Next, we have:
\begin{align}
    \label{july86}
   I_1^k & = \int_{B_{\frac{3}{2}R}}|a_\ell(x,\E u)- a(x,\E u)|| \E u_\ell-\E u|\, \d x 
   \\ \nonumber & \leq c_\kappa \int_{B_{\frac{3}{2}R}}\widetilde \phi (|a_\ell(x,\E u)- a(x,\E u)|)\, \dx + \kappa\int_{B_{\frac{3}{2}R}}\phi(| \E u_\ell-\E u|)\, \d x.
\end{align}
We now choose $\kappa$ small enough in
\eqref{july86} and combine the resultant inequality with \eqref{july78}, to obtain:
\begin{align}
    \label{july87}
     \int_{B_{\frac{3}{2}R}}\phi (|\E u_\ell-\E u|)\, \d x &  \leq 
    c \delta \int_{B_{\frac{3}{2}R}}\phi (|\E u|)\, \d x 
   {
     +c\eps\big(|B_{\frac{3}{2}R}|^{1+\frac{{i_\phi'}}{n}}+|B_{\frac{3}{2}R}|^{1+\frac{{s_\phi'}}{n}}\big)
     }
     + 
      c_\eps\bigg( \int_{B_{\frac{3}{2}R}}\widetilde{\phi_{0,n}}(c|f_\varepsilon-f|)\,\dx\bigg)^{1+\frac{ {i_\phi'}}{n}}
       \\ \nonumber & \quad
     +
     c_\eps\bigg( \int_{B_{\frac{3}{2}R}} \widetilde{\phi_{0,n}}(c|f_\ell-f|)\,\dx\bigg)^{1+\frac{{s_\phi'}}{n}} +
c_\kappa\int_{B_{\frac{3}{2}R}}\widetilde \phi (|a_\ell(x,\E u)- a(x,\E u)|)\, \dx.
\end{align}
As
$$a_\ell(x,\E u)\to a(x,\E u)\qquad\quad \text{a.e. in}\,\,  B_{\frac{3}{2}R},$$
and, by \eqref{B2b} and \eqref{ip2b},
\begin{align}
    \label{july79}
    \widetilde \phi(|a_\ell(x,\E u)- a(x,\E u)|)& \leq c  \widetilde\phi(\phi'(|\E u|))   \leq c'
      \phi(|\E u|) \quad \text{a.e. in $B_{\frac{3}{2}R}$,}
\end{align}
one infers from the dominated convergence theorem that
\begin{align}
    \label{july90}
\lim_{\ell\to \infty}\int_{B_{\frac{3}{2}R}}\widetilde \phi (|a_\ell(x,\E u)- a(x,\E u)|)\, \dx =0.
\end{align}
Moreover, thanks to \cite[Lemma 5]{Gossez-studia},
\begin{align}
    \label{july83}
   \lim _{\ell\to \infty}
    \int_{B_{\frac{3}{2}R}}\widetilde{\phi_{0,n}}(c|f_\ell-f|)\,\dx =0.
\end{align}
Equations \eqref{july87}, \eqref{july90}, and \eqref{july83} ensure that
\begin{align}
    \label{july80}
    \limsup_{\ell\to \infty}\int_{B_{\frac{3}{2}R}}\phi (|\E u_\ell-\E u|)\, \d x & \leq 
     c\delta \int_{B_{\frac{3}{2}R}}\phi (|\E u|)\, \d x
   {
     +c\eps\big(|B_{\frac{3}{2}R}|^{1+\frac{ {i_\phi'}}{n}}+|B_{\frac{3}{2}R}|^{1+\frac{ {s_\phi'}}{n}}\big)
     }.
\end{align}
Hence, Equation \eqref{july70} follows, owing to the arbitrariness of $\varepsilon$ and $\delta$.}
\\ From \eqref{july70} and the Korn inequality \eqref{korn-0} we deduce that there exists a subsequence of $\{u_\ell\}$, still denoted by $\{u_\ell\}$, and a function $H\in L^1(B_{\frac{3}{2}R})$
such that
\begin{align}
    \label{aug1}
    \E u_\ell \to \E u \quad \text{and }\quad \nabla u_\ell \to \nabla u \quad \text{a.e. in $B_{\frac{3}{2}R}$,}
\end{align}
and 
\begin{align}
    \label{aug1'}
    \phi(|\nabla u_\ell-\nabla u|) \leq H \quad \text{a.e. in $B_{\frac{3}{2}R}$.}
\end{align}
Thus,
\begin{align}
    \label{aug3}
    \phi(|\nabla u_\ell|) \leq c( \phi(|\nabla u_\ell- \nabla u|) + \phi(|\nabla u|) ) \leq    c( H + \phi(|\nabla u|)),
\end{align}
whence, by dominated convergence,
\begin{align}
    \label{aug4}
    \lim_{\ell\to \infty}\int_{B_{\frac{3}{2}R}}\phi(|\nabla u_\ell|) \, \d x = \int_{B_{\frac{3}{2}R}}\phi(|\nabla u|) \, \d x.
\end{align}
Consequently, from \eqref{est-tau-h-Veps'}, via Fatou's lemma (applied for each fixed $h$) and \eqref{aug4}, one infers that
\begin{align}\label{aug5}
\sup_{0<|h|\leq \frac 34R}\int_{B_{ {\frac 34R}}}\frac{|\tau_hV(\E u)|^2}{|h|^{2\sigma}}\,\dx
     & \le {c_R}
    \bigg(1+
       \bigg(\int_{B_{ \frac 32 R}}{k}^{{m}}\,\dx\bigg)^{{\frac{2m_\sigma}{{m(m_\sigma}-\frac{n}{\beta})}}}  \bigg)  \int_{B_{\frac{3}{2}R}}\phi(|\nabla u|)\,\dx\\\nonumber 
      &\quad+
   c
{\bigg(\int_{B_{\frac32 R}}\psi(|f|)+1\,\dx\bigg)^{1+\frac{({1-\theta}) i_\phi'}{n}}
+c\bigg(\int_{B_{\frac 32R}}\psi(|f|)+1\,\dx\bigg)^{1+\frac{({1-\theta}) s_\phi'}{n}}
       }.
   \end{align}
 The inequality \eqref{oct109} follows by    replacing $R$ with $\frac 43R$ in \eqref{aug5}.  
 \end{proof}

\section{Proof of Theorem \ref{main}, Part (ii),  regularity of $\pi$.}\label{sec:proof3}

We conclude with a proof of Part (ii) of Theorem \ref{main}. As hinted in Section \ref{sec:intro}, the fractional regularity of $\pi$ is derived via the first equation in the system \eqref{equa}, thanks to the information on the fractional regularity of $V(\E u)$ already established in Part (i).

\begin{proof}[Proof of Theorem \ref{main}, Part (ii)] 
   As in the previous proofs, in what follows $c$, $c'$, $c''$ will denote positive constants which  may change from one Equation to another one, and may depend on $n, \nu, L, \phi, \psi$.
  The dependence on additional parameters will be explicitly indicated via an index in the relevant constant.
\\
Let $B_R$ be any ball such that $B_{\frac 32R} \Subset \Omega$.
 Given $\varphi\in L^{\phi}(B_{\frac R2}) $,
 an application   of Bogovski\u\i's Lemma  to the function $\varphi-\varphi_{B_{\frac R2}}$ provides us with a function $w\in W^{1,\phi}_0(B_{\frac R2})$ such that
     \begin{equation}\label{Bogw}
         \begin{cases}
             \mathrm{div}\,w=\varphi-\varphi_{B_{\frac R2}}\cr 
        \displaystyle\int_{B_{\frac R2}}\phi(|\nabla w|)\,\dx\le c \int_{B_{\frac R2}}\phi(|\varphi-\varphi_{B_{\frac R2}}|)\,\dx.
         \end{cases}
     \end{equation}
Let  $\zeta\in
C^\infty_0(\rn)$ be cut-off function such that  $0\leq \zeta \le 1$, with $\zeta =1$ in $B_{\frac R8}$,  ${\rm spt} \zeta \subset B_{\frac R4}$ and
$|\nabla\zeta|\le\frac cR$.
Let $h\in \mathbb R$ be such that $0<|h|<  \min\{\frac R4, 1\}$. The following chain holds:  
\begin{align}\label{eta-pi-dual}
  \int_{B_{{\frac 34R}}}\tau_h(\zeta\pi)\varphi\,\dx
 & =
  \int_{B_{\frac R4}}\zeta\pi\tau_{-h}\varphi\,\dx
  =
  \int_{B_{\frac R4}}\zeta\pi\tau_{-h}\big[\varphi-\varphi_{B_{\frac R2}}\big]\,\dx =
  \int_{B_{\frac R4}}\zeta\pi\dive(\tau_{-h}w)\,\dx\\\nonumber
  &=
  \int_{B_{\frac R4}}\pi\dive(\zeta\tau_{-h}w)\,\dx
  -
  \int_{B_{\frac R4}}\pi\nabla\zeta\cdot\tau_{-h}w\,\dx = \mathrm{I}+\mathrm{II}.
\end{align}
Now, we exploit the fact that $(u,\pi)$ is a weak solution of
system~\eqref{equa}. The function $\zeta\tau_{-h}w\in W^{1,\phi}(\Omega,\R^n)$ and is compactly supported in $B_{\frac R4}$. Using this  test function
in \eqref{weak-equa} enables one to deduce that
\begin{align}\label{pre-I-bound}
  \mathrm{I}
  &=
  -\int_{B_{\frac R4}} f\cdot \zeta \tau_{-h}w\,\dx
  +
  \int_{B_{\frac R4}}  a(x,\E u) \cdot \E(\zeta\tau_{-h}w) \,\dx \\\nonumber
  &=
  -\int_{B_{\frac R4}} f\cdot \zeta \tau_{-h}w\,\dx
  +
  \int_{B_{\frac R2}}
  [ \tau_h(\zeta a(x,\E u))] \cdot \E w  
  \,\dx
  +
  \int_{B_{\frac R4}}a(x,\E u) \cdot (\tau_{-h}w\otimes
  \nabla\zeta)\,\dx\\\nonumber
  &
  \le
  \int_{B_{\frac R4}} \zeta|f|\,|\tau_{-h}w|\,\dx
  +
  \int_{B_{\frac R2}}
  |\tau_h[a(x,\E u)]|\,|\E w| \,\dx
  +
  |h|\|\nabla\zeta\|_{L^\infty}\int_{B_{\frac R2}}
  |a(x,\E u)|\,|\E w| \,\dx
\\\nonumber
  &\qquad
  +
  \|\nabla\zeta\|_{L^\infty}\int_{B_{\frac R4}}|a(x,\E u)|\,|\tau_{-h}w|\,\dx
  = \mathrm{I}_1+\mathrm{I}_2+\mathrm{I}_3+\mathrm{I}_4.
\end{align}
Note that the inequality rests upon Equation \eqref{product-rule}.
\\
To estimate $\mathrm{I}_1$, let us set 
$$\Xi_\lambda = \int_{B_{\frac R2}}\phi\Big(\frac{|\nabla w|}{\lambda}\Big)\dx$$
for $\lambda>0$. From a chain analogous to \eqref{est-RHS-psi} and the Sobolev-Poincar\'e inequality~\eqref{SPnew1} we obtain:
\begin{align}\label{aug10}
  \mathrm{I}_1&=|h|^{\theta}\lambda(|B_{\frac R2}|+\Xi_\lambda)^{({1-\theta})/n}\int_{B_{\frac R4}} \zeta|f|\bigg(\frac{|\tau_{-h}w|}{\lambda|h|}\bigg)^{\theta}\bigg(\frac{|\tau_{-h}w|}{\lambda (|B_{\frac R2}|+\Xi_\lambda)^{1/n}}\bigg)^{1-\theta}\,\dx\\ \nonumber
  &\le   c|h|^{\theta}\lambda(|B_{\frac R2}|+\Xi_\lambda)^{({1-\theta})/n}\int_{B_{\frac R4}}\left[\phi\bigg(\frac{|\tau_{-h}w|}{\lambda|h|}\bigg)+\phi_{0,n}\bigg(\frac{|\tau_{-h}w|}{\lambda(|B_{\frac R2}|+\Xi_\lambda)^{1/n}}\bigg)+\psi(|f|)+1\right]\, \dx\\
  &\le \nonumber
    c|h|^{\theta}\lambda(|B_{\frac R2}|+\Xi_\lambda)^{({1-\theta})/n}\int_{B_{\frac R2}}\left[\phi\big(\tfrac{1}{\lambda}|\nabla w|\big)
    +\phi_{0,n}\bigg(\frac{{  2}|w|}{\lambda(|B_{\frac R2}|+\Xi_\lambda)^{1/n}}\bigg)+\psi(|f|)+1\right]\, \dx\\
 &\le \nonumber
   c|h|^{\theta}\lambda(|B_{\frac R2}|+\Xi_\lambda)^{({1-\theta})/n}\int_{B_{\frac R2}}\left[\phi\big(\tfrac{1}{\lambda}|\nabla w|\big)+\psi(|f|)+1\right]\,\dx\\
    \nonumber &=
   c|h|^{\theta}\lambda(|B_{\frac R2}|+\Xi_\lambda)^{({1-\theta})/n}\bigg(\Xi_\lambda +\int_{B_{\frac R2}}(\psi(|f|)+1)\,\dx\bigg)
\\  \nonumber & \leq 
c|h|^{\theta}\lambda|B_{\frac R2}|^{({1-\theta})/n}\bigg(\Xi_\lambda +\int_{B_{\frac R2}}(\psi(|f|)+1)\,\dx\bigg)
  + c|h|^{\theta}\Xi_\lambda^{({1-\theta})/n}\bigg(\Xi_\lambda +\int_{B_{\frac R2}}(\psi(|f|)+1)\,\dx\bigg).
\end{align}
Fix $\kappa \in (0,1]$. By starting from \eqref{aug10} instead of \eqref{bound-RHS} and arguing as in the proof of \eqref{bound-RHS-2} one can deduce that
\begin{align}
   \mathrm{I}_1
  &\le
     |h|^{\theta}\Bigg[ c\kappa\int_{B_{\frac R2}}\phi(|\nabla w|)\,\dx
   +
   c_\kappa\bigg(\int_{B_{\frac R2}}(\psi(|f|)+1)\,\dx\bigg)^{1+\frac{({1-\theta}) i_\phi'}{n}}
   +
   c_\kappa\bigg(\int_{B_{\frac R2}}(\psi(|f|)+1)\,\dx\bigg)^{1+\frac{({1-\theta}) s_\phi'}{n}}\Bigg].
 \end{align}
 Now, consider $\mathrm{I}_2$. Observe that the exponent $\varrho$ defined in the statement
is such that
$\varrho \leq \alpha$.
 Recalling \eqref{def-A-B}, \eqref{ip2} and using {Lemma \ref{lem:ipholder-beta}}, with $\beta$ replaced with $\varrho$,  tells us that
\begin{align}\label{aug14}
    \mathrm{I}_2&\le  \int_{B_{\frac R2}}
  \mathcal{A}_h\,|\E w| \,\dx+\int_{B_{\frac R2}}
  \mathcal{B}_h\,|\E w| \,\dx
  \le  c |h|^{\varrho} \int_{B_{\frac R2}} \frac{\phi''(|\E(h)|)||\tau_h\E u|}{|h|^{\varrho}}|\E w|\,\dx \\ \nonumber
& \quad +|h|^\varrho\int_{B_{\frac R2}}({k_\varrho(x)+k_\varrho(x+he_i)})\phi'(|\E u|)|\E w|\,\dx =\mathrm{I}_{2,1}+\mathrm{I}_{2,2},
\end{align}
{for some function $k_\varrho\in L^{m_\varrho}(B_R)$ with ${m_\varrho}>\frac{n}{\varrho}$.}
From \eqref{shift4}, Young's inequality, \eqref{homogeneity-phi-star-shifted} and inequality \eqref{bound-phi-star}, one can infer that
\begin{align*}
    \mathrm{I_{2,1}}&\le   c{|h|^{\varrho}}\int_{B_{\frac R2}}\frac{\phi_{|\E u|}'(|\tau_h\E u|)}{|h|^{\varrho}}\left|\E w\right|\,\dx
 \le c_\kappa{|h|^{\varrho}}\int_{B_{\frac R2}}\widetilde{\phi_{|\E u|}}\left(\frac{\phi_{|\E u|}'(|\tau_h\E u|)}{|h|^\varrho}\right)\,\dx+\kappa{|h|^{\varrho}}\int_{B_{\frac R2}}\phi_{|\E u|}\left(|\E w|\right)\,\dx
\\ \nonumber
    &\le c_\kappa \frac{{|h|^{\varrho}}}{|h|^{\varrho\max\{i_\phi',2\}}}\int_{B_{\frac R2}}\phi_{|\E u|}(|\tau_h\E u|)\,\dx+\kappa{|h|^{\varrho}}\int_{B_{\frac R2}}\phi_{|\E u|}(|\E w|)\,\dx.
\end{align*}
 Notice that $\varrho\max\{i_\phi',2\}=2\sigma$. Hence, thanks to 
\eqref{a-coercivity-split} and \eqref{eq:removal-shift2}, 
\begin{align}\label{aug12}
    \mathrm{I_{2,1}}&\le 
   c_\kappa \frac{{|h|^{\varrho}}}{|h|^{2\sigma}} \int_{B_{\frac R2}}|\tau_hV(\E u)|^2\,\dx+c\kappa{|h|^{\varrho}}\int_{B_{\frac R2}}(\phi(|\E u|)+\phi(|\E w|))\,\dx \\ \nonumber
&\le 
 c_\kappa|h|^\varrho    \sup_{0<|h|<\frac R2}
 \int_{B_{\frac R2}}\frac{|\tau_hV(\E u)|^2}{|h|^{2\sigma}}\, \d x+ c\kappa{|h|^{\varrho}}\int_{B_{\frac R2}}\phi(|\E u|)\, \d x 
 +c\kappa|h|^{\varrho}\int_{B_{\frac R2}}\phi(|\E w|)\,\d x.
\end{align}
An application of Young's  and  H\"older's inequalities, \eqref{homogeneity-phi-star},   \eqref{july59}, and \eqref{oct3} yields:
   \begin{align}\label{aug11}
  \mathrm{I}_{2,2}
  &\le
    |h|^\varrho\Bigg[\kappa\int_{B_{\frac R2}}\phi(|\E w|)\dx
    +
    c_\kappa\int_{B_{\frac R2}}\widetilde \phi\Big({(k_\varrho(x)+k_\varrho(x+he_i))}\phi'(|\E u|)\Big)\dx\Bigg]\\
  &\le \nonumber 
    |h|^\varrho\Bigg[\kappa\int_{B_{\frac R2}}\phi(|\E w|)\dx
    +
    c_\kappa\int_{B_{\frac R2}}(1+{k_\varrho(x)+k_\varrho(x+he_i)})^{i_\phi'}\widetilde \phi(\phi'(|\E u|))\dx\Bigg]\\
  &\le \nonumber 
    |h|^\varrho\Bigg[\kappa\int_{B_{\frac R2}}\phi(|\E w|)\dx
    +
    c_\kappa\bigg(\int_{B_{\frac R2}}(1+k_\varrho(x)+k_\varrho(x+he_i))^{{m_\varrho}}\dx\bigg)^{\frac{i_\phi'}{{m_\varrho}}}
\bigg(\int_{B_{\frac R2}}\big[\phi(|\E u|)\big]^{\frac{ {m_\varrho}}{{m_\varrho}-i_\phi'}}\dx\bigg)^{\frac{ {m_\varrho}-i_\phi'}{{m_\varrho}}}\Bigg]
\\  &\le \nonumber 
    |h|^\varrho\Bigg[\kappa\int_{B_{\frac R2}}\phi(|\E w|)\dx
    +
    c_\kappa
\bigg(\int_{B_{\frac R2}}\big[\phi(|\E u|)\big]^{\frac{ {m_\varrho}}{ {m_\varrho}-i_\phi'}}\dx\bigg)^{\frac{{m_\varrho}-i_\phi'}{{m_\varrho}}}
\\ \nonumber & \qquad \qquad + c_\kappa\bigg(\int_{B_{R}}k_\varrho^{{m_\varrho}}\dx\bigg)^{\frac{i_\phi'}{{m_\varrho}}}\bigg(\int_{B_{\frac R2}}\big[\phi(|\E u|)\big]^{\frac{{m_\varrho}}{{m_\varrho}-i_\phi'}}\dx\bigg)^{\frac{{m_\varrho}-i_\phi'}{ {m_\varrho}}}
\Bigg]
\\  &\le \nonumber 
    {|h|^\varrho\Bigg[\kappa\int_{B_{\frac R2}}\phi(|\E w|)\dx
    +
    c_{\kappa,R}
\bigg(\int_{B_{\frac R2}}\big[\phi(|\E u|)\big]^{\frac{ {m_\varrho}}{{m_\varrho}-i_\phi'}}\dx\bigg)^{\frac{{m_\varrho}-i_\phi'}{ {m_\varrho}}}}
\\ \nonumber & {\qquad \qquad + c_{\kappa,R}\bigg(\int_{B_{
{
R}}}{ k}^{ {m}}\dx\bigg)^{\frac{i_\phi'}{{m}}}\bigg(\int_{B_{\frac R2}}\big[\phi(|\E u|)\big]^{\frac{ {m_\varrho}}{ {m_\varrho}-i_\phi'}}\dx\bigg)^{\frac{{m_\varrho}-i_\phi'}{{m_\varrho}}}
\Bigg]}.
\end{align}
Notice that the application of H\"older's inequality is legitimate, as 
${m_\varrho}>\frac{n}{\varrho}\geq i_\phi'$. Also, the last integral is finite, since $\frac{{m_\varrho}}{{m_\varrho}-i_\phi'}<\frac{n}{n-2\sigma}$.
Indeed, by \eqref{besov-emb}, the inclusion $V(\E u)\in B^{\sigma,2, \infty}_{\rm loc }(\Omega)$ implies that $V(\E u) \in L^q_{\rm loc}(\Omega)$ for every $q<\frac{2n}{n-2\sigma}$, and hence $\phi(|\E u|)\in L^{q}_{\rm loc}(\Omega)$ for every $q<\frac{n}{n-2\sigma}$. 
 Equations \eqref{aug14}--\eqref{aug11} imply that 
\begin{align}\label{I2-bound}
 \mathrm{I}_2
 &\le
 c\kappa|h|^{\varrho}\int_{B_{\frac R2}}\phi(|\E w|)\,\dx\\ \nonumber
 &+
 c_\kappa|h|^\varrho \Bigg[
 \sup_{0<|h|<\frac R2}
\int_{B_{\frac R2}}\frac{|\tau_hV(\E u)|^2}{|h|^{2\sigma}}\, \d x
 +\int_{B_{R}}\phi(|\E u|)\,\dx \Bigg]+
c_{\kappa,R}|h|^\varrho \Bigg[\bigg(\int_{B_{\frac R2}}\big[\phi(|\E u|)\big]^{\frac{  {m_\varrho}}{  {m_\varrho}-i_\phi'}}\dx\bigg)^{\frac{  {m_\varrho}-i_\phi'}{  {m_\varrho}}}
\\ \nonumber & \qquad \qquad +  \bigg(\int_{B_{R}}  {k}^{  {m}}\dx\bigg)^{\frac{i_\phi'}{  {m}}}\bigg(\int_{B_{\frac R2}}\big[\phi(|\E u|)\big]^{\frac{  {m_\varrho}}{  {m_\varrho}-i_\phi'}}\dx\bigg)^{\frac{  {m_\varrho}-i_\phi'}{  {m_\varrho}}}
\Bigg].
\end{align}
As far as the term $\mathrm{I}_3$ is concerned, since $\|\nabla\zeta\|_{L^\infty}\le \frac{c}{R}$, by Young's inequality and \eqref{est-diff-quotient}, one has that 
\begin{align}\label{Itre}
    \mathrm{I_3}&\le 
\frac{c|h|}{R}\int_{B_{\frac R2}}
  \phi'(|\E u|)\,|\E w| \,\dx
\le
c|h|\int_{B_{\frac R2}}
 \Big(
 c_\kappa\widetilde \phi\big(\phi'(|\E u|)/R\big) +\kappa\phi(|\E w|)\Big)\,\dx
\\ \nonumber &\le
c_\kappa |h|{\left(\frac{1}{R^{s_\phi'}}+\frac{1}{R^{i_\phi'}}\right)}\int_{B_{\frac R2}}\phi(|\E u|)\,\dx + c\kappa|h|\int_{B_{\frac R2}}
\phi(|\E w|)\,\dx.
\end{align}
The term $\mathrm{I}_4$ can be bounded  similarly to $\mathrm{I}_3$ as follows: 
\begin{align}\label{Iquattro}
    \mathrm{I_4}&\le 
\frac{c|h|}{R}\int_{B_{\frac R4}}
  \phi'(|\E u|)\,\frac{|\tau_{-h} w|}{|h|} \,\dx
\le
c|h|\int_{B_{\frac R4}}
\left[c_\kappa\widetilde\phi\big(\phi'(|\E u|/R)\big) + \kappa\phi\left(\frac{|\tau_{-h} w|}{|h|}\right)\right]\,\dx
\\ \nonumber
&\le
c_\kappa |h|{\left(\frac{1}{R^{s_\phi'}}+\frac{1}{R^{i_\phi'}}\right)}\int_{B_{\frac R2}}
 \phi(|\E u|) 
+   c\kappa|h|\int_{B_{\frac R2}}\phi\left(|{\nabla} w|\right) 
\,\dx.
\end{align}
 \\Exploiting the 
 estimates for $\mathrm{I}_i$, $i=1, \dots 4$, established above  to bound the rightmost side of~\eqref{pre-I-bound} yields:
\begin{align}\label{aug20}
  \mathrm{I}
 &\le c\kappa|h|^\varrho\int_{B_{\frac R2}}\phi(|\nabla w|)|)\,\dx
 \\ \nonumber &+
  c_\kappa|h|^\varrho\Bigg[\sup_{0<|h|<\frac R2}
\int_{B_{\frac R2}}\frac{|\tau_hV(\E u)|^2}{|h|^{2\sigma}}\, \d x +\bigg(\int_{B_{\frac R2}}(\psi(|f|)+1)\,\dx\bigg)^{1+\frac{({1-\theta}) i_\phi'}{n}}
  +
\bigg(\int_{B_{\frac R2}}(\psi(|f|)+1)\,\dx\bigg)^{1+\frac{({1-\theta}) s_\phi'}{n}}\Bigg]
\\ \nonumber
& +  c_{\kappa,R}|h|^\varrho\Bigg[
 \int_{B_{R}}\phi(|\E u|)\,\dx +
\bigg(\int_{B_{\frac R2}}\big[\phi(|\E u|)\big]^{\frac{{m_\varrho}}{{m_\varrho}-i_\phi'}}\dx\bigg)^{\frac{{m_\varrho}-i_\phi'}{{m_\varrho}}}
\\ \nonumber & \qquad \qquad \qquad \qquad \qquad \qquad +  \bigg(\int_{B_{R}}{k}^{{m}}\dx\bigg)^{\frac{i_\phi'}{  {m}}}\bigg(\int_{B_{\frac R2}}\big[\phi(|\E u|)\big]^{\frac{  {m_\varrho}}{  {m_\varrho}-i_\phi'}}\dx\bigg)^{\frac{  {m_\varrho}-i_\phi'}{  {m_\varrho}}}\Bigg].
\end{align}
 Here, we have also made use of the fact $\varrho\le \theta\le 1$.
Owing to Young's inequality,  the  inequalities \eqref{est-diff-quotient}   and \eqref{homogeneity-phi-star}, and  the fact that $|h|\leq 1$, the term $\mathrm{II}$ admits the following bound:
\begin{align}\label{aug21}
  \mathrm{II}
  &\le
  c|h|\int_{B_{\frac R4}}\frac{|\tau_{-h}w|}{|h|}\frac{|\pi|}{R}\,\dx \le
  |h|\bigg(
\kappa\int_{B_{\frac R4}}\phi\bigg(\frac{|\tau_{-h}w|}{|h|}\bigg)\,\dx
    +
    c_\kappa\int_{B_{\frac R4}}\widetilde \phi\bigg(\frac{|\pi|}{R}\bigg)\,\dx
    \bigg)\\
 \nonumber  &\le
    |h|^{\varrho} \bigg(
    \kappa\int_{B_{\frac R4}}\phi(|\nabla w|)\,\dx
    +
    {c_\kappa\left(\frac{1}{R^{s_\phi'}}+\frac{1}{R^{i_\phi'}}\right)}
    \int_{B_{\frac R4}}\widetilde \phi(|\pi|)\,\dx
    \bigg).
\end{align}
Now, combine Equations \eqref{eta-pi-dual}, \eqref{aug20}, \eqref{aug21}, and make use of 
the inequality ~\eqref{Bogw} to estimate the integrals involving $w$ to deduce that 
\begin{align}\label{pressure-1}
  \bigg|\int_{B_{{\frac 34R}}}&\frac{\tau_h(\zeta\pi)}{|h|^\varrho}\varphi\,\dx\bigg|
\le
   c\kappa\int_{B_{\frac R2}}\phi(|\varphi|)\,\dx 
   \\ \nonumber & + 
   c_\kappa \Bigg[\sup_{0<|h|<\frac R2}
\int_{B_{\frac R2}}\frac{|\tau_hV(\E u)|^2}{|h|^{2\sigma}}\, \d x +\bigg(\int_{B_{\frac R2}}(\psi(|f|)+1)\,\dx\bigg)^{1+\frac{({1-\theta}) i_\phi'}{n}}
 +
\bigg(\int_{B_{\frac R2}}(\psi(|f|)+1)\,\dx\bigg)^{1+\frac{({1-\theta}) s_\phi'}{n}}\Bigg]
\\ \nonumber
& +  c_{\kappa,R} \Bigg[
 \int_{B_{R}}\phi(|\E u|)\,\dx +
\bigg(\int_{B_{\frac R2}}\big[\phi(|\E u|)\big]^{\frac{{m_\varrho}}{{m_\varrho}-i_\phi'}}\dx\bigg)^{\frac{{m_\varrho}-i_\phi'}{{m_\varrho}}}
\\ \nonumber & \qquad \qquad \qquad \qquad \qquad \qquad +  \bigg(\int_{B_{R}}{k}^{{m}}\dx\bigg)^{\frac{i_\phi'}{  {m}}}\bigg(\int_{B_{\frac R2}}\big[\phi(|\E u|)\big]^{\frac{  {m_\varrho}}{  {m_\varrho}-i_\phi'}}\dx\bigg)^{\frac{  {m_\varrho}-i_\phi'}{  {m_\varrho}}} +\int_{B_{\frac R2}}\widetilde \phi(|\pi|)\,\dx\Bigg]
\end{align}
for every $\varphi\in L^{\phi}(B_{\frac R2})$.
 Choose
\begin{equation}\label{aug23}
  \varphi=\frac{\xi}{|\xi|^2}\widetilde \phi\big(|\xi|\big),
  \quad\mbox{where } \,\,\,\xi=\frac{\tau_h(\zeta\pi)}{|h|^\varrho}.
\end{equation}
We claim that $\varphi\in L^\phi(B_{\frac R2})$, and hence 
 $\varphi$
is an admissible test function in~\eqref{pressure-1}.
Indeed,
\begin{equation*}
  \phi(|\varphi|)
  =
\phi\bigg(\frac{\widetilde \phi\big(|\xi|\big)}{|\xi|}
  \bigg)
  \le \widetilde 
 \phi\big(|\xi|\big),
\end{equation*}
where the inequality holds thanks to \eqref{complementary-N-functions}, after exchanging the roles of $\phi$ and $\widetilde \phi$. Since Equation \eqref{pressure-1} ensures that
$\xi\in
L^{\widetilde \phi}(B_{\frac R2})$, our claim follows.  
By our choice of $\varphi$, the inequality \eqref{pressure-1}
implies that
\begin{align}\label{oct10'}
   \int_{B_{\frac 34R}}&\widetilde \phi\bigg(\frac{|\tau_h(\zeta\pi)|}{|h|^\varrho}\bigg)\dx
\le
   c\kappa\int_{B_{\frac R2}}\widetilde \phi\bigg(\frac{|\tau_h(\zeta\pi)|}{|h|^\varrho}\bigg)\dx
   \\ \nonumber & + 
   c_\kappa \Bigg[\sup_{0<|h|<\frac R2}
\int_{B_{\frac R2}}\frac{|\tau_hV(\E u)|^2}{|h|^{2\sigma}}\, \d x +\bigg(\int_{B_{\frac R2}}(\psi(|f|)+1)\,\dx\bigg)^{1+\frac{({1-\theta}) i_\phi'}{n}}
 +
\bigg(\int_{B_{\frac R2}}(\psi(|f|)+1)\,\dx\bigg)^{1+\frac{({1-\theta}) s_\phi'}{n}}\Bigg]
\\ \nonumber
& +  c_{\kappa,R} \Bigg[
 \int_{B_{R}}\phi(|\E u|)\,\dx +
\bigg(\int_{B_{\frac R2}}\big[\phi(|\E u|)\big]^{\frac{{m_\varrho}}{{m_\varrho}-i_\phi'}}\dx\bigg)^{\frac{{m_\varrho}-i_\phi'}{{m_\varrho}}}
\\ \nonumber & \qquad \qquad \qquad \qquad \qquad \qquad +  \bigg(\int_{B_{R}}{k}^{{m}}\dx\bigg)^{\frac{i_\phi'}{  {m}}}\bigg(\int_{B_{\frac R2}}\big[\phi(|\E u|)\big]^{\frac{  {m_\varrho}}{  {m_\varrho}-i_\phi'}}\dx\bigg)^{\frac{  {m_\varrho}-i_\phi'}{  {m_\varrho}}} +\int_{B_{\frac R2}}\widetilde \phi(|\pi|)\,\dx\Bigg].
\end{align}
Observe that $\tau_h(\zeta\pi)(x)=\tau_h\pi(x)$ if $x\in B_{\frac {R}{16}}$  and
$|h|\le \frac{R}{16}$. 
Hence, by setting $h_1=\min\{1, \frac R{16}\}$ and $\kappa=\frac1{2c}$, and using the estimate \eqref{aug5}
one obtains:
\begin{align}\label{aug25}
    \sup_{|h|\leq h_1}&\int_{B_{{\frac{R}{16}}}%\frac 34R}
    }\widetilde \phi\bigg(\frac{|\tau_h\pi|}{|h|^\varrho}\bigg)\dx 
  \le
    \sup_{|h|\leq h_1}\int_{B_{\frac 34R}}\widetilde \phi\bigg(\frac{|\tau_h(\zeta\pi)|}{|h|^\varrho}\bigg)\dx
   \\ \nonumber &
\le
  {c_R}
    \bigg(1+
       \bigg(\int_{B_{ {{\frac 32 R}}}}{k}^{{m}}(x)\,\dx\bigg)^{{\frac{2m_\sigma}{{m(m_\sigma}-\frac{n}{\beta})}}}  \bigg)  \int_{B_{\frac{3}{2}R}}\phi(|\nabla u|)\,\dx
       \\ \nonumber &
       \quad + c_R \bigg(1+  \bigg(\int_{B_R}{k}^{{m}}(x)\,\dx\bigg)^{\frac{i_\phi'}{{m}}}\bigg)\bigg(\int_{B_{\frac R2}}\big[\phi(|\E u|)\big]^{\frac{  {m_\varrho}}{  {m_\varrho}-i_\phi'}}\dx\bigg)^{\frac{  {m_\varrho}-i_\phi'}{  {m_\varrho}}}  
       \\\nonumber 
      &\quad+ c_R\int_{B_{\frac R2}}\widetilde \phi(|\pi|)\,\dx
   +c
{\bigg(\int_{B_{\frac 32R}}(\psi(|f|)+1)\,\dx\bigg)^{1+\frac{({1-\theta}) i_\phi'}{n}}
+c\bigg(\int_{B_{\frac 32R}}(\psi(|f|)+1)\,\dx\bigg)^{1+\frac{({1-\theta}) s_\phi'}{n}}
       }.
\end{align}
 Consider a covering of $B_{\frac 34R}$ by balls of radius $\frac R{32}$,  apply the inequality \eqref{aug25} with $R$ replaced with $\frac R2$ to each of these balls and sum the resultant inequalities. So doing, one obtains an inequality like \eqref{aug25}, with $\frac R{16}$ replaced with $\frac 34 R$ in the integral on the leftmost side, and $h_1= \min\{1, \frac R{32}\}$.
Finally, if $h_1\leq |h|\leq \frac34 R$, then, by \eqref{homogeneity-phi-star} and \eqref{est-translation},
\begin{align}
    \label{oct14}
   \int_{B_{\frac 34R}}\widetilde \phi\bigg(\frac{|\tau_h\pi|}{|h|^\varrho}\bigg)\dx
   \leq 
   c\bigg(\frac 1{R^{\varrho i_\phi'}}+1\bigg)\int_{B_{\frac 34R}}\widetilde \phi\big(|\tau_h\pi|\big)\dx \leq c\bigg(\frac 1{R^{\varrho i_\phi'}}+1\bigg)\int_{B_{\frac 32R}}\widetilde \phi(|\pi|)\dx. 
\end{align}
 Thus,  the supremum in \eqref{aug25} can be extended to every $h$ such that $|h|\le \frac 34R$, provided that the integral of $\widetilde\phi(|\pi|)$ on the rightmost side is extended over $B_{\frac 32R}$. A replacement of $R$ with $\frac 43R$
in the resultant inequality shows that
\begin{align}
    \label{boundpi'}
    \sup_{|h|<R}    \int_{B_{R}}\widetilde \phi\bigg(\frac{|\tau_h\pi|}{|h|^\varrho}\bigg)\, \d x&\le
 c_1 
\bigg(\int_{B_{2R}}\psi(|f|)\,\dx\bigg)^\varkappa  +
c_2 \Bigg[\bigg(\int_{B_{2R}}\phi(|\nabla u|)^\varsigma\,\dx\bigg)^{\frac 1\varsigma} 
 + \int_{B_{2R}}\widetilde\phi(|\pi|)\,\dx\Bigg]
   +c_3
\end{align}
for some constants $c_1$, $c_2$ and $c_3$ depending on the same quantities as in Equation \eqref{boundpi}. Denote by $\Lambda$ the right-hand side of \eqref{boundpi'}. Owing to the property \eqref{{kt}}, we deduce from \eqref{boundpi} that
\begin{align}
        \label{boundpi''}
    \sup_{|h|<R}    \int_{B_{R}}\widetilde \phi\bigg(\frac{|\tau_h\pi|}{(\Lambda +1)|h|^\varrho}\bigg)\, \d x
    \leq \frac{1}{\Lambda +1} \sup_{|h|<R}    \int_{B_{R}}\widetilde \phi\bigg(\frac{|\tau_h\pi|}{|h|^\varrho}\bigg)\, \d x \leq 1.
\end{align}
Hence, the inequality 
\eqref{boundpi} follows from \eqref{boundpi'}, with $c_3$ replaced with $c_3+1$. The inclusion  \eqref{reg-pi} is just a consequence of \eqref{boundpi}.
\end{proof}

 \par\noindent {\bf Data availability statement.} Data sharing not applicable to this article as no datasets were generated or analysed during the current study.

\section*{Compliance with Ethical Standards}\label{conflicts}

\smallskip
\par\noindent
{\bf Funding}. This research was partly funded by:
\\ (i) GNAMPA   of the Italian INdAM - National Institute of High Mathematics (grant number not available)  (A.Cianchi);
\\ (ii) Research Project   of the Italian Ministry of Education, University and
Research (MIUR) Prin 2022 ``Partial differential equations and related geometric-functional inequalities'',
grant number 20229M52AS, cofunded by PNRR (A.Cianchi);
\\ (iii)  GNAMPA Project 2025, grant number E5324001950001, ``Regolarità di soluzioni di equazioni paraboliche a crescita nonstandard degeneri'' (F.Giannetti, A.Passarelli di Napoli);
\\(iv)   Centro Nazionale per la Mobilità Sostenibile (CN00000023) - Spoke 10 Logistica Merci, grant number E63C22000930007,  funded by PNRR (A.Passarelli di Napoli).

\bigskip
\par\noindent
{\bf Conflict of Interest}. The authors declare that they have no conflict of interest.

\end{document}